\documentclass[twoside,11pt]{amsart}

\usepackage[foot]{amsaddr}

\author{Yuliy Baryshnikov}
\address[Baryshnikov]{Departments of Mathematics and Electrical and Computer Engineering, University of Illinois, Urbana, IL 61801, USA}

\author{Matthew D. Kvalheim}
\address[Kvalheim]{Department of Mathematics, University of Michigan, Ann Arbor, MI 48109}

\email{ymb@illinois.edu, kvalheim@umich.edu}

\title[Flux in tilted potential systems: negative resistance and persistence]{Flux in tilted potential systems: negative resistance and persistence}

\usepackage{import} 

\usepackage[utf8]{inputenc}
\usepackage[T1]{fontenc}
\usepackage{lmodern}
\usepackage{amsmath}
\usepackage{amssymb}
\usepackage{amsthm}
\usepackage{stmaryrd} 
\usepackage{mathtools}
\usepackage{tikz-cd}
\usepackage{subfig}

\usepackage{thmtools}
\usepackage{thm-restate} 
\usepackage{enumerate}

\newcommand{\concept}[1]{\textbf{#1}}

\newcommand{\ip}[2]{\langle #1, #2 \rangle}
\newcommand{\dist}[2]{\textnormal{dist}( #1, #2 )}
\newcommand{\distT}[3]{\textnormal{dist}_{#3}( #1, #2 )}

\def\comp{\Gamma_m}
\def\bv{\mathbf{v}}
\def\tM{\tilde{M}}

\newcommand{\compd}{\overrightarrow{\Gamma}_m}
\newcommand{\Ed}{\overrightarrow{E}_m}
\newcommand{\Eu}{E_m}

\newcommand{\N}{\mathbb{N}}
\newcommand{\Z}{\mathbb{Z}}
\newcommand{\R}{\mathbb{R}}

\newcommand{\Ws}{W^s}
\newcommand{\Wu}{W^u}

\newcommand{\slot}{\,\cdot\,} 

\newcommand{\id}{\textnormal{id}}
\newcommand{\interior}{\textnormal{int}}
\newcommand{\cl}{\textnormal{cl}}
\newcommand{\image}{\textnormal{im}}
\newcommand{\cfo}{\alpha}
\newcommand{\cft}{\beta}
\newcommand{\dom}{\textnormal{dom}}
\newcommand{\Hdr}{H_{\textnormal{dR}}}
\newcommand{\Hom}{H}
\newcommand{\cC}{C}
\newcommand{\Hess}{\textnormal{Hess}}
\newcommand{\CRST}{\mathsf{CRST}}
\newcommand{\RST}{\mathsf{RST}}
\newcommand{\ST}{\mathsf{ST}}
\newcommand{\cycle}{\mathsf{cycle}}
\newcommand{\E}{\mathbb{E}}
\newcommand{\Prob}{\mathbb{P}}
\newcommand{\tProbe}{\tilde{\mathbb{P}}^\varepsilon}
\newcommand{\dV}{dx}
\newcommand{\vtx}{V}
\newcommand{\tvtx}{\tilde{\vtx}}
\newcommand{\tv}{\tilde{v}}
\newcommand{\tx}{\tilde{x}}
\newcommand{\Aut}{\textnormal{Aut}}
\newcommand{\edg}{E}
\newcommand{\gph}{\Gamma}
\newcommand{\gh}{\gph_{\Pi}}
\newcommand{\Eh}{\edg_\Pi}

\newcommand{\cm}{J}
\newcommand{\cme}{\cm_{\varepsilon}}

\newcommand{\dm}{\rho}
\newcommand{\dme}{\dm_{\varepsilon}}

\newcommand{\vf}{\mathfrak{X}}
\newcommand{\flux}{\mathcal{F}}
\newcommand{\fluxe}{\flux_\varepsilon}
\newcommand{\fluxec}{\flux_{\varepsilon,c}}
\newcommand{\bw}{\mathbf{w}}
\newcommand{\bu}{\mathbf{u}}
\newcommand{\dft}{\bv}
\newcommand{\dfte}{\dft_\varepsilon}
\newcommand{\tdft}{\tilde{\dft}}

\newcommand{\af}{\mathcal{S}}
\newcommand{\ind}{\textnormal{ind}}
\newcommand{\Cont}{C}
\newcommand{\vp}{\varphi}
\newcommand{\length}{\textnormal{length}}

\newcommand{\qp}{Q}
\newcommand{\tqp}{\tilde{\qp}}
\newcommand{\qpd}{\qp_{\dft}}
\newcommand{\qptd}{\qp_{\tdft}}

\newcommand{\qptdz}{\qp_{\tdft, z}}

\newcommand{\tqpd}{\tilde{\qp}_{\dft}}
\newcommand{\tqptd}{\tilde{\qp}_{\tdft}}
\newcommand{\height}{\textnormal{height}}
\newcommand{\heightc}{\height_c}
\newcommand{\heightz}{\height_0}
\newcommand{\trace}{\textnormal{trace}}

\newcommand{\src}{\mathfrak{s}}
\newcommand{\tgt}{\mathfrak{t}}
\newcommand{\st}{\src\tgt}
\newcommand{\tor}{\mathbb{T}}
\newcommand{\sph}{\mathbb{S}}

\DeclarePairedDelimiter\norm{\lVert}{\rVert}

\newcommand{\thistheoremname}{}
\newtheorem*{genericthm}{\thistheoremname}
{\renewcommand{\thistheoremname}{Theorem~\ref{#1}$'$}%
	\begin{genericthm}}
	{\end{genericthm}}

\theoremstyle{definition}
\newtheorem{Th}{Theorem}[section]
\newtheorem{Def}[Th]{Definition}
\newtheorem*{Def*}{Definition}
\newtheorem{Lem}[Th]{Lemma}
\newtheorem{Co}[Th]{Corollary}
\newtheorem{Prop}[Th]{Proposition}
\newtheorem{Ex}[Th]{Example}
\newtheorem{Rem}[Th]{Remark}
\newtheorem{Assump}[Th]{Assumption}
\newtheorem{Quest}[Th]{Question}

\usepackage{marvosym}
\usepackage[
hypertexnames,%
citecolor=blue,%
colorlinks=true,%
linkcolor=red%
]{hyperref}
\usepackage[all]{hypcap}

\begin{document}
	
	\maketitle
	\begin{abstract}	
	Many real-world systems are well-modeled by Brownian particles subject to gradient dynamics plus noise arising, e.g., from the thermal fluctuations of a heat bath.
	Of central importance to many applications in physics and biology (e.g., molecular motors) is the net steady-state particle current or ``flux'' enabled by the noise and an additional driving force.
	However, this flux cannot usually be calculated analytically.
	Motivated by this, we investigate the steady-state flux generated by a nondegenerate diffusion process on a general compact manifold; such fluxes are essentially equivalent to the stochastic intersection numbers of Manabe (1982).
	In the case that noise is small and the drift is ``gradient-like'' in an appropriate sense, we derive a graph-theoretic formula for the small-noise asymptotics of the flux using Freidlin-Wentzell theory.
	When additionally the drift is a local gradient sufficiently close to a generic global gradient, there is a natural flux equivalent to the entropy production rate---in this case our graph-theoretic formula becomes Morse-theoretic, and the result admits a description in terms of persistent homology.
	As an application, we provide a mathematically rigorous explanation of the paradoxical ``negative resistance'' phenomenon in Brownian transport discovered by Cecchi and Magnasco (1996).
	\end{abstract}	
	
	\tableofcontents
\section{Introduction}\label{sec:intro}
To quote \cite{reimann2001giant}: \begin{quote}
Thermal diffusion in a tilted periodic potential plays a prominent role in Josephson junctions, rotating dipoles in external fields, superionic conductors, charge density waves, synchronization phenomena, diffusion on crystal surfaces, particle separation by electrophoresis, and biophysical processes such as intracellular transport, to name just a few.
\end{quote}
Relevant references may be found in \cite{reimann2001giant, risken1996fokker,reimann2002brownian}.
It is useful to describe many such systems by a ``microscopic'' model of the form ($\dot{x} = dx/dt$)
\begin{equation}\label{eq:sde}
\dot{x} = \dft(x) + \sqrt{2 \varepsilon} \xi(t), \qquad x(t), \xi(t) \in \R^n, \quad \varepsilon > 0,
\end{equation}
where $\dft$ is a vector field on $\R^n$, $\xi$ is a Gaussian white noise process, and \eqref{eq:sde} is interpreted rigorously as a stochastic differential equation (SDE) \cite{ikeda1981stochastic,gardiner2004handbook, oksendal2013stochastic}.
Diffusion in a ``tilted periodic potential'' refers to the case that, after a coordinate rescaling,
\begin{equation}\label{eq:tilt}
\dft(x) = -\nabla U(x) + F = -\nabla \underbrace{(U(x)-F\cdot x)}_{\tilde{U}(x)}, \qquad F\in \R^n, \quad \forall k\in \Z^n\colon U(x+k)\equiv U(x)
\end{equation}
for a constant driving force $F$.
The potential $U$ is spatially periodic, but the effective potential $\tilde{U}$ is not; the landscape $\text{graph}(\tilde{U})\coloneqq \{(x,\tilde{U}(x))\}\subset \R^{n+1}$ is the result of ``tilting'' $\text{graph}(U)$ in the direction of $F$.
The position $x(t)$ of a particle subject to \eqref{eq:sde} and \eqref{eq:tilt} can be imagined as the projection of a particle performing stochastic gradient descent on the landscape $\text{graph}(\tilde{U})$.

Since $\dft$ in \eqref{eq:tilt} is spatially periodic, it is natural to view \eqref{eq:sde} as an SDE on the cube $[0,1]^n$ with opposite boundary faces identified, the flat $n$-torus $\tor^n$, so that the components of $x$ are defined modulo $1$. 
In this way we are led to consider SDEs on a torus; alternatively, for many systems the components of $x$ may be angular variables, so that the state space is naturally a torus.\footnote{It is amusing to imagine $x(t)$ as the projected position of a particle performing stochastic gradient descent on an ``impossible landscape'' over $\tor^n$ in the sense of \cite{penrose1958impossible,penrose1986escher,penrose1992cohomology}. The landscape is ``impossible'' since $\dft$ is not globally the gradient of any function on $\tor^n$ if $F\neq 0$, but it is \emph{locally} a gradient in a neighborhood of every point.
}	
(We consider more general state spaces later.)

For a variety of applications it is useful to imagine many particles moving independently according to \eqref{eq:sde}.
In the continuum limit the normalized macroscopic particle density $\rho$ satisfies, in steady-state, the stationary Fokker-Planck equation
\begin{equation}\label{eq:fp-1}
0 = \nabla \cdot \underbrace{(\rho \dft - \varepsilon \nabla \rho)}_{J} =  \nabla \cdot J.
\end{equation} 
Alternatively, the steady-state probability density of a particle moving according to \eqref{eq:sde} satisfies \eqref{eq:fp-1}, so we refer to $J$ as the \concept{steady-state probability current}.
The net steady-state current or \concept{flux} in, say, the direction $(1,0,\ldots, 0)$ is defined by the standard flux integral of $J$ through any hypersurface $\{x^1 = a\}$, 
\begin{equation}\label{eq:flux-int-1}
\begin{split}
\flux &\coloneqq \int_0^1\cdots \int_0^1 J^1(a,x^2,\ldots, x^n) dx^2\cdots dx^n,
\end{split}
\end{equation} 
where $x = (x^1,\ldots, x^n)$ and $J = (J^1,\ldots, J^n)$.
Since $\nabla \cdot J = 0$, the divergence theorem implies that the right side of \eqref{eq:flux-int-1} is indeed independent of $a\in [0,1]$, so it coincides with its average over $a\in [0,1]$:
\begin{equation}\label{eq:flux-int-2}
\begin{split}
\flux & = \int_{[0,1]^n}J^1(x) dx = \int_{[0,1]^n}\dft^1(x)\rho(x) dx
\overset{\textnormal{a.s.}}{=} \lim_{t\to\infty} \frac{x^1(t)}{t},
\end{split}
\end{equation} 
where $\dft = (\dft^1,\ldots, \dft^n)$, ``a.s.'' means ``almost surely'' (with probability one), the preceding text explains the first equality, and the remaining two equalities are explained now.
The definition $J \coloneqq \rho \dft - \varepsilon \nabla \rho$ implies that $J^1 = \rho \dft^1 - \varepsilon \partial \rho/ \partial x^1$, and periodicity of $\rho$ implies that the integral of the second term vanishes, yielding the second equality in \eqref{eq:flux-int-2}.
A result of Manabe involving ergodicity implies the remaining (a.s.) equality \cite[Thm~4.1]{manabe1982stochastic}.
The right side of \eqref{eq:flux-int-2} is a ``microscopic'' quantity, while the other quantities in \eqref{eq:flux-int-1}, \eqref{eq:flux-int-2} are ``macroscopic''.
It is quite interesting that, for a typical (Morse) periodic potential $U$ and sufficiently small forcing $F$, the steady-state flux is nonzero in some direction if and only if \emph{both} $F$ and the noise intensity are nonzero (Prop.~\ref{prop:flux-positive-c1f-case}). 
Hence flux can be ``harvested'' from the noise if there is a biasing force, in the sense that flux vanishes if the noise vanishes.

In the same way that knowing the voltage-current characteristics of an electrical conductor is important, for applications it seems important to understand how the driving force $F$ affects flux.
However, to compute the flux via, e.g., \eqref{eq:flux-int-1} one needs to obtain $J$ by solving \eqref{eq:fp-1} for $\rho$, and this cannot be done analytically except for the $1$-dimensional case or in special situations like $F = 0$ (in which case $J\equiv 0$ \cite[p.~279, Thm~4.6]{ikeda1981stochastic}). 
In spite of this technical difficulty, several surprising properties of the flux have been demonstrated.
For example, when the driving force $F = (c,0,\ldots,0)$ for $c\in \R$, one might expect the flux $\flux(c)$ to increase monotonically as a function of $c$.
Indeed, Cecchi and Magnasco showed that this is true in the $1$-dimensional case $n=1$ \cite{cecchi1996negative}; however, they presented an example with $n = 2$ in which $\flux(c)$ is numerically demonstrated to decrease as $c$ is increased within a certain range.
In other words, they found a ``Brownian conductor'' with \concept{negative resistance} (or ``conductance'', or ``mobility'').\footnote{In the example in \cite{cecchi1996negative} the dynamics are periodic only in one spatial direction, so the natural reduced state space is a cylinder rather than a torus; we will not dwell on this technical detail.}

The goal of the present paper is to rigorously approximate flux for a broad class of stochastic systems with sufficient accuracy to enable rigorous predictions of negative resistance in Brownian conductors.  
In \S\ref{sec:nr-torus} we achieve this in a concrete example of the type \eqref{eq:sde}, \eqref{eq:tilt} with $n = 2$.
To make the problem tractable we restrict attention to the applications-relevant case of small noise, and we content ourselves with seeking small-noise asymptotics of the flux in the sense of large deviations \cite{freidlin2012random, varadhan2016large}.
That is, we seek $\psi (c) >0$ so that
\begin{equation}\label{eq:psi-intro-1}
\lim_{\varepsilon\to 0}(-\varepsilon \ln \flux(c)) = \psi(c),
\end{equation}
where for now we assume that $\flux(c)> 0$ for all $\varepsilon > 0$.
Note that $c$ is fixed while $\varepsilon \to 0$, with asymptotics in $c$ beyond the scope of this paper.
Given $c_1 < c_2$ with $\psi(c_1)< \psi(c_2)$,  \eqref{eq:psi-intro-1} implies that $\flux(c_2) < \flux(c_1)$ for all $\varepsilon$ sufficiently small: there is negative resistance.
Rather than restrict our attention to SDEs on the torus, we will accomplish this goal in the more general setting of diffusion processes on compact manifolds.
 
\subsection{General setting of the paper}\label{sec:general-setting}
In this paper smooth always means $\Cont^\infty$.
Let $M$ be a closed (i.e., compact and boundaryless) connected smooth $n$-dimensional manifold ($1\leq n < \infty$).
We henceforth switch notation from $x(t)$ to $X_t$.
For every $\varepsilon > 0$ let $(X_t^\varepsilon, \Prob^\varepsilon_x)$ be a diffusion process on $M$ (see \cite{mckean2005stochastic,ikeda1981stochastic,hsu2002stochastic}), and suppose that in any system of smooth local coordinates $(x^i)$ its infinitesimal generator $L_\varepsilon$ can be written in the form\footnote{There are two reasons for allowing $b^i_\varepsilon$, $\dfte$ to depend on $\varepsilon$. First, $\varepsilon$-dependence is cheap: removing it creates no simplifications whatsoever in any of our proofs. The second reason is more fundamental: $\varepsilon$-dependence of $b^i_\varepsilon$ depends on the choice of local coordinates (cf. \cite[pp.~135--136]{freidlin2012random}); alternatively, $\dfte$ in \eqref{eq:generator-laplacian} will typically still be $\varepsilon$-dependent if $b^i_\varepsilon\equiv b^i$ is $\varepsilon$-independent, and vice versa. 
Thus, generality is added for free and awkwardness is avoided by allowing $\varepsilon$-dependence of both $b^i_\varepsilon$ and $\dfte$.\label{footnote:epsilon-dependent-drift}}
\begin{equation}\label{eq:generator}
L_\varepsilon = \sum_i b^{i}_{\varepsilon}(x)\frac{\partial}{\partial x^i} + \varepsilon \sum_{i,j}a^{ij}(x)\frac{\partial^2}{\partial x^i \partial x^j},
\end{equation}
where $x\mapsto (a^{ij}(x))$ takes values in the symmetric and strictly positive-definite matrices.
For simplicity we assume that the coefficients $a^{ij}$ are smooth and that the coefficients $b^{i}_{\varepsilon}$ are smooth for all $\varepsilon > 0$.
A direct computation shows that the inverse matrices $(a^{-1})_{ij}$ are the coordinate representations of a smooth Riemannian metric $G$ on $M$.
Denoting by $\Delta\coloneqq \nabla \cdot \nabla \coloneqq \textnormal{div}_G \textnormal{ grad}_G$ the associated Laplace-Beltrami operator, a direct coordinate computation shows that $L_\varepsilon$ takes the form
\begin{equation}\label{eq:generator-laplacian}
L_\varepsilon = \dfte + \varepsilon \Delta
\end{equation}
for a well-defined family of smooth vector fields $\dfte$ (cf. \cite[pp.~273--274]{ikeda1981stochastic}).\footnote{The divergence operator $\nabla \cdot$ and hence also the Laplace-Beltrami operator $\Delta = \nabla \cdot \nabla$ are well-defined even on a nonorientable Riemannian manifold $(M,G)$ \cite[Ex.~16.31]{lee2013smooth}. We also remark that, when $\dfte \equiv 0$ and $\varepsilon = \frac{1}{2}$, the diffusion associated to $L_\varepsilon$ is the Brownian motion associated to the metric $G$ \cite[p.~271, Def.~4.2]{ikeda1981stochastic}.}
Here $\dfte$ is viewed as a differential operator via identification with its Lie derivative in the standard way.
We assume that $\dfte \to \dft$ uniformly as $\varepsilon \to 0$, where $\dft$ is a $\Cont^1$ vector field on $M$.\footnote{Note that the solution of the SDE \eqref{eq:sde} with spatially periodic smooth coefficients defines an example of such a diffusion process on $\tor^n$ with $L_\varepsilon = \dft + \varepsilon \Delta$, where $\Delta$ is the Laplace-Beltrami operator on the flat $n$-torus \cite[Sec.~4.3]{mckean2005stochastic}. \label{footnote:diffusion-as-SDE}}
For each $\varepsilon > 0$ the diffusion has a  $\Cont^\infty$ probability density $\dme$ on $M$ (with respect to the Riemannian density $\dV$ of $G$ \cite[p.~432]{lee2013smooth}) uniquely solving the stationary Fokker-Planck equation
\begin{equation}\label{eq:fp-cont}
0 = \nabla \cdot (\dme \dfte) - \varepsilon \Delta \dme  =  \nabla \cdot \underbrace{(\dme \dfte - \varepsilon \nabla \dme)}_{\cme} = \nabla \cdot \cme,
\end{equation}
where $\cme$ is the \concept{steady-state probability current}.

Let $\cfo$ be a closed one-form on $M$ with de Rham cohomology class $[\cfo]\in \Hdr^1(M)$ \cite[Ch.~17]{lee2013smooth}.
In \S \ref{sec:flux} we argue that the following equalities hold, are independent of the closed form representing $[\cfo]$, and constitute the correct way to generalize \eqref{eq:flux-int-1}, \eqref{eq:flux-int-2} to define the steady-state \concept{$[\cfo]$-flux} $\fluxe([\cfo])$:
\begin{equation}\label{eq:flux-gen-intro}
\fluxe([\cfo])\coloneqq \int_M \cfo(\cme)\, \dV \overset{\textnormal{a.s.}}{=} \lim_{t\to\infty} \frac{1}{t} \int_{X^\varepsilon_{[0,t]}} \cfo,
\end{equation}
where the quantity on the right involves the line integral of $\cfo$ along the sample path $X^\varepsilon_{[0,t]}$ of the diffusion up to time $t$ \cite[Sec.~VI.6]{ikeda1981stochastic}.
This implies that the flux is a linear map $\fluxe\colon \Hdr^1(M)\to \R$.
When $[\cfo]$ is Poincar\'{e} dual to a transversely oriented closed hypersurface $N$, \eqref{eq:flux-gen-intro} coincides with the flux integral of $\cme$ over $N$, as explained in \S\ref{sec:flux}.
Note that the flux $\flux$ in \eqref{eq:flux-int-1}, \eqref{eq:flux-int-2} is the same thing as $\fluxe([d\theta^1])$, where $\theta^1$ is the first circle-valued coordinate on $\tor^n$ viewed as $[0,1]^n$ with opposite boundary faces identified.

\subsection{Contributions and organization of the paper}\label{sec:contributions}
\subsubsection{Summary of the first main result}\label{sec:summary-first-main-result}
One of our main contributions is to rigorously prove that negative resistance occurs for a broad class of diffusion processes.
We accomplish this using a special case (Theorem~\ref{th:qualititative-special-case-intro}) of our first main result (Theorem~\ref{th:qualitative}), which we now describe.
Recall that $\dfte \to \dft$ uniformly as $\varepsilon \to 0$ and assume that
\begin{equation}\label{eq:dft-decomp-intro}
\dft = - \nabla U + c\cft^\sharp, \qquad c > 0,
\end{equation}
where $U\in \Cont^2(M) = \Cont^2(M,\R)$ and $\cft^\sharp$ is the metric dual via $G$ of a $\Cont^1$ closed one-form  $\cft$.\footnote{This generalizes the situation of \eqref{eq:tilt} with $F = (c,0,\ldots, 0)$, with $\cft$ corresponding to the one-form $d\theta^1$ on $\tor^n$.}
Vector fields of this type are studied in the Morse-Novikov theory \cite{novikov1982hamiltonian,farber2004topology,pajitnov2006circle}.
One of our other results (Prop.~\ref{prop:flux-positive-c1f-case}) suggests that the contribution of the term $c\cft^\sharp$ in \eqref{eq:dft-decomp-intro} should be interpreted as ``pushing'' the flux of Theorem~\ref{th:qualititative-special-case-intro} in the ``direction'' of $\cft$, analogous to $F$ in \eqref{eq:tilt} and \eqref{eq:flux-int-1}, \eqref{eq:flux-int-2}.
A generic function $U\in \Cont^2(M)$ satisfies the following \cite[Rem.~6.11]{banyaga2006lectures}.\footnote{A critical point of $U$ is a zero of $\nabla U$.
A function $U$ is Morse if its Hessian at every critical point is a nondegenerate bilinear form.
The (Morse) \concept{index} of a critical point $p$ of a Morse function $U$ is the number of negative eigenvalues of the Hessian of $U$ at $p$.}
\begin{restatable}[]{Assump}{AssumpMorseSmaleIntro}\label{assump:morse-smale-intro} $U$ has a unique global minimizer, $U$ takes distinct values on distinct index-$1$ critical points, and $U$ is Morse-Smale ($U$ is a Morse function and all pairwise intersections of (un)stable manifolds of zeros of $\nabla U$ are transverse).
\end{restatable}

	\begin{figure}
		\centering
		\def\svgwidth{0.7\linewidth}
\begingroup%
  \makeatletter%
  \providecommand\color[2][]{%
    \errmessage{(Inkscape) Color is used for the text in Inkscape, but the package 'color.sty' is not loaded}%
    \renewcommand\color[2][]{}%
  }%
  \providecommand\transparent[1]{%
    \errmessage{(Inkscape) Transparency is used (non-zero) for the text in Inkscape, but the package 'transparent.sty' is not loaded}%
    \renewcommand\transparent[1]{}%
  }%
  \providecommand\rotatebox[2]{#2}%
  \newcommand*\fsize{\dimexpr\f@size pt\relax}%
  \newcommand*\lineheight[1]{\fontsize{\fsize}{#1\fsize}\selectfont}%
  \ifx\svgwidth\undefined%
    \setlength{\unitlength}{1372.13241511bp}%
    \ifx\svgscale\undefined%
      \relax%
    \else%
      \setlength{\unitlength}{\unitlength * \real{\svgscale}}%
    \fi%
  \else%
    \setlength{\unitlength}{\svgwidth}%
  \fi%
  \global\let\svgwidth\undefined%
  \global\let\svgscale\undefined%
  \makeatother%
  \begin{picture}(1,0.47561052)%
    \lineheight{1}%
    \setlength\tabcolsep{0pt}%
    \put(0.26137981,0.46730484){\makebox(0,0)[lt]{\lineheight{1.25}\smash{\begin{tabular}[t]{l}graph of $f$, where $df = p^*(-\dft^\flat)$\end{tabular}}}}%
    \put(0,0){\includegraphics[width=\unitlength,page=1]{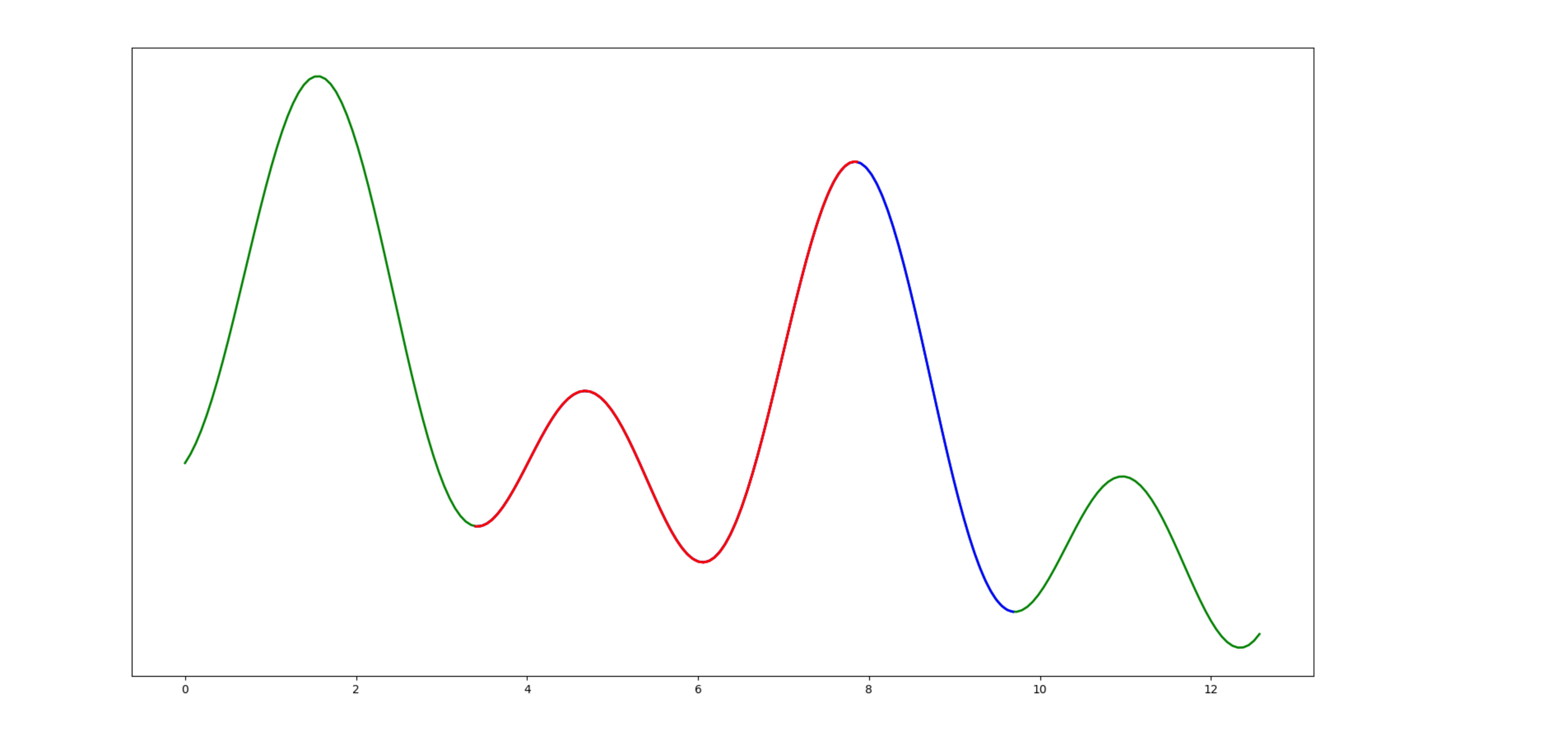}}%
    \put(-0.00090743,0.2421078){\makebox(0,0)[lt]{\lineheight{1.25}\smash{\begin{tabular}[t]{l}$f(x)$\end{tabular}}}}%
    \put(0.43199537,0.00160669){\makebox(0,0)[lt]{\lineheight{1.25}\smash{\begin{tabular}[t]{l}$x$\end{tabular}}}}%
    \put(0,0){\includegraphics[width=\unitlength,page=2]{height-illustration.pdf}}%
    \put(0.88824598,0.22883362){\color[rgb]{1,0,0}\makebox(0,0)[lt]{\lineheight{1.25}\smash{\begin{tabular}[t]{l}$\heightc(\gamma)$\end{tabular}}}}%
    \put(0.49922625,0.01714585){\color[rgb]{0,0,1}\makebox(0,0)[lt]{\lineheight{1.25}\smash{\begin{tabular}[t]{l}$\tilde{\gamma}$\end{tabular}}}}%
    \put(0,0){\includegraphics[width=\unitlength,page=3]{height-illustration.pdf}}%
  \end{picture}%
\endgroup%

		\caption{An illustration of the definition \eqref{eq:height-loop} of the height of a loop $\gamma$. Here $M$ is the circle $\sph^1$ viewed as the interval $[0,2\pi]$ with ends identified, $p\colon \R\to \sph^1$ is the universal cover $p(x)=x \mod 2 \pi$, and $\tilde{\gamma}$ (blue horizontal curve) is a lift of $\gamma$ to $\tM$. Imagining $\tilde{\gamma}$ as lifted so as to travel along $\textnormal{graph}(f)$, only the red portion contributes to $\heightc(\gamma)$.}\label{fig:heights}
	\end{figure}

Under this assumption, the implicit function theorem yields a unique $\Cont^1$ curve $c\mapsto v_*(c)$ of $c$-dependent asymptotically stable zeros of $\dft$, defined for $c$ in some nondegenerate interval $c\in [0,c_0)$, such that $v_*(0)$ is the global minimizer of $U$.
Given $p\in M$, it is expedient to define the \concept{loop space} $\Omega_p M$ to be the set of continuous paths $\gamma\colon [0,1]\to M$ satisfying $\gamma(0)=\gamma(1)=p$.
Consider the closed one-form $\dft^\flat = -dU + c\cft$ and define, for any $\gamma\in \Omega_p M$, the \concept{height} of the loop $\gamma\in \Omega_p M$ via\footnote{Recall that \emph{closed} one-forms can be integrated over merely continuous paths \cite[p.~163]{farber2004topology}.}
\begin{equation}\label{eq:height-loop}
\heightc(\gamma)\coloneqq \sup_{t\in [0,1]} \int_{\gamma|_{[0,t]}}(-\dft^\flat).
\end{equation}
See Fig.~\ref{fig:heights}.
Using \eqref{eq:height-loop}, for any $c\geq 0$ we define
\begin{equation}\label{eq:h-star-def}
h_*(c)\coloneqq \inf \{\heightc(\gamma)\colon \gamma\in \Omega_{v_*(c)} M \textnormal{ and } \int_\gamma \cft > 0\}.
\end{equation}
\begin{Rem}[The height is positive: $h_*(c) > 0$]\label{rem:h-star-greater-than-zero}
Since $\dft^\flat$ is closed, $v_*(c)$ has a homologically trivial open neighborhood $W$ on which $\dft|_W = -\nabla f$ for some $f\in \Cont^2(W)$. 
Since $v_*(c)$ is asymptotically stable for $\dft$, $f$ attains a local minimum at $v_*(c)$.
This, \eqref{eq:height-loop}, and $-\dft^\flat|_W = df$ imply the existence of $k > 0$ such that $\heightc(\gamma) \geq k$ for any loop $\gamma\in \Omega_{v_*(c)}$ whose image is not contained in $W$.
The image of any loop $\gamma\in \Omega_{v_*(c)}$ satisfying $\int_\gamma \cft > 0$ is not contained in the homologically trivial set $W$, so the preceding sentence and \eqref{eq:h-star-def} imply that $h_*(c)\geq k > 0$.
\end{Rem}
As we show in \S \ref{sec:th-2-implies-th-1}, the following is a special case of Theorem~\ref{th:qualitative}; it and Cor.~\ref{co:qualitative-intro-nr} are illustrated in an example in \S \ref{sec:nr-torus}.
Here we use the notation $\fluxec$ rather than $\fluxe$ to emphasize the dependence on $c$.
\begin{restatable}[]{Th}{ThmQualitativeSpecial}\label{th:qualititative-special-case-intro} Consider for each $\varepsilon > 0$ the diffusion process with generator $\dfte + \varepsilon \Delta$ on the closed connected Riemannian manifold $M$, where each $\dfte$ is a smooth vector field and $\dfte\to \dft \coloneqq -\nabla U + c\cft^\sharp$ uniformly as $\varepsilon \to 0$.
Assume that $U\in \Cont^2(M)$ satisfies Assumption~\ref{assump:morse-smale-intro} and that the $\Cont^1$ one-form $\cft$ is closed but not exact.
Then for sufficiently small $c > 0$, the steady-state $[\cft]$-flux of the diffusion with generator $\dfte + \varepsilon \Delta$  satisfies $\fluxec([\cft])>0$ for all sufficiently small $\varepsilon > 0$, and
\begin{equation}\label{eq:qualititative-special-case intro}
\lim_{\varepsilon\to 0}(-\varepsilon \ln \fluxec([\cft]))=h_*(c).
\end{equation}
\end{restatable}
\begin{Co}\label{co:qualitative-intro-nr}
Assume the hypotheses of Theorem~\ref{th:qualititative-special-case-intro}, and assume that the map $c\mapsto h_*(c)$ is strictly increasing on some nonempty interval of the form $(0,c_0)$.
Then for all sufficiently small $c_2 > c_1 > 0$ and all sufficiently small $\varepsilon > 0$, $\flux_{\varepsilon,c_2}([\cft])< \flux_{\varepsilon,c_1}([\cft])$: there is \concept{negative resistance}.
\end{Co}
In the course of proving  Theorems~\ref{th:qualititative-special-case-intro} and \ref{th:qualitative}, we also obtain a result (Theorem~\ref{th:qualitative-measure}) on the small-noise asymptotics of the invariant measure of the diffusion.
	\begin{figure}
		\centering
		\def\svgwidth{1.0\linewidth}
\begingroup%
  \makeatletter%
  \providecommand\color[2][]{%
    \errmessage{(Inkscape) Color is used for the text in Inkscape, but the package 'color.sty' is not loaded}%
    \renewcommand\color[2][]{}%
  }%
  \providecommand\transparent[1]{%
    \errmessage{(Inkscape) Transparency is used (non-zero) for the text in Inkscape, but the package 'transparent.sty' is not loaded}%
    \renewcommand\transparent[1]{}%
  }%
  \providecommand\rotatebox[2]{#2}%
  \newcommand*\fsize{\dimexpr\f@size pt\relax}%
  \newcommand*\lineheight[1]{\fontsize{\fsize}{#1\fsize}\selectfont}%
  \ifx\svgwidth\undefined%
    \setlength{\unitlength}{1301.01454727bp}%
    \ifx\svgscale\undefined%
      \relax%
    \else%
      \setlength{\unitlength}{\unitlength * \real{\svgscale}}%
    \fi%
  \else%
    \setlength{\unitlength}{\svgwidth}%
  \fi%
  \global\let\svgwidth\undefined%
  \global\let\svgscale\undefined%
  \makeatother%
  \begin{picture}(1,0.44388436)%
    \lineheight{1}%
    \setlength\tabcolsep{0pt}%
    \put(0,0){\includegraphics[width=\unitlength,page=1]{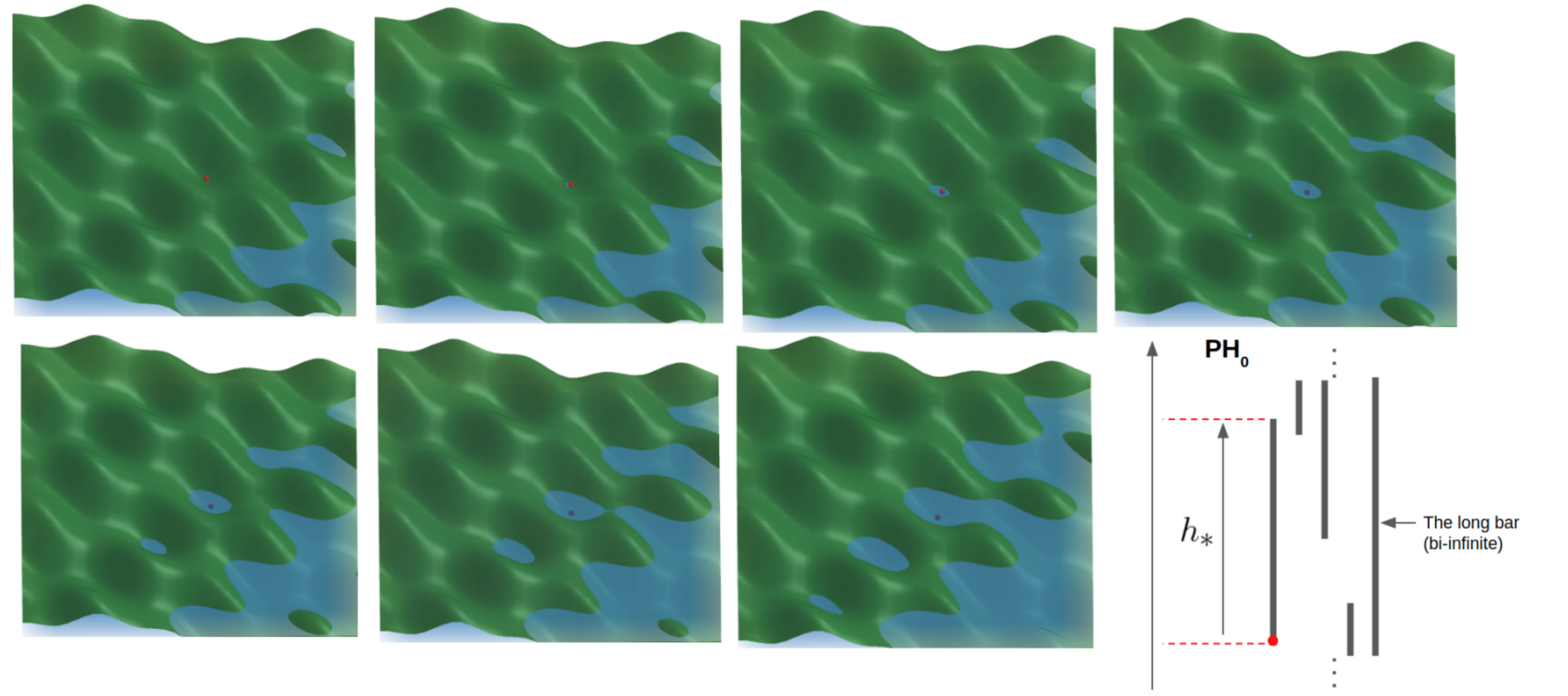}}%
    \put(0.71400327,0.20283642){\makebox(0,0)[lt]{\lineheight{1.25}\smash{\begin{tabular}[t]{l}$a$\end{tabular}}}}%
    \put(0,0){\includegraphics[width=\unitlength,page=2]{persistent-homology-remark.pdf}}%
    \put(0.74167395,0.09791822){\makebox(0,0)[lt]{\lineheight{1.25}\smash{\begin{tabular}[t]{l}$h_*(c)$\end{tabular}}}}%
    \put(0.78087426,0.21206037){\makebox(0,0)[lt]{\lineheight{1.25}\smash{\begin{tabular}[t]{l}$P\Hom_0$\end{tabular}}}}%
    \put(0.91115708,0.10483614){\makebox(0,0)[lt]{\lineheight{1.25}\smash{\begin{tabular}[t]{l}exceptional \\bar\end{tabular}}}}%
    \put(0,0){\includegraphics[width=\unitlength,page=3]{persistent-homology-remark.pdf}}%
  \end{picture}%
\endgroup%

		\caption{An illustration of Rem.~\ref{rem:persistence} with $M = \tor^2$ and $\tM = \R^2$.
		The green landscape represents $\textnormal{graph}(f)$, the blue ocean is depicted at various sea levels $a$, and the red point is the lift to $\textnormal{graph}(f)$ of a point in $p^{-1}(v_*(c))$.
		For the specific example shown, the nonexceptional bar corresponding to the red point experiences only one merge: into the exceptional bar.}\label{fig:persistent-homology}
	\end{figure}

\begin{Rem}
Although the definition \eqref{eq:h-star-def} of $h_*(c)$ in Theorem~\ref{th:qualititative-special-case-intro} involves infinitely many loops $\gamma$, $h_*(c)$ can be computed from finitely many numbers associated to the Morse graph ($1$-skeleton) of $\dft$; see \S \ref{sec:main-results-morse} (Lem.~\ref{lem:h-star-defs-coincide}).
\end{Rem}

\begin{Rem}[Persistent homology interpretation; see Fig.~\ref{fig:persistent-homology}]\label{rem:persistence}
One interpretation of $h_*(c)$ is via the zeroth persistent homology ($P\Hom_0$) ``barcode'' \cite{ghrist2008barcodes} or the closely related merge tree \cite[pp.~149--150]{edelsbrunner2010computational}.
Let $p\colon \tilde{M}\to M$ be any smooth covering such that the pullback $p^*\cft$ is exact, so that there is $f\in \Cont^2(\tilde{M})$ with $p^*(\dft^\flat) = -df$.
Consider the bi-infinite filtration defined by the increasing family of sublevel sets $\{f< a\}$ with $a\in \R$.
Since $f$ is unbounded on the noncompact $\tilde{M}$, the $P\Hom_0$ barcode consists of infinitely many ``bars'' each representing a connected component of some $\{f<a\}$ and, with the exception of precisely one bar, indexed by the local minimum of $f$ at which the corresponding component is ``born''.
The exceptional bar is the only bar of infinite extent, and it is bi-infinitely so; imagining $\{f<a\}$ as the projected portion of the landscape $\textnormal{graph}(f)$ covered by an ocean with sea level $a$, the exceptional bar corresponds to the only noncompact component of the ocean. 
When two components merge as a result of increasing $a$, the corresponding bar born at the larger $a$ value terminates, and one imagines that this ``younger'' bar merges into the ``older'' bar.
In this way, each nonexceptional bar experiences a finite sequence of merges until it merges into the exceptional bar at some finite relative height $[a_{\text{(exceptional merge)}}- a_{\text{(nonexceptional birth)}}]$.
All bars beginning at a point in $p^{-1}(v_*(c))$ experience their exceptional merge at the same relative height: $h_*(c)$.
\end{Rem}

\begin{Rem}[Formulation in the language of partial differential equations]\label{rem:pde-pov}
As mentioned, for $\varepsilon > 0$ there is a unique positive solution $\dme$ to the ``advection-diffusion'' PDE with smooth coefficients
\begin{equation}\label{eq:pde-pov}
0 = \nabla \cdot (u\dfte) - \varepsilon \Delta u = \nabla \cdot \underbrace{(u \dfte - \varepsilon \nabla u)}_{J(u)}
\end{equation}
satisfying $\int_M \dme \dV = 1$ \cite[Thm~3]{zeeman1988stability}.
(All solutions to \eqref{eq:pde-pov} are smooth \cite[p.~139, Thm~4.8]{wells1980differential}.)  
For \emph{any} positive solution $u_\varepsilon$ to \eqref{eq:pde-pov}, linearity and uniqueness imply that $u_\varepsilon = \left(\int_M u_\varepsilon \dV\right)\dme$.
It follows that $J(u_\varepsilon)\coloneqq u_\varepsilon \dfte - \varepsilon \nabla u_\varepsilon$ satisfies $J(u_\varepsilon) = \left(\int_M u_\varepsilon \dV\right) \cme $ where $\cme = J(\dme)$, so Theorem~\ref{th:qualititative-special-case-intro} can be reformulated as a result on singularly perturbed PDE: given a family of positive solutions $(u_\varepsilon)_{\varepsilon > 0}$ to \eqref{eq:pde-pov} such that $\lim_{\varepsilon \to 0} \varepsilon \ln\left(\int_M u_\varepsilon \dV\right) = 0$ and with $U$, $\dft$, $\dfte$, $c$, and $\cft$ satisfying the hypotheses of Theorem~\ref{th:qualititative-special-case-intro}, 
$$\lim_{\varepsilon\to 0}\left(-\varepsilon \ln \int_M \cft(J(u_\varepsilon))\dV\right) = h_*(c).$$
Similar remarks can be made for Theorems~\ref{th:qualitative}, \ref{th:qualitative-measure}, \ref{th:flux-manifold-mc-CRST-ld} and Prop.~\ref{prop:one-minimizer} (cf., e.g.,  \cite[p.~178, Thm~5.4]{freidlin2012random}).
\end{Rem}

\subsubsection{Brief overview of the second main result}\label{sec:brief-overview-second-main-result}
Our second main contribution (Theorem~\ref{th:flux-manifold-mc-CRST-ld}) is to compute the small-noise large deviations of the flux when $\dft$ belongs to the more general class of vector fields which are ``gradient-like'' in the sense that their chain recurrent set consists of a finite number of hyperbolic zeros (see Def.~\ref{def:chain-recurrent} or \cite[pp.~36--37]{conley1978isolated}). 
In particular this means that every $\dft$-integral curve converges to a zero of $\dft$ in both forward and backward time, but this condition also precludes the existence of heteroclinic cycles.
To prove Theorem~\ref{th:flux-manifold-mc-CRST-ld} we approximate the diffusion by a countable Markov chain, a natural idea \cite[Sec.~6]{chernyak2009noneq}, but one which requires estimates to carry out rigorously---we do this using Freidlin-Wentzell theory \cite{freidlin2012random}.
In the directed graph underlying this Markov chain there are a finite number of vertices corresponding to the index-$0$ zeros of $\dft$, but an infinite number of edges corresponding to path homotopy classes between vertices, and these edges are decorated by the corresponding line integrals of $-\cfo$.
This (or at least some homological data) is needed for the Markov chain to contain enough information to determine the large deviations of the flux, and it motivates our introduction of a path-homotopical refinement of the standard Freidlin-Wentzell quasipotential.
We use Theorem~\ref{th:flux-manifold-mc-CRST-ld} as a tool in our proof of Theorem~\ref{th:qualitative}.

\subsubsection{Outline of the sequel}\label{sec:outline-sequel}
The remainder of the paper is organized as follows.
After discussing related work, in \S \ref{sec:nr-torus} we illustrate Theorem~\ref{th:qualititative-special-case-intro} by using it to rigorously demonstrate the occurrence of negative resistance in a concrete example.
In \S \ref{sec:flux} we motivate the definition of flux, relate it to prior literature, and establish some of its basic properties. 
In \S \ref{sec:main-results-morse} we introduce some Morse-theoretic concepts before stating our main result (Theorem~\ref{th:qualitative}) on small-noise flux asymptotics when the limiting drift is dual to a closed one-form close to a generic exact one-form.
Also in \S \ref{sec:main-results-morse} is a result (Theorem~\ref{th:qualitative-measure}) concerning the small-noise asymptotics of the invariant measure under the same assumptions, a strengthening of Theorem~\ref{th:qualitative} in a certain special case (Prop.~\ref{prop:one-minimizer}), a proof of Theorem~\ref{th:qualititative-special-case-intro}, and a proof that the definitions of $h_*$ given in \eqref{eq:h-star-def} and \eqref{eq:h-star-def-morse} coincide. 
Our main result for the more general situation that the chain recurrent set of the drift vector field consists of a finite number of hyperbolic zeros (Theorem~\ref{th:flux-manifold-mc-CRST-ld}) is stated in \S \ref{sec:drift-finite-hyperbolic-chain} after first introducing the Freidlin-Wentzell action functional and a path-homotopical refinement of the Freidlin-Wentzell quasipotential.
The proof of Theorem~\ref{th:flux-manifold-mc-CRST-ld} is carried out in \S \ref{sec:proof-th-flux-manifold-mc-CRST-ld}, and the proofs of Theorems~\ref{th:qualitative} and \ref{th:qualitative-measure} and Prop.~\ref{prop:one-minimizer} are carried out in \S \ref{sec:proofs-morse}.
(Theorem~\ref{th:flux-manifold-mc-CRST-ld} plays a crucial role in our proof of  Theorem~\ref{th:qualitative}, but the latter theorem does not seem to trivially follow from the former.) 
In \S \ref{sec:main-results-morse} there arise natural questions as to whether certain hypotheses in Theorems~\ref{th:qualitative} and \ref{th:qualitative-measure} can be removed (Rem.~\ref{rem:th-qualitative-sufficiently-close}), and \S \ref{sec:counterexamples} contains several counterexamples which yield negative answers to these questions.
In \S \ref{sec:conclusion} we summarize our  contributions, speculate on their potential impact, and discuss prospects for future work.
Finally, the proofs of certain intermediate results are deferred to an appendix (App.~\ref{app:proofs}) in an attempt to improve the flow of the paper.

\subsection{Related work}\label{sec:related-work}
Understanding properties of nonequilibrium steady states---those for which the steady-state current $\cme$ of \eqref{eq:fp-cont} is not identically zero---is one of the most important problems of statistical physics and thermodynamics \cite{degroot1962non,kubo1985statistical,gallavotti1995dynamical,ruelle1999smoothnoneq,dorfman1999intro, zwanzig2001nonequilibrium, ottinger2005beyond,ge2012landscapes,zhang2012stochastic}.
The flux of \eqref{eq:flux-gen-intro} and \S \ref{sec:flux} is one such property which, in the case that $M$ is $1$-dimensional, has received significant attention in the physics literature \cite{magnasco1993forced,risken1996fokker,reimann2001giant,reimann2002brownian,reimann2002diffusion,hanggi2009artificial,seifert2012stochastic,faggionato2012representation,cheng2015long,proesmans2019large}, usually under names such as ``mean velocity'' and ``particle current'' and with notation such as $\langle \dot{x} \rangle$.
Flux in the case that $M$ is $2$-dimensional (typically a torus or cylinder obtained from symmetry reduction of a system on $\R^2$) has also been studied in the physics literature; in addition to the review \cite[Sec.~5.9]{reimann2002brownian}, we also mention \cite{qian1998vector,kostur2000numerical} and the discovery of negative resistance in \cite{cecchi1996negative} which played a significant role in motivating the present work.
We mention that the negative resistance (or conductance, or mobility) referred to here is ``differential'', in contradistinction to the ``absolute'' negative resistance arising, e.g., from time-inhomogeneous noise \cite{eichhorn2002brownian} or driving forces periodic in space and time \cite{collet2008asymptotic,joubad2015langevin}. 

Of the physics-oriented literature of which we are aware, the topological viewpoint of flux in \cite{chernyak2009noneq} (referred to as ``stochastic currents'' or ``topological currents'' therein)  most closely matches our own viewpoint; the reader desiring a discussion (with picture) supplementing \S\ref{sec:flux} is referred to \cite[Sec.~3]{chernyak2009noneq}. 
Though only discrete state spaces are considered, another work informing our viewpoint on flux is \cite{chernyak2013algebraic}; flux is also studied for discrete state spaces in \cite{ren2011duality,wachtel2015fluctuating,altaner2015fluctuating}.

In the mathematical literature, Manabe introduced and studied the basic properties of asymptotic stochastic intersection (or rotation) numbers and proved that these are essentially equivalent to our definition of flux; see Rem.~\ref{rem:flux-relationship-with-jiang-manabe} and see \cite[Sec.~5.4]{jiang2004mathematical} for a textbook reference.
Using the ``intersection number'' point of view, we mention that there are higher-dimensional generalizations of flux as asymptotic stochastic intersection numbers of higher-dimensional objects evolving under the diffusion \cite{catanzaro2016stochastici,catanzaro2016stochasticii}.
Our definition of flux in \eqref{eq:flux-def} mirrors Schwartzman's analogous notion of asymptotic cycles for deterministic systems \cite{schwartzman1957asymptotic}; see Rem.~\ref{rem:flux-relationship-with-asymptotic-cycles} for more details and references.

To the best of our knowledge, the  small-noise asymptotics of the steady-state flux has not been studied except possibly in the case that $M$ is $1$-dimensional (see Ex.~\ref{ex:tilt-pot-circle}), although the large-time asymptotics of ``transient'' flux has been considered in several of the references just mentioned.
To study these small-noise asymptotics, we make use of Freidlin-Wentzell theory.
The most relevant references on the subject for our purposes are \cite{ventsel1970small, freidlin2012random}, and especially \cite[Ch.~6]{freidlin2012random}; some introductions to Freidlin-Wentzell theory we have found useful are \cite{varadhan1984large,varadhan2001prob,berglund2013kramers,touchette2018intro}.
Another introduction to Freidlin-Wentzell theory is \cite[Ch.~6]{bovier2015meta}; we mention that the potential-theoretic and Witten Laplacian \cite{witten1982supersymmetry} approaches respectively taken in \cite{bovier2015meta} and  \cite{lepeutrec2013precise} are two alternative approaches to metastability which have produced asymptotics for certain quantities that are sharper than those of the Freidlin-Wentzell approach.
As in \cite[Ch.~6]{freidlin2012random}, rooted spanning trees in certain directed graphs play an important role in our approach, but so do the cycle-rooted spanning trees studied in \cite{pitman2018tree}.

Especially relevant for the results in \S \ref{sec:main-results-morse} and their proofs are the basic ideas of Morse theory \cite{milnor1963morse,milnor1965hcob,nicolaescu2011invitation}, such as an understanding of the $1$-skeleton of the Morse-Thom-Smale-Witten complex \cite[Ch.~6]{pajitnov2006circle}, and some basic ideas of Novikov's generalization of Morse theory to circle-valued potentials and more generally closed one-forms \cite{novikov1982hamiltonian,farber2004topology,pajitnov2006circle}.
Under the more general assumptions of \S \ref{sec:drift-finite-hyperbolic-chain}, the role of Morse-Novikov theory is replaced by Conley theory \cite{conley1978isolated}. 
Due to their presence in the very definition of flux, closed one-forms play a central role in our approach to studying stochastic dynamics in this paper.
We conclude this section by mentioning that closed one-forms have also played roles in studying deterministic dynamical systems: in addition to the roles they play in Schwartzman's theory of asymptotic cycles and in Morse-Novikov theory mentioned above, ``Lyapunov one-forms'' have been used in \cite{farber2004lyap,farber2004smooth} to formulate a refinement of Conley's ``fundamental theorem of dynamical systems'' \cite{norton1995fundamental,robinson1999dynamical,kvalheim2021conley} and in \cite{byrnes2007differential,byrnes2010topological,kvalheim2021families} to prove existence of periodic orbits.

\section{Illustrating Theorem~\ref{th:qualititative-special-case-intro}: a negative resistance example}\label{sec:nr-torus}
In this section we use Cor.~\ref{co:qualitative-intro-nr} to rigorously establish the existence of negative resistance in a concrete example.
(This example is revisited in Ex.~\ref{ex:nr-example-2} to illustrate Theorem~\ref{th:qualitative} and Prop.~\ref{prop:one-minimizer}.)
	\begin{figure}
		\centering
		\def\svgwidth{1.0\linewidth}
\begingroup%
  \makeatletter%
  \providecommand\color[2][]{%
    \errmessage{(Inkscape) Color is used for the text in Inkscape, but the package 'color.sty' is not loaded}%
    \renewcommand\color[2][]{}%
  }%
  \providecommand\transparent[1]{%
    \errmessage{(Inkscape) Transparency is used (non-zero) for the text in Inkscape, but the package 'transparent.sty' is not loaded}%
    \renewcommand\transparent[1]{}%
  }%
  \providecommand\rotatebox[2]{#2}%
  \newcommand*\fsize{\dimexpr\f@size pt\relax}%
  \newcommand*\lineheight[1]{\fontsize{\fsize}{#1\fsize}\selectfont}%
  \ifx\svgwidth\undefined%
    \setlength{\unitlength}{1166.01745605bp}%
    \ifx\svgscale\undefined%
      \relax%
    \else%
      \setlength{\unitlength}{\unitlength * \real{\svgscale}}%
    \fi%
  \else%
    \setlength{\unitlength}{\svgwidth}%
  \fi%
  \global\let\svgwidth\undefined%
  \global\let\svgscale\undefined%
  \makeatother%
  \begin{picture}(1,0.48781433)%
    \lineheight{1}%
    \setlength\tabcolsep{0pt}%
    \put(0,0){\includegraphics[width=\unitlength,page=1]{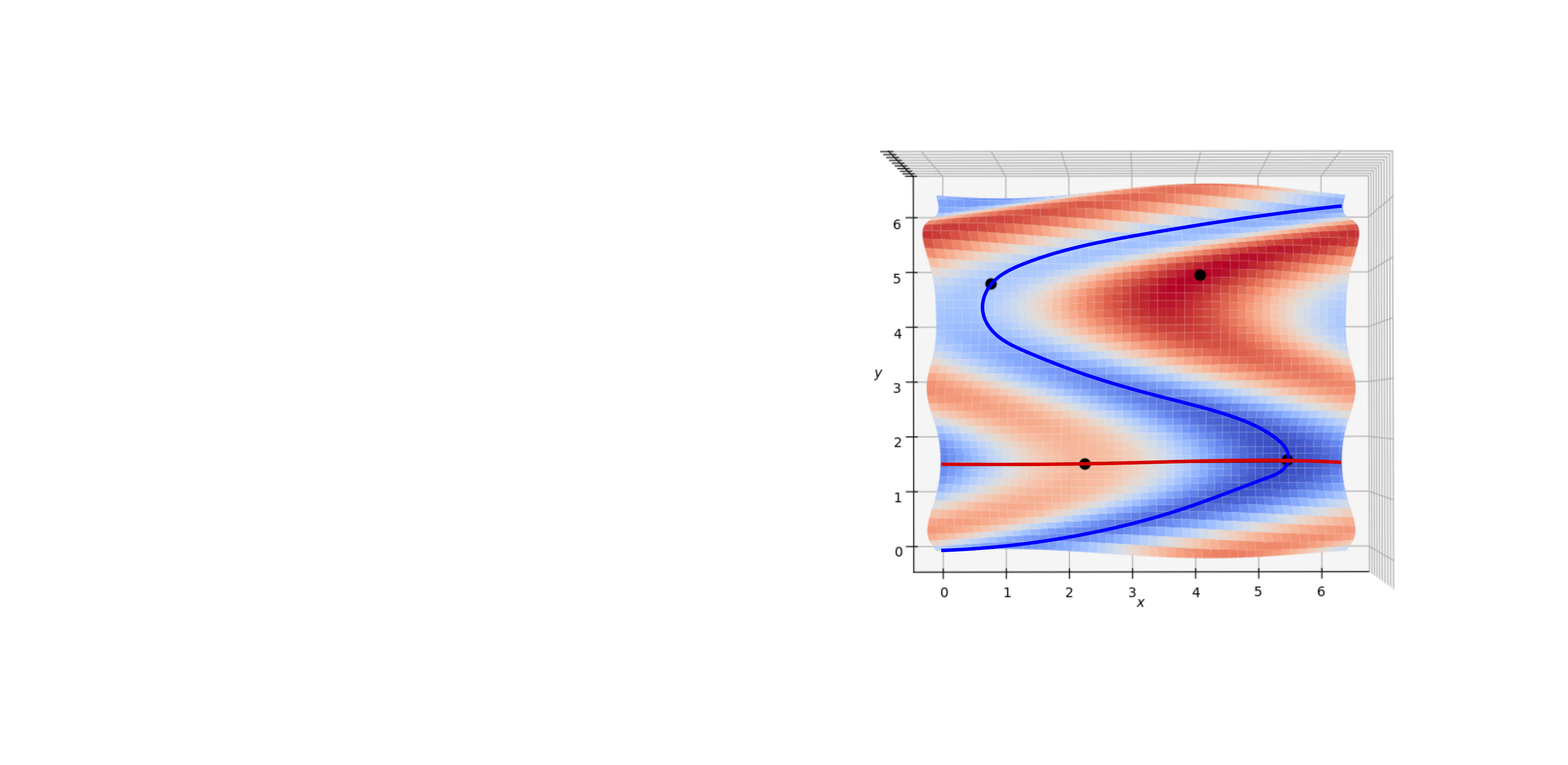}}%
    \put(0.77463709,0.28510094){\makebox(0,0)[lt]{\lineheight{1.25}\smash{\begin{tabular}[t]{l}$m$\end{tabular}}}}%
    \put(0.83381288,0.1641765){\makebox(0,0)[lt]{\lineheight{1.25}\smash{\begin{tabular}[t]{l}$v$\end{tabular}}}}%
    \put(0.64213479,0.28124165){\makebox(0,0)[lt]{\lineheight{1.25}\smash{\begin{tabular}[t]{l}$s_1$\end{tabular}}}}%
    \put(0.70002414,0.1680358){\makebox(0,0)[lt]{\lineheight{1.25}\smash{\begin{tabular}[t]{l}$s_2$\end{tabular}}}}%
    \put(0,0){\includegraphics[width=\unitlength,page=2]{NR-ex.pdf}}%
  \end{picture}%
\endgroup%

		\caption{Two views of the graph of the Morse function $U(x,y)\coloneqq 3-\sin y-2\cos(x-y-4\cos(y-a))$ on the flat $2$-torus viewed as the square $[0,2\pi]\times [0,2\pi]$ with opposite boundary faces identified, where $a \coloneqq \arccos(1/4)$. 
		There are precisely four critical points: one minimum $v$, one maximum $m$, and two saddles $s_1$, $s_2$.
		The blue and red curves are hand-drawn approximations of the unstable manifolds of the two saddles with respect to $-\nabla U$.
		Since $U(s_2) > U(s_1)$ and $s_1$ is not in the closure of the unstable manifold of $s_2$, $U$ satisfies Assumption~\ref{assump:morse-smale-intro}.
		Note: the unstable manifolds of the two saddles coincide with the two edges of the \emph{undirected Morse graph} of $-\nabla U$ with vertex set $\{v\}$, to be defined in \S \ref{sec:setup-morse}; the \emph{directed Morse graph} has edge set consisting of two oppositely oriented copies of each of these undirected edges, with the same vertex set $\{v\}$.
		}\label{fig:NR-ex-1}
	\end{figure}
	
Let $M$ be the flat $2$-torus $\tor^2$ viewed as the square $[0,2\pi]^2$ with opposite boundary faces identified, coordinates $(x,y)$, and Riemannian metric given by the restriction of the Euclidean metric to $[0,2\pi]^2$.	
Consider $U\in \Cont^\infty(M)$ defined by
\begin{equation}\label{eq:ex-torus}
U(x,y)\coloneqq  3-\sin y-2\cos(x-y-4\cos(y-a)), \qquad a \coloneqq \arccos(1/4).
\end{equation} 
This function was chosen because its graph contains the S-shaped ``valley'' shown in blue in Fig.~\ref{fig:NR-ex-1}.
Using the notations $\partial_x$ and $\partial_y$ for partial derivatives, we compute
\begin{equation}\label{eq:ex-torus-U-partial-derivs}
\partial_x U(x,y) = 2\sin(x-y-4\cos(y-a)), \qquad \partial_y U(x,y) =[4\sin(y-a)-1]\partial_x U(x,y) -\cos y.
\end{equation}
Setting $\partial_x U = \partial_y U = 0$ and solving for $(x,y)$ using \eqref{eq:ex-torus-U-partial-derivs}, we find that $U$ has four critical points:
\begin{alignat*}{2}
v &= (\pi/2 + \sqrt{15}, \pi/2),  \qquad &&m = (5\pi/2-\sqrt{15}, 3\pi/2),\\
s_1 &= (3\pi/2-\sqrt{15}, 3\pi/2), \qquad &&s_2= (\sqrt{15}-\pi/2,\pi/2),
\end{alignat*}
at which $4\cos(y-a) = \sqrt{15}\sin y$ and $4\sin(y-a)=\sin y$, so
\begin{alignat*}{2}
U(v) &= 0, \qquad &&U(m) = 6,\\
U(s_1) &= 2, \qquad && U(s_2) = 4.
\end{alignat*}
Next, again using that $4\cos(y-a) = \sqrt{15}\sin y$ and $4\sin(y-a)=\sin y$ at critical points, for any $(x_0,y_0)\in \{v,s_1,s_2,m\}$ we calculate
\begin{alignat*}{2}
\partial_x^2 U(x_0,y_0) &= 2\cos(x_0-y_0-\sqrt{15}\sin y_0), \qquad && \partial_y^2 U= \sin y_0 + (1-\sin y_0 )^2\partial_x^2 U(x_0,y_0),\\
\partial_x \partial_y U(x_0,y_0) & = (\sin y_0-1)\partial_x^2 U(x_0,y_0) , \qquad && \partial_y \partial_x U(x_0,y_0) = \partial_x \partial_y U(x_0,y_0).
\end{alignat*}
From the first expression it follows that $\partial_x^2 U(v) = 2$, $\partial_x^2 U (m) = -2$, $\partial_x^2 U(s_1) = 2$, and $\partial_x^2 U(s_2) = -2$.
From this and the above, it follows that the Hessian matrices of $U$ (containing the second derivatives) at the critical points are
\begin{alignat*}{2}
\Hess_v U &= \begin{bmatrix}
2 & 0\\
0 & 1
\end{bmatrix}, \qquad && \Hess_m U = \begin{bmatrix}
-2 & 4\\
4 & -9
\end{bmatrix},\\
\Hess_{s_1} U &= \begin{bmatrix}
2 & -4\\
-4 & 7
\end{bmatrix}, \qquad && \Hess_{s_2}U = \begin{bmatrix}
-2 & 0\\
0 & 1
\end{bmatrix}.
\end{alignat*}
We see by inspection that both eigenvalues of $\Hess_v U$ are positive and that $\Hess_{s_2} U$ has one positive and one negative eigenvalue. 
Since $\det(\Hess_m U) > 0$,  $\trace(\Hess_m U) < 0$, and $\det(\Hess_{s_1} U) < 0$, and since the eigenvalues of symmetric matrices are real, it follows that both eigenvalues of $\Hess_m U$ are negative and that $\Hess_{s_1} U$ has one positive and one negative eigenvalue.
Hence $U$ is a Morse function with one minimum $v$, two index-$1$ saddle points $s_1, s_2$, and one maximum $m$.
Since $U(s_2) > U(s_1)$ and $s_1$ is not in the closure of the unstable manifold of $s_2$ (see Fig.~\ref{fig:NR-ex-1}), $U$ satisfies Assumption~\ref{assump:morse-smale-intro}.

Letting $\cft = dx$ with (Euclidean) metric dual given by the coordinate vector field $\cft^\sharp = \partial_x$, we now consider a diffusion process on $\tor^2$ with generator $L_\varepsilon = \dft + \varepsilon \Delta$, where
\begin{equation*}
\dft = -\nabla U + c \partial_x
\end{equation*}
for $c \geq 0$ as in \eqref{eq:dft-decomp-intro} and $\Delta = \partial_x^2 + \partial_y^2$.\footnote{As observed in Footnote~\ref{footnote:diffusion-as-SDE}, this diffusion process can be constructed as the solution of the SDE $dX_t = \dft(X_t)dt + \sqrt{2\varepsilon}dW_t$, where $W_t$ is Brownian motion \cite[Sec.~4.3]{mckean2005stochastic}.}
Note that $v_*(0)$ as defined in \S \ref{sec:contributions} satisfies $v_*(0) = v$.
Referring to \eqref{eq:height-loop} and \eqref{eq:h-star-def}, the hues in Fig.~\ref{fig:NR-ex-1} show that the infimum
\begin{equation*}
\begin{split}
h_*(0) \coloneqq& \inf \{\heightz(\gamma)\colon \gamma\in \Omega_{v_*(0)} \tor^2 \textnormal{ and } \int_\gamma dx > 0\}\\	
=& \inf \left\{\sup_{t\in [0,1]}U(\gamma(t))\colon \gamma\in \Omega_{v_*(0)} \tor^2 \textnormal{ and } \int_\gamma dx > 0.\right\}
\end{split}	
\end{equation*}
is attained by any loop $\gamma\in \Omega_{v_*(0)}$ which starts at $v_*(0)$ (darkest blue in Fig.~\ref{fig:NR-ex-1}), ascends (as measured by $U$) while traveling up and to the left within the blue valley to $s_1$ (lightest blue in Fig.~\ref{fig:NR-ex-1}), then descends while traveling up and to the right within the blue valley back to $v_*(0)$.  
Since $U(v_*(0)) = 0$, it follows that $h_*(0) = U(s_1)$.
When $c > 0$ we can write $\dft|_{(0,2\pi)\times [0,2\pi]} = -\nabla \tilde{U}_c$, where the effective ``tilted potential'' $\tilde{U}_c(x,y)\coloneqq U(x,y)-cx$ is smooth when restricted to the image of $(0,2\pi)\times [0,2\pi]$ in $\tor^2$, but $\tilde{U}_c$ does not extend to a continuous function on $\tor^2$ when $c > 0$. 
Let us write $s_1 = (x_1,y_1)$ and $v = (x_v, y_v)$.
Since (i) the graph of $\tilde{U}_c$ is obtained simply by adding a linear downward ``tilt'' in the direction of increasing $x$ to $\textnormal{graph}(U)$ in Fig.~\ref{fig:NR-ex-1} 
and (ii) there are smooth curves $c\mapsto s_1(c)$ and $c\mapsto v_*(c)$ of $c$-dependent hyperbolic zeros of $\dft$ satisfying $s_1(0) = s_1$ and $v_*(0) = v$ by the implicit function theorem, we see that $h_*(c) = \tilde{U}_c(s_1(c))- \tilde{U}_c(v_*(c))$ for sufficiently small $c> 0$.
Thus, the $c$-derivative $h_*'(0)$ satisfies 
\begin{align*}
h_*'(0) &= \frac{d}{dc}[\tilde{U}_c(s_1(c))- \tilde{U}_c(v_*(c))]|_{c=0}\\
& = \ip{\underbrace{\nabla U(s_1)}_{0}}{s_1'(0)} - \ip{\underbrace{\nabla U(v)}_{0}}{v_*'(0)} + \underbrace{(x_v-x_1)}_{> 0} > 0.
\end{align*}
Thus, $c\mapsto h_*(c)$ is a strictly increasing function for $c$ sufficiently small.
Cor.~\ref{co:qualitative-intro-nr} then implies that, for sufficiently small $c_2 > c_1 > 0$ and all sufficiently small $\varepsilon > 0$, $\flux_{\varepsilon,c_2}([dx])< \flux_{\varepsilon, c_1}([dx])$: there is negative resistance.

\section{Flux}\label{sec:flux}
\subsection{General setup}\label{sec:setup-flux}
As explained in \S \ref{sec:general-setting}, we may assume that the $\varepsilon$-family of nondegenerate diffusion processes $(X^\varepsilon_t, \Prob_x^\varepsilon)$ on the closed connected smooth $n$-dimensional manifold $M$ has generator
\begin{equation}\label{eq:diff-gen}
L_\varepsilon = \dfte + \varepsilon \Delta,
\end{equation}
where $\Delta = \nabla \cdot \nabla$ is the Laplace-Beltrami operator of some smooth Riemannian metric $G$ on $M$, the \concept{drift} vector field $\dfte$ is smooth for each $\varepsilon> 0$, and $\dfte\to \dft$ uniformly as $\varepsilon \to 0$. 
In the sequel we mostly use the notations $\ip{\slot}{\slot}$ and $\norm{\slot}$ instead of $g(\slot,\slot)$. 
Compactness of $M$ implies the existence of a unique \concept{stationary (probability) density} $\dme\in \Cont^\infty(M)$ satisfying the stationary Fokker-Planck equation
\begin{equation}\label{eq:fokker-planck}
0 = \nabla \cdot \underbrace{(\dme \dfte - \varepsilon \nabla \dme)}_{\cme} \eqqcolon  \nabla \cdot \cme,
\end{equation}
and to which all initial probability densities converge as $t\to\infty$ in the $\Cont^0$ topology \cite[Thm~3]{zeeman1988stability}.
The $\Cont^\infty$ divergence-free vector field $\cme \coloneqq  \dme \dfte - \varepsilon \nabla \dme$  on $M$ is the \concept{steady-state (probability) current.}

\begin{Rem}\label{rem:vanishing-current-iff-gradient}
The following chain of equivalent statements concerning the diffusion process with generator \eqref{eq:diff-gen} may provide some intuition concerning $\cme$ \cite{kent1978time, ikeda1981stochastic}.
The ``detailed balance'' condition $\dme \dfte = \varepsilon\nabla \dme$ holds \cite[p.~825]{kent1978time} $\iff$ $\cme \equiv 0$ $\iff$ $\dfte = \varepsilon\nabla \ln(\dme)$ $\iff$ $\dfte = -\nabla U_\varepsilon$ for some $U_\varepsilon\in \Cont^\infty(M)$ (in which case $\dme \propto e^{-\frac{1}{\varepsilon}U_\varepsilon}$) $\iff$ the diffusion is ``reversible'' \cite[p.~827]{kent1978time} $\iff$ the diffusion is ``$\dme$-symmetric'' in the sense that $\dme(x)\rho_\varepsilon(t,x,y) = \dme(y)\rho_\varepsilon(t,y,x)$ for all $x,y\in M$, where $\rho_\varepsilon(t,x,y)$ is the probability density of a transition from $x$ to $y$ at time $t$ \cite[p.~824]{kent1978time} $\iff$ the generator $L_\varepsilon$ is ``symmetrizable'' in the sense that there exists a Borel measure $\mu_\varepsilon$ on $M$ with $\int_M (L_\varepsilon f)gd\mu_\varepsilon = \int_M f(L_\varepsilon g)d\mu_\varepsilon$ for all $f,g\in \Cont^\infty(M)$, in which case $d\mu_\varepsilon(x) = \dme(x)dx$ \cite[pp.~275--276]{ikeda1981stochastic}.
\end{Rem}

\subsection{Motivation of flux and its basic properties}\label{sec:motivation-flux-properties}
 
To motivate our general definition of steady-state flux, let $N\subset M$ be a closed hypersurface equipped with a transverse orientation determined by a choice $\hat{n}$ of unit normal vector field.\footnote{Recall that we are not assuming $M$ is orientable. If $M$ happens to be orientable, then $N$ is transversely orientable if and only if $N$ is orientable, and all of the differential forms appearing in the proof of Prop.~\ref{prop:poinc-dual-special-case} can be assumed to be the usual kind (forms of even type \cite[Ch.~2]{derham1984differentiable}).} 
The usual flux of $\cme$ through $N$ is $\int_{N}\ip{\cme}{\hat{n}}dy$, where $dy$ is the Riemannian density \cite[p.~432]{lee2013smooth} of the metric restricted to $N$.
The following result follows from the nonorientable version of Poincar\'{e} duality \cite[p.~87,~Thm~7.8]{bott1982differential}.
\begin{Prop}\label{prop:poinc-dual-special-case}
Let $M$ be a closed Riemannian manifold and $N\subset M$ be a closed hypersurface equipped with a smooth unit normal vector field $\hat{n}$.
Then there exists a unique cohomology class $[\cfo]\in \Hdr^1(M)$ such that, for any $\cfo\in [\cfo]$ and any divergence-free smooth vector field $J$ on $M$, 
\begin{equation}\label{eq:prop:poinc-dual-special-case-1}
\int_{N}\ip{J}{\hat{n}}dy = \int_M \cfo(J)\dV.
\end{equation}
\end{Prop}
\begin{proof}
Let $\lrcorner$ denote the interior product of vector fields and forms.
Since the pullback to $N$ of $J\lrcorner \dV$ via the inclusion $N\hookrightarrow M$ is $\ip{J}{\hat{n}}dy$, the left side of \eqref{eq:prop:poinc-dual-special-case-1} is equal to $\int_N J\lrcorner \dV$. 
Here $J \lrcorner \dV$ is an $(n-1)$-form twisted by the orientation bundle \cite{bott1982differential}---or, alternatively, an $(n-1)$-form of odd type \cite[Ch.~2]{derham1984differentiable}---which can be integrated over a transversely oriented hypersurface \cite[Sec.~3.4(b)]{frankel2012geometry}.
Now let $[\cfo]\in \Hdr^1(M)$ be the unique de Rham cohomology class Poincar\'{e} dual to $N$ \cite[p.~87,~Thm~7.8]{bott1982differential}.
We compute
\begin{equation}\label{eq:flux-poinc-duality}
\int_N J \lrcorner dx = \int_N  *(J^\flat) = \int_M \cfo \wedge *(J^\flat) = \int_M \cfo(J) \dV, 
\end{equation}
where $(\slot)^\flat$ is the inverse to $(\slot)^\sharp$ sending a vector field to its dual one-form via the metric $G$, and $*$ is the Hodge star mapping $k$-forms of even type to $(n-k)$-forms of odd type \cite[p.~101]{derham1984differentiable}.
The first equality follows since $J \lrcorner \dV = *(J^\flat)$, the second equality follows by the definition of Poincar\'{e} duality\footnote{Here we are using a sign convention for the Poincar\'{e} dual of $M$ which is different from the one in \cite[p.~51]{bott1982differential}; the cited convention would stipulate that $\int_N *(J^\flat) = \int_M *(J^\flat) \wedge \cfo =  (-1)^{n-1} \int_M \cfo \wedge  *(J^\flat) $.} and the fact that $d*(J^\flat) = 0$ since $\nabla \cdot J = 0$, and the third equality follows since $\cfo\wedge *\cft = \cfo(\cft^\sharp)\dV$ for any one-forms $\cfo$ and $\cft$ (of even type). 
\end{proof}
Motivated by Prop.~\ref{prop:poinc-dual-special-case}, we define the (steady-state) \concept{flux} of the diffusion with generator \eqref{eq:diff-gen} to be the linear map
\begin{equation}\label{eq:flux-def}
\fluxe\colon \Hdr^1(M)\to \R, \qquad \fluxe([\cfo])\coloneqq \int_M \cfo(\cme)\dV,
\end{equation}
and we refer to $\fluxe([\cfo])$ as the (steady-state) \concept{$[\cfo]$-flux}.
This map is well-defined since $\nabla \cdot \cme = 0$ implies that the integral does not depend on the representative $\cfo$ of the cohomology class.
If $\cfo$ happens to be the unique harmonic representative of $[\cfo]$, then $\int_M \cfo(\nabla \dme)\, \dV = 0$, so the alternative formula $\fluxe([\cfo]) = \int_{M}\cfo(\dfte)\dme \dV$ also holds in this special case (cf. the second integral in \eqref{eq:flux-int-2}).
Manabe proved that the following equality holds with probability one \cite[Thm~4.1]{manabe1982stochastic}:\footnote{Actually, in \cite{manabe1982stochastic} it is assumed that $M$ is orientable, but the claim for nonorientable $M$ follows by considering the lifted diffusion process $(\hat{X}^\varepsilon_t,  \hat{\Prob}^{\varepsilon}_x)$ on the orientation (double) covering $\hat{\pi}\colon \hat{M}\to M$ \cite[p.~394]{lee2013smooth}.}
\begin{equation}\label{eq:flux-micro}
\fluxe([\cfo]) \overset{\textnormal{a.s.}}{=} \lim_{t\to\infty} \frac{1}{t}\int_{X^{\varepsilon}_{[0,t]}}\cfo.
\end{equation}
Here $\int_{X^{\varepsilon}_{[0,t]}}\cfo$ is the line integral of $\cfo$ along the diffusion process restricted to the time interval $[0,t]$ \cite[Sec.~2]{manabe1982stochastic}, \cite[Ch.~VI.6]{ikeda1981stochastic}.
For the torus example of \S \ref{sec:intro}, the expressions in \eqref{eq:flux-def} and \eqref{eq:flux-micro} correspond to the second and fourth expressions in \eqref{eq:flux-int-2} by taking $\cfo = dx^1$.

\begin{Rem}\label{rem:flux-relationship-with-jiang-manabe}
The right side of \eqref{eq:flux-micro} is called a ``rotation number'' in \cite[Sec.~5.4]{jiang2004mathematical} and is essentially equivalent to a ``stochastic intersection number'' as defined in \cite{manabe1982stochastic}.
\end{Rem}
\begin{Rem}\label{rem:flux-relationship-with-asymptotic-cycles}
Our definition \eqref{eq:flux-def} of flux and its characterization \eqref{eq:flux-micro} are closely related to Schwartzman's notion of the \emph{asymptotic cycle} of a continuous flow on a compact metric space \cite{schwartzman1957asymptotic}; see also \cite[App.~16.3]{arnold1968ergodic}.
In fact, if $\dft_0$ is a smooth vector field generating a flow which is ergodic with respect to a smooth invariant probability density $\dm_0$, and $\flux_0\colon \Hdr^1(M)\to \R$ is defined by setting $\varepsilon = 0$ in \eqref{eq:fokker-planck} and \eqref{eq:flux-def}, then the corresponding equality \eqref{eq:flux-micro} holds for $\dm_0$-almost every initial condition, and $\flux_0\colon \Hdr^1(M)\to \R$ coincides with Schwartzman's $(\dm_0\dV)$-asymptotic cycle \cite[Sec.~3]{athanassopoulos1995some}.
A similar observation was made in \cite[p.~122]{jiang2004mathematical}.
The terminology ``asymptotic cycle''  comes from the fact that, since $\Hdr^1(M)$ is the dual space of the first real singular homology $H_1(M;\R)$ by the de Rham and universal coefficient theorems, $\flux_0$ can be identified with a real homology class.
A similar perspective on $\fluxe$ with $\varepsilon > 0$ was taken by Manabe \cite{manabe1982stochastic} who referred to the ``asymptotic homological position'' of the diffusion $X^\varepsilon_t$, which coincides with $\fluxe$ when the latter is identified with a real homology class. 
\end{Rem}

Much of the motivation for the present paper concerns the case that $\dfte \equiv \cfo^\sharp$ is dual to a smooth closed one-form $\cfo$ (though the main results of \S \ref{sec:main-results-morse} require only that $\dfte\to \cfo^\sharp$ uniformly as $\varepsilon \to 0$; cf. Footnote~\ref{footnote:epsilon-dependent-drift}).
For this case the next result, for a fixed $\varepsilon > 0$, shows that the flux $\fluxe([\cfo])$ is nonnegative and vanishes if and only if $\cfo$ is exact.
This is achieved by establishing \eqref{eq:flux-witten-fisher} which, unlike \eqref{eq:flux-def} and \eqref{eq:flux-micro} (and, e.g., \eqref{eq:flux-int-1} and \eqref{eq:flux-int-2}), establishes steady-state $[\cfo]$-flux as a manifestly nonnegative quantity when $\dfte = \cfo^\sharp$.

\begin{Prop}\label{prop:flux-positive-c1f-case}
Under the assumptions of \S\ref{sec:setup-flux}, the drift vector field $\dfte$, stationary density $\dme$, and steady-state current $\cme$ satisfy
\begin{equation}\label{eq:J-squared-int}
\int_M \ip{\dfte}{\cme}\dV = \int_M \frac{\norm{\cme}^2}{\dme} \dV,
\end{equation}	
and these quantities are zero if and only if $\dfte = -\nabla U$ for some $U\in \Cont^\infty(M)$.
In particular, if $\cfo$ is the smooth one-form dual to $\dfte = \cfo^\sharp$ so that  $\ip{\dfte}{\cme}=\cfo(\cme)$, 
\begin{equation}\label{eq:flux-witten-fisher}
\fluxe([\cfo]) = \int_M \frac{\norm{\cme}^2}{\dme} \dV \geq 0
\end{equation}
with equality if and only if $\cfo$ is exact.
\end{Prop}
\begin{Rem}\label{rem:entropy-prod}
Since the right side of \eqref{eq:J-squared-int} is equal to 
\begin{equation}\label{eq:entropy-prod}
\int_M \norm{\cme/\dme}^2\dme\dV = \int_M \norm{\dfte-\nabla (\ln \dme)}^2\dme\dV,
\end{equation}
by \cite[Thm~5.3.6]{jiang2004mathematical} it coincides with the \emph{entropy production rate} \cite[Def.~5.3.4]{jiang2004mathematical}.
\end{Rem}

\begin{Rem}\label{rem:loc-symm}
In the terminology of \cite[p.~276]{ikeda1981stochastic}, the case that $\dfte = \cfo^\sharp$ is dual to a closed one-form is precisely the case that the diffusion is \emph{locally symmetrizable} \cite[p.~279,~Thm~4.6(ii)]{ikeda1981stochastic}.
\end{Rem}

\begin{proof}
    Since $\nabla \cdot \cme = 0$, it follows that
    \begin{equation}
    \frac{\norm{\cme}^2}{\dme} = \ip{\dfte - \nabla(\ln \dme)}{\cme} = \ip{\dfte}{\cme} - \nabla \cdot [(\ln \dme)\cme].
    \end{equation}
    Since $\partial M = \varnothing$, the divergence theorem (which holds even if $M$ is not orientable \cite[Thm~16.48]{lee2013smooth}) implies that the integral of $\nabla \cdot [(\ln \dme)\cme]$ over $M$ vanishes.
    This proves \eqref{eq:J-squared-int}, and the quantities in that equation are zero if and only if $\cme$ is identically zero.
    By Rem.~\ref{rem:vanishing-current-iff-gradient}, this is the case if and only if $\dfte = -\nabla U$ for some $U\in \Cont^\infty(M)$.\footnote{Here is a proof. If $\cme \equiv 0$ then $\dfte = \nabla (\ln \dme)$ by definition of $\cme$. Conversely, if $\dfte = -\nabla U$ for some $U\in \Cont^\infty(M)$, then $e^{-\frac{1}{\varepsilon}U}\nabla U + \varepsilon\nabla(e^{-\frac{1}{\varepsilon}U}) = e^{-\frac{1}{\varepsilon}U}\nabla U  - \varepsilon\frac{1}{\varepsilon}e^{-\frac{1}{\varepsilon}U}\nabla U = 0$, so the unique solution $\dme$ to \eqref{eq:fokker-planck} is a scalar multiple of $e^{-\frac{1}{\varepsilon}U}$, and the computation just performed shows that the corresponding probability current $\cme \propto e^{-\frac{1}{\varepsilon}U}\nabla U + \varepsilon\nabla(e^{-\frac{1}{\varepsilon}U}) \equiv 0$.}
    Eq.~\eqref{eq:flux-witten-fisher} and the statement thereafter follow immediately.
\end{proof}

\section{Tilted potentials: drifts dual to Morse-Smale closed one-forms}\label{sec:main-results-morse}
\subsection{Setup}\label{sec:setup-morse} 
As in \S \ref{sec:setup-flux}, let $M$ be a closed connected Riemannian manifold of positive dimension.
In this section we will also assume that $\dft = \cfo^\sharp$ is the metric dual of a $\Cont^1$ closed one-form $\cfo$ and that all zeros of $\dft$ are hyperbolic.
Each zero $z\in \dft^{-1}(0)$ then has well-defined stable and unstable manifolds $\Ws(z)$ and $\Wu(z)$ for the flow of $\dft$ and a well-defined \concept{(Morse) index} $\ind(z)\coloneqq \dim(\Wu(z))$.
We will also assume that $\dft$ is ``close to being a generic gradient'' in the following sense.
\begin{Assump}\label{assump:morse-smale}
Every initial condition converges to a zero of $\dft$ under the flow of $\dft$, and all stable and unstable manifolds of zeros of $\dft$ intersect transversely.
\end{Assump}
\begin{Rem}\label{rem:equivalent-assump-chain-recurrent}
Equivalent statements of Assumption~\ref{assump:morse-smale} are: (i) $\dft$ is a Morse-Smale\footnote{We recall that a \concept{Morse-Smale vector field} $\dft$ is one whose chain recurrent set consists of finitely many equilibria and periodic orbits with (un)stable manifolds having only pairwise transverse intersections \cite[pp.~118--119]{palis1980geometric}. 
A function $U\in \Cont^2(M)$ is \concept{Morse-Smale} if the vector field $-\nabla U$ is Morse-smale (this depends on the Riemannian metric).\label{footnote:morse-smale-def}} vector field without nonstationary periodic orbits, and (ii) $\dft$ is a Morse-Smale vector field with a finite chain recurrent set (Def.~\ref{def:chain-recurrent}). 
\end{Rem}
These assumptions result in the standard Morse-Thom-Smale-Witten complex \cite[Ch.~6]{pajitnov2006circle}, from which we extract the \emph{undirected Morse graph} (see Fig.~\ref{fig:NR-ex-1} for an example).
\begin{Def}\label{def:morse-graph-undirected}
The \concept{undirected Morse graph} $\comp = (V,\Eu,\st)$ is the undirected graph with vertex set $V$ consisting of the index-$0$ zeros of $\dft$ and edge set $\Eu$ consists of the unstable manifolds of index-$1$ zeros of $\dft$.
The map $\st\colon \Eu\to \{S\in  2^V\colon \#(S)=2\}$ sends an edge $e$ to the two-element subset containing its end vertices.
\end{Def}
Note that $\vtx$ coincides with the set of zeros which are stable under the flow of $\dft$.
\begin{Rem}\label{rem:multiple-loop-edges}
$\comp$ may be naturally viewed as a $1$-dimensional CW complex embedded in $M$.
Note also that $\comp$ may contain loop edges at a single vertex and multiple edges between the same pair of vertices; similarly for the \emph{directed Morse graph} defined now (see Fig.~\ref{fig:NR-ex-1} and its caption for an example). 
\end{Rem}

\begin{Def}\label{def:morse-graph-directed}
The \concept{directed Morse graph} $\compd = (V,\Ed, \src, \tgt)$ is the directed graph consisting of two oppositely directed copies of each edge in $\comp$, with the same vertex set as $\comp$.
The maps $\src,\tgt\colon \Ed\to V$ send directed edges to their source and target vertices.
We denote by $q\colon \Ed\to \Eu$ the two-to-one map sending a directed edge to its undirected version in $\Eu$; note that $\{\src(e),\tgt(e)\} = \st(q(e))$ for all $e\in \Ed$.  
\end{Def}
The restriction of $\cfo$ to any edge of $\comp$ is exact, with a primitive that increases from the ends toward a unique global maximizer in the interior of the edge (an index-$1$ zero of $\dft$).
Thus, the following definition makes sense.
\begin{Def}\label{eq:gains-morse} 
For each $e\in \Ed$, we define the \concept{gain} $g(e) > 0$ to be the integral of $-\cfo$ along any path in $q(e)$ from $\src(e)$ to the unique index-$1$ zero of $\dft$ in $q(e)$.
\end{Def}
\begin{figure}
	\centering
	\def\svgwidth{1.0\linewidth}
\begingroup%
  \makeatletter%
  \providecommand\color[2][]{%
    \errmessage{(Inkscape) Color is used for the text in Inkscape, but the package 'color.sty' is not loaded}%
    \renewcommand\color[2][]{}%
  }%
  \providecommand\transparent[1]{%
    \errmessage{(Inkscape) Transparency is used (non-zero) for the text in Inkscape, but the package 'transparent.sty' is not loaded}%
    \renewcommand\transparent[1]{}%
  }%
  \providecommand\rotatebox[2]{#2}%
  \newcommand*\fsize{\dimexpr\f@size pt\relax}%
  \newcommand*\lineheight[1]{\fontsize{\fsize}{#1\fsize}\selectfont}%
  \ifx\svgwidth\undefined%
    \setlength{\unitlength}{825.3896175bp}%
    \ifx\svgscale\undefined%
      \relax%
    \else%
      \setlength{\unitlength}{\unitlength * \real{\svgscale}}%
    \fi%
  \else%
    \setlength{\unitlength}{\svgwidth}%
  \fi%
  \global\let\svgwidth\undefined%
  \global\let\svgscale\undefined%
  \makeatother%
  \begin{picture}(1,0.2909419)%
    \lineheight{1}%
    \setlength\tabcolsep{0pt}%
    \put(0,0){\includegraphics[width=\unitlength,page=1]{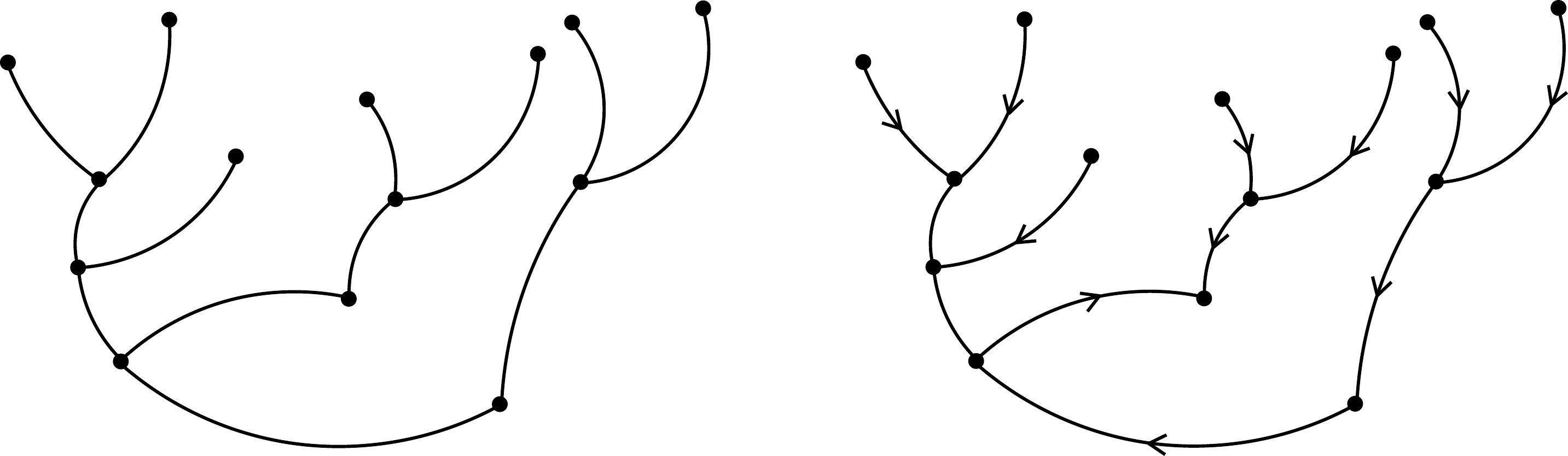}}%
    \put(0.7549608,0.08183485){\color[rgb]{0,0,0}\makebox(0,0)[lt]{\lineheight{1.25}\smash{\begin{tabular}[t]{l}$v$\end{tabular}}}}%
    \put(0,0){\includegraphics[width=\unitlength,page=2]{st-rst.pdf}}%
  \end{picture}%
\endgroup%

	\caption{Left: a spanning tree in an undirected graph. Right: a rooted spanning tree with root $v$ in a directed graph. All vertices in both graphs are shown. In both types of graphs we allow loop edges at a single vertex and multiple edges between the same pair of vertices, but edges other than those forming the spanning trees are not shown above.}\label{fig:st-rst}
\end{figure}

Given $v\in \vtx$, we denote by $\RST(\compd;v)\subset 2^{\Ed}$ the set of \concept{rooted spanning trees} with edges directed toward the root $v$ \cite{pitman2018tree}, and we define $\RST(\compd)\coloneqq \bigcup_{v\in \vtx}\RST(\compd;v)$.
We denote by $\ST(\comp)$ the set of \concept{undirected spanning trees} in $\comp$.
See Fig.~\ref{fig:st-rst}.
We assume for now that there is a unique minimizer
\begin{equation}\label{eq:morse-rst-minimizer}
T_*\in \arg\min_{T\in \RST(\compd)} \sum_{e\in T}g(e),
\end{equation}
and let $v_*\in \vtx$ be the root of $T_*$. 
\begin{Def}\label{def:morse-heights}
The \concept{height} $h(v)$ of a vertex $v\in \vtx$ is defined to be the integral of $-\cfo$ along the unique path in $q(T_*)$ leading from $v_*$ to $v$.
For $e\in \Ed$, we define the \concept{height} $h(e)\coloneqq h(\src(e)) + g(e)$.
\end{Def}
 
\begin{Rem}\label{rem:edge-height-homology}
Modulo choices of orientations, the first homology group $\Hom_1(\comp)$ of $\comp$ has a natural basis enumerated by $\Eu\setminus q(T_*)$; to any such edge $e$ one associates the unique simple cycle in $\{e\}\cup q(T_*)$.
From this observation and Def.~\ref{def:morse-heights} we see that, for any edge $e\in q(T_*)$ with $\{e_1,e_2\} \coloneqq q^{-1}(e)$, $h(e_1)=h(e_2)$.
Furthermore, if the cycle in $\Hom_1(\comp)$ corresponding to an edge $e\in \Eu\setminus q(T_*)$ maps to zero in $\Hom_1(M)$, then $h(e_1)=h(e_2)$.
More generally, the integral of $\cfo$ along the cycle in $\Hom_1(\comp)$ corresponding to an edge $e\in \Eu\setminus q(T_*)$ is nonzero if and only if $h(e_1) \neq h(e_2)$, where $\{e_1,e_2\} \coloneqq q^{-1}(e)$.
\end{Rem}
\begin{Lem}\label{lem:h-values-positive}
$h(v) \geq 0$ for all $v\in \vtx$ and $h(e)>0$ for all $e\in \Ed$.
\end{Lem}
\begin{proof}
If $h(v)< 0$ for some $v\in \vtx$, reversing the edges along the path in $T_*$ connecting $v$ to $v_*$ would yield $T\in \RST(\compd;v)$ satisfying $\sum_{e\in T}g(e)< \sum_{e\in T_*}g(e)$, contradicting \eqref{eq:morse-rst-minimizer}.
If $e\in \Ed$, it follows that $h(e) = g(e)+h(\src(e)) > h(\src(e))\geq 0$.
\end{proof}
Given a directed edge $e\in \Ed$, we denote by $\bar{e}\in \Ed$ its \concept{reversal}: $q(\bar{e})=q(e)$, but $\bar{e}$ has the opposite orientation.
Using Def.~\ref{def:morse-heights} we define, in the case that $\cfo$ is not exact,
\begin{equation}\label{eq:h-star-def-morse}
\begin{split}
h_* &\coloneqq \min \{h(e)\colon e\in \Ed \textnormal{ and } h(e) \neq h(\bar{e})\}\\
&= \min \{h(e)\colon e\in \Ed \textnormal{ and } h(e) < h(\bar{e})\}.
\end{split}
\end{equation}
When $\cfo$ is exact, both sets on the right are empty; when $\cfo$ is not exact, they are not.
In \S \ref{sec:th-2-implies-th-1} we show (Lem.~\ref{lem:h-star-defs-coincide}) that this definition of $h_*$ coincides with that given in \eqref{eq:h-star-def} under the assumptions of \S \ref{sec:summary-first-main-result}.
 
\subsection{Main result} 
Theorem~\ref{th:qualitative} below is our first main result; in \S \ref{sec:th-2-implies-th-1} we use it to prove Theorem~\ref{th:qualititative-special-case-intro} of \S \ref{sec:summary-first-main-result}. 
For the statement we introduce the notation $\comp^U = (\vtx^U, \Eu^U, \st)$ for the analogous undirected Morse graph of a Morse-Smale gradient $-\nabla U$, while all other notations remain as in \S\ref{sec:setup-morse}.
Given $e\in \Eu^U$, we denote by $U(e)\coloneqq \max_{x\in e}U(x)$ the value of $U$ at the unique index-$1$ zero of $-\nabla U$ in $e$.

\begin{restatable}[]{Th}{ThmQualitative}\label{th:qualitative}
Let $U\in \Cont^2(M)$ be a Morse-Smale function on a closed connected Riemannian manifold $M$, such that $U$ has a unique global minimizer and
\begin{equation}\label{eq:th-qualitative-minimizer}
\textnormal{the minimizer of}\quad \min_{T\in \ST(\comp^U)}\sum_{e\in T}U(e) \quad \textnormal{is unique.}
\end{equation}
Then for any closed but not exact $\Cont^1$ one-form $\cfo$ sufficiently close to $-dU$ in the $\Cont^1$ topology, $\dft = \cfo^\sharp$ satisfies Assumption~\ref{assump:morse-smale} and the minimizer $T_*$ in \eqref{eq:morse-rst-minimizer} is unique.
Moreover, if $\dfte$ is a smooth vector field on $M$ for each $\varepsilon > 0$ and $\dfte\to \dft$ uniformly as $\varepsilon \to 0$, then the steady-state $[\cfo]$-flux \eqref{eq:flux-def} of the diffusion with generator $\dfte + \varepsilon \Delta$ satisfies $\fluxe([\cfo]) > 0$ for all sufficiently small $\varepsilon > 0$, and
\begin{equation}\label{eq:th:qualitative-flux-formula}
\lim_{\varepsilon \to 0}(-\varepsilon \ln \fluxe([\cfo])) = h_*.
\end{equation}
\end{restatable}
\begin{Rem}\label{rem:th-qualitative-sufficiently-close}
One might ask whether the conclusion of Theorem~\ref{th:qualitative} remains true if the hypothesis that $\cfo$ is sufficiently close to $-dU$ in the $C^1$ topology is replaced with the (weaker) hypothesis that $\cfo^\sharp$ satisfies Assumption~\ref{assump:morse-smale} and $T_*$ in \eqref{eq:morse-rst-minimizer} is unique, and similarly for the conclusion of Theorem~\ref{th:qualitative-measure} below.
After developing the necessary tools, in \S \ref{sec:counterexamples} we construct counterexamples demonstrating that the answer to this question is negative for multiple reasons even if certain natural conditions are imposed on the Morse graph $\compd$ and gains $g(\slot)$. 
However, we will see in Ex.~\ref{ex:tilt-pot-circle} that the stronger hypothesis is not needed in Theorem~\ref{th:qualitative} in the special case that $\dim(M)=1$, but Ex.~\ref{ex:cex-1d-measure} in \S \ref{sec:counterexamples} shows that the stronger hypothesis is still needed for the $1$-dimensional case in Theorem~\ref{th:qualitative-measure}.
Another special case in which the stronger hypothesis is not needed in Theorem~\ref{th:qualitative} (or, trivially, in Theorem~\ref{th:qualitative-measure}) is the subject of Prop.~\ref{prop:one-minimizer}.
\end{Rem}

\begin{restatable}[]{Prop}{PropOneMinimizer}\label{prop:one-minimizer}
Assume that $\dft = \cfo^\sharp$ is dual to a closed but not exact $\Cont^1$ one-form on a closed connected Riemannian manifold $M$ and satisfies Assumption~\ref{assump:morse-smale}.
Further assume that $\dft$ has precisely one index-$0$ zero.
Then if $\dfte$ is a smooth vector field on $M$ for each $\varepsilon > 0$ and $\dfte\to \dft$ uniformly as $\varepsilon \to 0$, the steady-state $[\cfo]$-flux \eqref{eq:flux-def} of the diffusion with generator $\dfte + \varepsilon \Delta$ satisfies $\fluxe([\cfo]) > 0$ for all sufficiently small $\varepsilon > 0$, and
\begin{equation}\label{eq:prop-one-minimizer}
\lim_{\varepsilon \to 0}(-\varepsilon \ln \fluxe([\cfo])) = h_*.
\end{equation}
\end{restatable}

The following result concerning the invariant measure of the diffusion is obtained as a consequence of the methods used to prove Theorem~\ref{th:qualitative}.
For the statement, recall the definitions of $\compd=(\vtx,\Ed,\src,\tgt)$ and $h(\slot)$ of \S \ref{sec:setup-morse}, and denote by $B_r(x)$ the metric ball of radius $r \geq 0$ centered at $x\in M$.

\begin{restatable}[]{Th}{ThmQualitativeMeasure}\label{th:qualitative-measure}
Let $U\in \Cont^2(M)$ be a Morse-Smale function on a closed connected Riemannian manifold $M$, such that $U$ has a unique global minimizer and
\begin{equation}\label{eq:th-qualitative-measure-minimizer}
\textnormal{the minimizer of}\quad \min_{T\in \ST(\comp^U)}\sum_{e\in T}U(e) \quad \textnormal{is unique.}
\end{equation}
For any closed but not exact $\Cont^1$ one-form $\cfo$ sufficiently close to $-dU$ in the $\Cont^1$ topology, $\dft = \cfo^\sharp$ satisfies Assumption~\ref{assump:morse-smale} and the minimizer $T_*$ in \eqref{eq:morse-rst-minimizer} is unique.
Moreover, if $\dfte$ is a smooth vector field on $M$ for each $\varepsilon > 0$ and $\dfte\to \dft$ uniformly as $\varepsilon \to 0$, then the stationary probability density $\dme \in \Cont^\infty(M)$ of the diffusion process with generator $\dfte + \varepsilon \Delta$ satisfies the following estimates.
For any $\delta > 0$ there is $k > 0$ such that, for any $v\in \vtx$ and $\varepsilon, r\in (0,k)$,
\begin{equation}\label{eq:th-qualitative-measure-main}
e^{-\frac{1}{\varepsilon}(h(v)+\delta)} < \int_{B_r(v)}\dme(x) dx < e^{-\frac{1}{\varepsilon}(h(v)-\delta)}.
\end{equation}
Additionally, the invariant measure $\mu_\varepsilon$ with density $\dme$ converges weakly to the Dirac measure $\delta_{v_*}$ as $\varepsilon \to 0$, where $v_*$ is the root of $T_*$.  
\end{restatable}

\begin{Ex}\label{ex:nr-example-2}
Consider the negative resistance example of \S \ref{sec:nr-torus} depicted in Fig.~\ref{fig:NR-ex-1}.
As argued in \S \ref{sec:nr-torus}, $U\in \Cont^\infty(M)$ is a Morse-Smale function on the flat $2$-torus $\tor^2$.
Since $-\nabla U$ has precisely one index-$0$ zero $v_*(0)$, there is only one element in $\ST(\comp^U)$: the spanning tree with one vertex and no edges.
Thus, \eqref{eq:th-qualitative-minimizer} and \eqref{eq:th-qualitative-measure-minimizer} are satisfied, so the hypotheses of Theorem~\ref{th:qualitative} and \ref{th:qualitative-measure} are satisfied.

In particular, Theorem~\ref{th:qualitative-measure} implies that the invariant measure of the diffusion with generator $\dft + \varepsilon \Delta$ converges weakly to the Dirac measure at the index-$0$ zero of $-\nabla U$ as $\varepsilon \to 0$.

To understand the conclusions of Theorem~\ref{th:qualitative}, note that the undirected Morse graph $\comp^U$ consists of one index-$0$ zero and two edges corresponding to the unstable manifolds of the two index-$1$ saddles.
The directed graph $\compd^U$ then has the same vertex set but two directed edges for each saddle, so there are four directed edges in total: $e_1',\bar{e}_1', e_2', \bar{e}_2'$.
If $\dft$ is sufficiently close to $-\nabla U$ in the $C^1$ topology, Theorem~\ref{th:qualitative} implies that $\dft$ satisfies Assumption~\ref{assump:morse-smale}, and the implicit function and stable manifold theorems imply that $\comp^U$ perturbs to a nearby graph $\comp$ with the corresponding directed graph $\compd$ having edges $e_1,\bar{e}_1, e_2, \bar{e}_2$ with heights $h(\slot)$ nearby those of $\compd^U$.
If $e_1$ corresponds to the edge traveling through the ``blue valley'' in Fig.~\ref{fig:NR-ex-1} with positive winding number in the $x$-direction, then it is clear that $h_* = h(e_1)$.
Thus, the small-$\varepsilon$ asymptotics of the flux $\fluxe([\cfo])$ given by Theorem~\ref{th:qualitative} for this example match those calculated in \S \ref{sec:nr-torus} using Theorem~\ref{th:qualititative-special-case-intro}.

For this example, Prop.~\ref{prop:one-minimizer} allows us to deduce a conclusion stronger than that obtained from Theorem~\ref{th:qualitative}.
For any $c > 0$ such that $\dft = -\nabla U + c \partial_x$ has only hyperbolic zeros, satisfies Assumption~\ref{assump:morse-smale}, and has precisely one index-$0$ zero, the small-$\varepsilon$ asymptotics of the flux are given by \eqref{eq:prop-one-minimizer}. 
\end{Ex}

\subsection{Theorem~\ref{th:qualitative} implies Theorem~\ref{th:qualititative-special-case-intro}}\label{sec:th-2-implies-th-1}
In this section we show that the definitions of $h_*$ given in \S \ref{sec:summary-first-main-result} and \S \ref{sec:setup-morse} coincide, and we prove Theorem~\ref{th:qualititative-special-case-intro} assuming Theorem~\ref{th:qualitative}.
Throughout this section $M$ is a closed connected Riemannian manifold.
In Def.~\ref{def:gamma-e}, Rem.~\ref{rem:h-star-alt-def}, and Lem.~\ref{lem:height-gamma-h-e-star}, \ref{lem:h-star-defs-coincide}, $\dft = \cfo^\sharp$ is dual to a $\Cont^1$ closed one-form $\cfo$ on $M$ satisfying Assumption~\ref{assump:morse-smale} and for which the minimizer $T_*\in \RST(\compd)$ in \eqref{eq:morse-rst-minimizer} is unique. 

For the following definition, recall that the loop space $\Omega_{v_*} M$ defined in \S \ref{sec:summary-first-main-result} is the set of continuous paths $\gamma\colon [0,1]\to M$ with $\gamma(0)=\gamma(1)=v_*$.
\begin{Def}\label{def:gamma-e}
Given $e\in \Ed\setminus q^{-1}(q(T_*))$, we denote by $\gamma_e\in \Omega_{v_*}M$ any loop satisfying $\gamma([0,1])\subset \comp$ and constructed by first following the unique path in $q(T_*)$ from $v_*$ to $\src(e)$, then following the path in $q(e)$ to $\tgt(e)$, then following the unique path in $q(T_*)$ from $\tgt(e)$ to $v_*$. (This defines $\gamma_e$ uniquely up to reparametrization.)
\end{Def}
\begin{Rem}\label{rem:h-star-alt-def}
Given $e\in \Ed \setminus q^{-1}(q(T_*))$ with reversal $\bar{e}$, note that $\int_{\gamma_e}(-\cfo) = h(e) - h(\bar{e})$ (cf. Rem.~\ref{rem:edge-height-homology}).
\end{Rem}

Given $\gamma\in \Omega_{v_*}M$, as in \eqref{eq:height-loop} we define
\begin{equation}\label{eq:height-loop-morse}
\height(\gamma)\coloneqq \sup_{t\in [0,1]} \int_{\gamma|_{[0,t]}}(-\dft^\flat) =  \sup_{t\in [0,1]} \int_{\gamma|_{[0,t]}}(-\cfo).
\end{equation}
\begin{Lem}\label{lem:height-gamma-h-e-star}
Let $e\in \Ed\setminus q^{-1}(q(T_*))$ and let $\gamma_{e}\in \Omega_{v_*} M$ be any loop constructed as in Def.~\ref{def:gamma-e}.  
Then 
\begin{equation}\label{eq:lem:height-gamma-h-e-star}
\height(\gamma_{e}) = h(e).
\end{equation}
\end{Lem}
\begin{proof}
From Def.~\ref{def:gamma-e} and \eqref{eq:height-loop-morse} it follows that $\height(\gamma_{e}) \geq h(e)$.
Assume, to obtain a contradiction, that $\height(\gamma_{e}) > h(e)$, and recall that $\bar{e}$ denotes the reversal of $e$.
Then Def.~\ref{def:morse-heights} and \ref{def:gamma-e} imply the existence of $e_*\in T_*$ and a directed path $e_1\cdots e_k$ in $T_*$ either (i) from $\src(e)$ to $\src(e_*)$ and satisfying $h(e) < h(e_*)$, or (ii) from $\tgt(e)$ to $\src(e_*)$ and satisfying $h(\bar{e}) < h(\bar{e}_*)=h(e_*)$ (Rem.~\ref{rem:edge-height-homology}).
Let $T\subset \RST(\compd;v_*)$ be constructed from $T_*$ by removing $e_*$, reversing the edges $e_1,\ldots, e_k$, and in case (i) adding $e$ and in case (ii) adding $\bar{e}$.
In case (i) we compute 
\begin{equation*}
\begin{split}
\sum_{e\in T_*}g(e) - \sum_{e\in T}g(e) &= g(e_*) - g(e) + \sum_{i=1}^k g(e_i) - g(\bar{e}_i) = g(e_*) - g(e) + \sum_{i=1}^k h(\tgt(e_i)) - h(\src(e_i))\\
&= g(e_*) - g(e) + h(\src(e_*)) - h(\src(e)) = h(e_*) - h(e) > 0,
\end{split}
\end{equation*}
contradicting the unique minimizing property of $T_*$ (see \eqref{eq:morse-rst-minimizer}). 
Similarly, in case (ii) 
\begin{equation*}
\begin{split}
\sum_{e\in T_*}g(e) - \sum_{e\in T}g(e) &= g(e_*) - g(\bar{e}) + \sum_{i=1}^k g(e_i) - g(\bar{e}_i) = g(e_*) - g(\bar{e}) + \sum_{i=1}^k h(\tgt(e_i)) - h(\src(e_i))\\
&= g(e_*) - g(\bar{e}) + h(\src(e_*)) - h(\tgt(e)) = h(e_*) - h(\bar{e}) > 0,
\end{split}
\end{equation*}
contradicting the minimizing property of $T_*$ and completing the proof.
\end{proof}

\begin{Lem}\label{lem:h-star-defs-coincide}
The quantity $h_*$ defined in \eqref{eq:h-star-def-morse} satisfies the following:
\begin{equation}\label{eq:h-star-equality}
h_* = \inf \{\height(\gamma)\colon \gamma\in \Omega_{v_*} M \textnormal{ and } \int_\gamma \cfo > 0\}
\end{equation}
\end{Lem}
\begin{proof}
Let $e_*\in \Ed\setminus q^{-1}(q(T_*))$ satisfy $h(e_*) = h_*$, and let $\gamma_{e_*}\in \Omega_{v_*}M$ be a loop as constructed in Def.~\ref{def:gamma-e}.
Lem.~\ref{lem:height-gamma-h-e-star} implies that $\height(\gamma_{e_*}) = h_*$, and Rem.~\ref{rem:h-star-alt-def} and \eqref{eq:h-star-def-morse} imply that $\int_\gamma \cfo =  h(\bar{e}_*)-h(e_*)> 0$.
It follows that $h_*\geq R$, where $R$ is the right side of \eqref{eq:h-star-equality}.

Conversely, fix $\varepsilon > 0$ and choose $\gamma\in \Omega_{v_*}M$ so that $\height(\gamma) < R + \varepsilon$ and $\int_{\gamma}\cfo > 0$.
By transversality \cite[p.~74,~Thm~2.1]{hirsch1976differential} we may assume that the only stable manifolds of $\dft$ having nonempty intersection with the image of $\gamma$ are those corresponding to zeros of $\dft$ with index $0$ or $1$.
This and $\int_{\gamma}\cfo \neq 0$ implies that there exists $e\in \Ed$ and $t' > 0$ such that $h(e) \neq h(\bar{e})$ and $\gamma(t')$ belongs to the stable manifold of $q(e)$.
Now let $t > 0$ be the \emph{smallest} such time.
Let $\Phi\colon \R \times M \to M$ be the flow of $\dft=\cfo^\sharp$ and define $\Phi^s\coloneqq \Phi(s,\slot)$.
Using the formula for differentiation of the integral of a differential form \cite[Eq.~7.2]{flanders1973differentiation}, we compute 
\begin{equation}
\frac{d}{ds} \int_{\Phi^s\circ \gamma|_{[0,t]}} (-\cfo) = -\norm{\dft\circ \Phi^s(\gamma(t))}^2 \leq 0.
\end{equation}
Since $\Phi^s(\gamma(t))$ converges to the index-$1$ zero in $q(e)$ as $s\to \infty$, it follows that $\int_{\gamma|_{[0,t]}}(-\cfo) \geq \min(h(e),h(\bar{e}))$ (since $t > 0$ was the smallest such time).
Since $h_*\leq \min(h(e),h(\bar{e}))$ by \eqref{eq:h-star-def-morse}, it follows that
$$h_* \leq \min(h(e),h(\bar{e})) \leq \int_{\gamma|_{[0,t]}}(-\cfo) \leq \height(\gamma) < R+\varepsilon.$$
Since $\varepsilon > 0$ was arbitrary, this implies that $h_*\leq R$ and completes the proof.
\end{proof}

We now prove Theorem~\ref{th:qualititative-special-case-intro}.
For convenience we restate this theorem and Assumption~\ref{assump:morse-smale-intro}.

\AssumpMorseSmaleIntro*
\ThmQualitativeSpecial*
\begin{Rem}\label{rem:th1-proof-h-v-defs-coincide}
The proof shows that, for sufficiently small $c>0$,  $\dft$ satisfies Assumption~\ref{assump:morse-smale} and the minimizer $T_*\in \RST(\compd)$ is unique.
Moreover, when $c$ is sufficiently small, $v_*(c)$ defined in \S \ref{sec:summary-first-main-result} coincides with the root of $T_*$ and (as follows from Lem.~\ref{lem:h-star-defs-coincide}) $h_*(c)$ defined in \S \ref{sec:summary-first-main-result} coincides with $h_*$ defined in \eqref{eq:h-star-def-morse}.
\end{Rem}
\begin{proof}
We first observe that Assumption~\ref{assump:morse-smale-intro} directly implies that $U$ is Morse, $U$ has a unique global minimizer, and all pairwise intersections of (un)stable manifolds of $\nabla U$ are transverse.
Assumption~\ref{assump:morse-smale-intro} also implies \eqref{eq:th-qualitative-minimizer} since the minimal spanning tree of a weighted undirected graph with distinct edge weights is always unique \cite[Ex.~20.5]{sedgewick2002algorithms}.
Additionally, it is immediate that $\cfo=-dU + c\cft$ converges to $-dU$ in the $\Cont^1$ topology as $c\to 0$.

Hence Theorem~\ref{th:qualitative} implies that, for all sufficiently small $c > 0$, $\dft = -\nabla U + c\cft^\sharp$ satisfies Assumption~\ref{assump:morse-smale} and the minimizer $T_*\in \RST(\compd)$ of \eqref{eq:morse-rst-minimizer} is unique. 
We use the notation $T_*(c)$ to emphasize the dependence of $T_*$ on $c$.
We temporarily introduce the notations $v_*^1(c)$ for the quantity $v_*(c)$ introduced in \S \ref{sec:summary-first-main-result} and $v^2_*(c)$ for the root of $T_*(c)$.

We claim that $v^1_*(c)$ and $v^2_*(c)$ are both well-defined for $c = 0$ and $v^1_*(0) = v^2_*(0)$.
When $c = 0$, the vector field $\dft$ satisfies Assumption~\ref{assump:morse-smale} and, for any $v\in \vtx$ and $T\in \RST(\compd;v)$, 
\begin{equation}\label{eq:comparison-h-star-gain-sum}
\sum_{e\in T}g(e) = U(v) +   \sum_{e\in T}U(e) - \sum_{v'\in \vtx}U(v').
\end{equation}
Since the third term on the right of \eqref{eq:comparison-h-star-gain-sum} is independent of $T$ and the second term is unchanged if $T$ is replaced by any $T'\in q^{-1}(q(T))$, we see that $v^2_*(0)$ (the root of $T_*(0)$) must minimize $U$ when $c=0$.
Thus, $v^1_*(0) = v^2_*(0)$.
The implicit function and stable manifold theorems \cite[p.~75,~Thm~6.2]{palis1980geometric} imply that the $1$-dimensional unstable manifolds of $\dft$, gains $g(\slot)$, and heights $h(\slot)$ depend continuously on $c$; since the set of rooted spanning trees in $\compd$ is finite and the minimizer $v^1_*(0)$ of $U$ is unique (by Assumption~\ref{assump:morse-smale-intro}), the root $v^2_*(c)$ of $T_*(c)$ must also coincide with $v^1_*(c)$ for $c> 0$ sufficiently small.

Letting $\cfo = -dU + c\beta$, linearity of $\fluxec(\slot)$ and Prop.~\ref{prop:flux-positive-c1f-case} imply that $\fluxec([\cfo])=c \fluxec([\cft]).$
Thus, Theorem~\ref{th:qualitative} implies that, for all sufficiently small $c > 0$, $\fluxec([\cft]) > 0$ for all sufficiently small $\varepsilon > 0$ and
$$\lim_{\varepsilon\to 0}(-\varepsilon \ln \fluxec([\cft])) = \lim_{\varepsilon\to 0}(-\varepsilon \ln \fluxec([\cfo])) = h_*,$$
where $h_*$ is as defined in \eqref{eq:h-star-def-morse}.
Since $v^1_*(c) = v^2_*(c)$ when $c$ is sufficiently small, Lem.~\ref{lem:h-star-defs-coincide} implies that this $h_*$ coincides with $h_*(c)$ defined in \eqref{eq:h-star-def}.
This implies \eqref{eq:qualititative-special-case intro} and completes the proof.
\end{proof}

\section{Drifts whose chain recurrent sets consist of a finite number of hyperbolic zeros }\label{sec:drift-finite-hyperbolic-chain}
Theorem \ref{th:qualitative} is derived in part from a more general result (Theorem~\ref{th:flux-manifold-mc-CRST-ld}) described in this section.
\subsection{Setup}\label{sec:setup-sec:drift-finite-hyperbolic-chain}
We first discuss preliminaries: we review the definitions of chain recurrence \cite{conley1978isolated} and cycle-rooted spanning trees \cite{pitman2018tree}, and we introduce a path-homotopical refinement of the Freidlin-Wentzell quasipotential \cite{freidlin2012random}.
\subsubsection{Chain recurrence}
Let $\Phi\colon \R\times M\to M$ be a continuous flow.
Defining $\Phi^t\coloneqq \Phi(t,\slot)$, that $\Phi$ is a flow means that $\Phi^0 = \id_M$ and $\Phi^{t+s}=\Phi^t\circ \Phi^s$ for all $s,t\in \R$.
The following definition is standard  \cite{conley1978isolated, hurley1995chain, robinson1999dynamical,alongi2007recurrence}.
\begin{Def}\label{def:chain-recurrent}
Given $\varepsilon, T> 0$ and $x,y\in M$, an \concept{$(\varepsilon,T)$-chain} from $x$ to $y$ is a tuple $$(x=x_0, x_1,\ldots, x_N =y; t_1,\ldots, t_N)$$ such that $t_i\geq T$ and  $\dist{\Phi^{t_i}(x_{i-1})}{x_i}\leq \varepsilon$ for all $i\geq 1$.
A point $x\in M$ is \concept{chain recurrent} if for all $\varepsilon, T>0$ there is an $(\varepsilon, T)$-chain from $x$ to itself.
The \concept{chain recurrent set} $R(\Phi)\subset M$ is the set of all chain recurrent points. 
\end{Def} 
If $\Phi$ is generated by a $\Cont^1$ vector field $\dft$ on $M$, we adopt the abuse of notation $R(\dft)\coloneqq R(\Phi)$ and speak of the chain recurrent set of $\dft$.
\subsubsection{Action functional and path-homotopical quasipotential}
Let $\dft$ be a continuous vector field on a Riemannian manifold $M$ (not necessarily compact).
Following Freidlin and Wentzell \cite[Ch.~6.1]{freidlin2012random}, we define the associated \concept{action functional} on continuous paths $\vp\in C([T_1,T_2],M)$ via 
\begin{equation}\label{eq:action-functional}
\af_{T_1,T_2}(\vp) \coloneqq \frac{1}{4}\int_{T_1}^{T_2} \norm{\dot{\vp}-\dft(\vp)}^2\, dt
\end{equation}
if $\vp$ is absolutely continuous and set $\af_{T_1,T_2}(\vp)\coloneqq +\infty$ otherwise.\footnote{The definition of $\af_{T_1,T_2}(\slot)$ in \cite[Ch.~6.1]{freidlin2012random} has a factor of $1/2$ in front of the integral instead of the factor of $1/4$ appearing in \eqref{eq:action-functional}. 
The source of the difference is that we use the convention $L_\varepsilon = \sum_i b^{i}_{\varepsilon}(x)\frac{\partial}{\partial x^i} + \varepsilon \sum_{i,j}a^{ij}(x)\frac{\partial^2}{\partial x^i \partial x^j}$ for the infinitesimal generator of a diffusion (see \eqref{eq:generator}), while the convention $L_\varepsilon = \sum_i b^{i}_{\varepsilon}(x)\frac{\partial}{\partial x^i} + \frac{\varepsilon}{2} \sum_{i,j}a^{ij}(x)\frac{\partial^2}{\partial x^i \partial x^j}$ used in \cite{freidlin2012random} includes a factor of $1/2$ on the second term.
Thus, the Riemannian metric $(a^{-1})_{ij}$ defined according to the convention of \cite{freidlin2012random} and used to define $\af_{T_1,T_2}(\slot)$ corresponds to $(1/2)(a^{-1})_{ij}$ according to our convention, and this extra factor of $1/2$ leads to the factor of $1/4$ in \eqref{eq:action-functional}.}
We also use the notation $\af_{T}(\slot)\coloneqq \af_{0,T}(\slot)$.
We use the notation $\af(\slot)$ instead $\af_{T_1,T_2}(\slot)$ if the domain $[T_1,T_2]$ of $\vp$ is clear from context or if we do not wish to emphasize the domain.\footnote{There is a related functional which is invariant under reparametrization of the path $\vp$, making $T_1$ and $T_2$ immaterial \cite{vanden2008geometric,heymann2015minimum}. This functional has some properties which are better and some properties which are worse than the functionals $\af_{T_1,T_2}(\slot)$ we consider, and $\af_{T_1,T_2}(\slot)$ seems more convenient for our purposes.}

\begin{Rem}[{\cite[p.~12]{ventsel1970small}}]\label{rem:finite-action-iff-L2}
Given $T_1\leq T_2$, define the $L^2$ inner product $\ip{\gamma}{\psi}_2\coloneqq \int_{T_1}^{T_2} \ip{\gamma}{\psi}dt$ and norm $\norm{\gamma}_2\coloneqq \sqrt{\ip{\gamma}{\gamma}}_2$.
Since
\begin{equation}\label{eq:finite-af-iffs-l2}
\begin{split}
4\af_{T_1,T_2}(\vp) &= \norm{\dot \vp}_2^2 - 2\ip{\dot{\vp}}{\dft(\vp)}_2 +  \norm{\dft(\vp)}_2^2\\
&\geq \norm{\dot \vp}_2^2 - 2\norm{\dot \vp}_2\norm{\dft(\vp)}_2 + \norm{\dft(\vp)}_2^2 \\
&= (\norm{\dot \vp}_2-\norm{\dft(\vp)}_2)^2,
\end{split}
\end{equation}
given $\vp\in \Cont([T_1,T_2],M)$ it follows that $\af_{T_1,T_2}(\vp) < +\infty$ if and only if $\vp$ is absolutely continuous and $\norm{\dot{\vp}}_2< + \infty$.
\end{Rem}

Freidlin and Wentzell use the action functional to define an asymmetric \emph{quasipotential} on pairs of points $(x,y)$ to be the infimal action over all continuous paths joining $x$ to $y$.
However, this definition does not account for the way that a path starting from $x$ wraps around $M$ before reaching $y$, and we will require some amount of such information for studying flux.
Hence we will introduce a refinement of the quasipotential which takes into account homotopy classes of paths.
(Mere homological information would suffice, but accounting for homotopical information presents no additional difficulties and leads to stronger intermediate results.) 

Two continuous paths $\vp_1,\vp_2 \colon [0,1] \to M$ with common endpoints are path-homotopic (as opposed to freely homotopic) if $\vp_1$ is homotopic to $\vp_2$ via a homotopy that fixes the endpoints \cite[p.~187]{lee2011topological}.
A path homotopy class is a maximal collection of path-homotopic paths.
Let $\Pi(M)$ denote the set of path homotopy classes and $\src,\tgt\colon \Pi(M)\to M$ be the maps sending classes to their source and target points.
If $e_1,e_2\in \Pi(M)$ and $\tgt(e_1)=\src(e_2)$, there is a well-defined path homotopy class $e_1e_2\in \Pi(M)$ with $\src(e_1e_2)=\src(e_1)$ and $\tgt(e_1e_2)=\tgt(e_2)$ defined to be the class of the concatenation of any pair of paths representing $e_1$ and $e_2$.
This concatenation operation makes $\Pi(M)$ into a groupoid called the \concept{fundamental groupoid} \cite[p.~15]{may1999concise}.
Given $e\in \Pi(M)$,  we denote by $\Cont_e([T_1,T_2],M)\subset \Cont([T_1,T_2],M)$ the subset of paths $\vp$ with $(t\mapsto \vp(t(T_2-T_1)+T_1))\in e$.
If $\vp\in \Cont_e([T_1,T_2],M)$ for some $T_1 \leq T_2$, we also write $[\vp] = e$.
\begin{Def}\label{def:quasipotential}
Given $e\in \Pi(M)$, we define the \concept{(path-homotopical) quasipotential} $\qpd(e)\in [0,+\infty)$ via
\begin{equation}
\qpd(e)\coloneqq \inf \{\af(\vp)\colon [\vp]=e\}.
\end{equation}
\end{Def}

Assume now that $\dft$ is a $\Cont^1$ vector field on $M$ such that $\dft^{-1}(0)$ consists only of hyperbolic zeros.
Recall that the (Morse) index of $z\in \dft^{-1}(0)$ is the dimension of the unstable manifold $\Wu(z)$ for the flow of $\dft$.
Denote by $\vtx\subset \dft^{-1}(0)$ those zeros with index $0$.
\begin{Def}\label{def:quasipotential-restricted}
We define the \concept{restricted quasipotential} $\tqp_{\dft}(e)\in [0,+\infty]$ by considering only those paths meeting $\vtx$ only at the endpoints:
\begin{equation}\label{eq:tqpd-def}
\tqpd(e)\coloneqq \inf\{\af(\vp)\colon [\vp]=e \,\,\textnormal{ and }\,\, (\vp|_{\interior(\dom(\vp))})^{-1}(\vtx)=\varnothing\}.
\end{equation}
Here $\dom(\vp)$ is the domain of $\vp$ and $\interior(\slot)$ denotes the topological interior.
For example, if $\vp\in \Cont([T_1,T_2],M)$ then $\interior(\dom(\vp))=(T_1,T_2)$.
\end{Def}

It is immediate from the definitions that $\tqpd \geq \qpd$.
We now show that equality holds when $\dim(M)\geq 2$.
\begin{Lem}\label{lem:tqp-equals-qp-transversality}
Let $\dft$ be a $\Cont^1$ vector field on a Riemannian manifold $M$ (not necessarily compact) such that $\dft^{-1}(0)$ consists only of hyperbolic zeros.
Assume that $\dim(M) \geq 2$. 
Then $\tqpd = \qpd$, and for every $e\in \Pi(M)$ and $\varepsilon > 0$ there is a smooth path $\vp$ satisfying $\af(\vp) < \qpd(e) + \varepsilon$.
\end{Lem}
\begin{proof}
Fix any $e\in \Pi(M)$.
It suffices to show that, for any $\vp\in \Cont_{e}([T_1,T_2],M)$ satisfying $\af(\vp)<\infty$ and $\varepsilon > 0$, there is a path $\psi\in \Cont_{e}([T_1,T_2],M)$ with $\psi|_{(T_1,T_2)}$ disjoint from $\dft^{-1}(0)$ and satisfying $\af(\psi)< \af(\vp) + \varepsilon$.
First note Rem.~\ref{rem:finite-action-iff-L2} implies that $\vp$ is absolutely continuous with $\norm{\dot{\vp}}_2<+\infty$, and for such a $\vp$ and any $\delta > 0$ there is a smooth $\gamma\in \Cont_e([T_1,T_2],M)$ such that $\gamma$ is uniformly $\delta$-close to $\vp$ and  $\norm{\dot{\gamma}-\dot{\vp}}_{2}<\delta$ (cf. \cite[p.~52, Lem.~5.1]{showalter1994hilbert}). 
Thus, the first line of \eqref{eq:finite-af-iffs-l2} implies that $\af(\gamma)<\af(\vp) + \varepsilon/2$ if $\delta$ is small enough.
Second countability of $M$ and hyperbolicity of $\dft^{-1}(0)$ imply that the latter set is countable, so the transversality theorem \cite[p.~74,~Thm~2.1]{hirsch1976differential} implies that we may uniformly $\Cont^1$-approximate $\gamma$ by a smooth path $\psi\in \Cont_e([T_1,T_2],M)$ such that $\psi|_{(T_1,T_2)}$ is disjoint from $\dft^{-1}(0)$ and  $\af(\psi)<\af(\gamma)+\varepsilon/2 < \af(\vp)+\varepsilon$.
\end{proof}
Given $x,y\in M$, we also define 
\begin{equation}\label{eq:unrefined-quasipot}
\qpd(x,y)\coloneqq \inf_{e\in \src^{-1}(x)\cap \tgt^{-1}(y)} \qpd(e)\quad \textnormal{and} \quad \tqpd(x,y)\coloneqq \inf_{e\in \src^{-1}(x)\cap \tgt^{-1}(y)} \tqpd(e). 
\end{equation}
The quantity on the left coincides with the quasipotential considered by Freidlin and Wentzell \cite[Ch.~6.1,~6.2]{freidlin2012random}.
However, the quantity on the right is \emph{not} the same as the restricted quasipotential considered in \cite[Ch.~6]{freidlin2012random} since our definition of $\tqpd$ involves only the \emph{stable} zeros of $\dft$ (those with index $0$).
But if $\dim(M)\geq 2$ so that Lem.~\ref{lem:tqp-equals-qp-transversality} applies, these quantities are the same.

\subsubsection{Cycle-rooted spanning trees}
Let $\gph = (V,E,\src,\tgt)$ be a directed graph; as in \S \ref{sec:setup-morse} we allow loop edges at a single vertex and multiple edges between the same pair of vertices.
We denote by $\RST(\gph;v)$ the set of rooted spanning trees with edges directed toward a root $v\in \vtx$ and set $\RST(\gph)\coloneqq \bigcup_{v\in \vtx} \RST(\gph;v)$. 

We say that a subset $\edg'\subset \edg$ is a \concept{cycle-rooted spanning tree} if there exists $v\in \vtx$ and $\edg''\in \RST(\gph;v)$ such that $\edg'$ is the union of $\edg''$ with a single edge in $\src^{-1}(v)$ \cite{pitman2018tree}.
See Fig.~\ref{fig:crst}.
Such an $\edg'$ contains exactly one directed cycle $\cC\subset E$, and this cycle contains $v$.
We denote the set of all cycle-rooted spanning trees containing the directed cycle $\cC\subset \edg$ by $\CRST(\gph;\cC)$, and we define $\CRST(\gph)\coloneqq \bigcup_{\cC}\CRST(\gph;\cC)$ to be the set of all cycle-rooted spanning trees.
We define the map $\cycle\colon \CRST(\gph)\to 2^\edg$ be the map which sends each $\edg'\in \CRST(\gph)$ to its unique directed cycle, so that $\cycle(\edg') = C$ if $\edg' \in \CRST(\gph;\cC)$. 

\begin{figure}
	\centering
	\def\svgwidth{0.5\linewidth}
\begingroup%
  \makeatletter%
  \providecommand\color[2][]{%
    \errmessage{(Inkscape) Color is used for the text in Inkscape, but the package 'color.sty' is not loaded}%
    \renewcommand\color[2][]{}%
  }%
  \providecommand\transparent[1]{%
    \errmessage{(Inkscape) Transparency is used (non-zero) for the text in Inkscape, but the package 'transparent.sty' is not loaded}%
    \renewcommand\transparent[1]{}%
  }%
  \providecommand\rotatebox[2]{#2}%
  \newcommand*\fsize{\dimexpr\f@size pt\relax}%
  \newcommand*\lineheight[1]{\fontsize{\fsize}{#1\fsize}\selectfont}%
  \ifx\svgwidth\undefined%
    \setlength{\unitlength}{375.13700855bp}%
    \ifx\svgscale\undefined%
      \relax%
    \else%
      \setlength{\unitlength}{\unitlength * \real{\svgscale}}%
    \fi%
  \else%
    \setlength{\unitlength}{\svgwidth}%
  \fi%
  \global\let\svgwidth\undefined%
  \global\let\svgscale\undefined%
  \makeatother%
  \begin{picture}(1,0.64013988)%
    \lineheight{1}%
    \setlength\tabcolsep{0pt}%
    \put(0,0){\includegraphics[width=\unitlength,page=1]{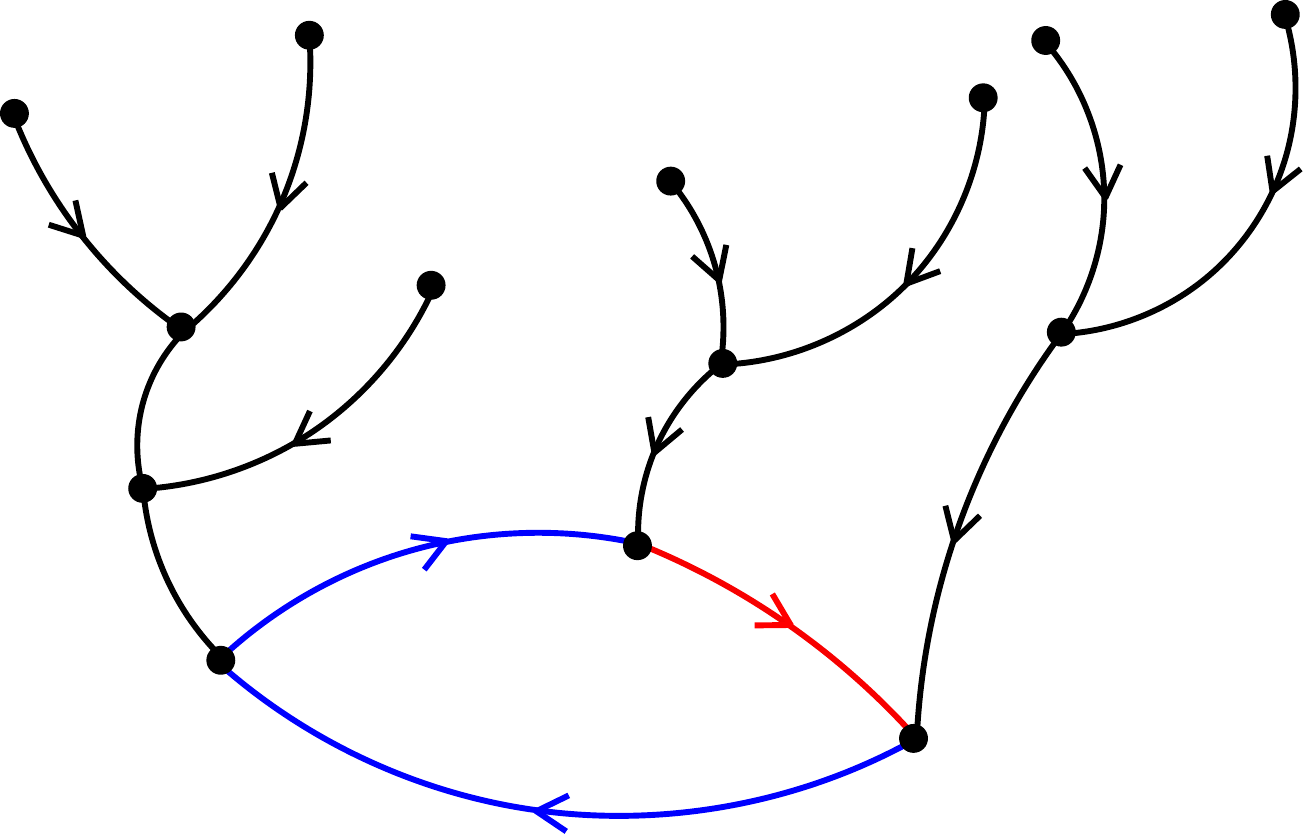}}%
    \put(0.46085646,0.18005489){\color[rgb]{0,0,0}\makebox(0,0)[lt]{\lineheight{1.25}\smash{\begin{tabular}[t]{l}$v$\end{tabular}}}}%
    \put(0,0){\includegraphics[width=\unitlength,page=2]{crst.pdf}}%
  \end{picture}%
\endgroup%

	\caption{The black dots are the vertex set of a directed graph $\gph = (\vtx,\edg,\src,\tgt)$.
	The union $\edg'\in \RST(\gph;v)$ of the black and blue edges is a rooted spanning tree with root $v\in \vtx$. 
    The union of the blue and red edges is a directed cycle $\cC  = \cycle(\edg') \subset \edg$, and the union of all edges shown (black, blue, and red) is a cycle-rooted spanning tree $\edg''\in \CRST(\gph;\cC)$.
    All vertices in $\vtx$ are shown, but edges not in $\edg''$ are not.}\label{fig:crst}
\end{figure}

\subsection{Main result}\label{sec:main-result-finite-hyp-chain-rec}
\begin{Def}\label{def:path-homotopical-graph}
Let $\dft$ be a $\Cont^1$ vector field on a closed manifold $M$ whose chain recurrent set $R(\dft)$ consists of a finite number of hyperbolic zeros.
(In particular, this implies that $R(\dft) = \dft^{-1}(0)$ so that all zeros of $\dft$ are hyperbolic, and that all trajectories of $\dft$ converge to $\dft^{-1}(0)$ in both forward and backward time.)
Define $\vtx \subset \dft^{-1}(0)$ to be those zeros with index $0$, and define $\Eh\subset \Pi(M)$ to be those path homotopy classes with source and target points in $\vtx$.
We define the \concept{path-homotopical graph} to be the directed graph $\gh\coloneqq (\vtx,\Eh,\src,\tgt)$.
\end{Def}

Given a closed one-form $\cfo$ on  $M$ and $e\in \Pi(M)$, note that the line integral $\int_e \cfo\coloneqq \int_{\vp}\cfo$ is well-defined independent of the choice of continuous path $\vp$ satisfying $[\vp]=e$.
We now state our main result for the present level of generality. 
\begin{restatable}[]{Th}{ThmFWGeneral}\label{th:flux-manifold-mc-CRST-ld}
	Let $\cfo$ be a closed one-form on a closed connected Riemannian manifold $M$, $\dft$ be a $\Cont^1$ vector field on $M$ whose chain recurrent set consists of a finite number of hyperbolic zeros, and $\dfte$ be a smooth vector field on $M$ for each $\varepsilon > 0$ with $\dfte\to \dft$ uniformly as $\varepsilon \to 0$.
	Let $\gh = (\vtx,\Eh,\src,\tgt)$ be as in Def.~\ref{def:path-homotopical-graph}.
	Given $e\in \Eh$ and $\edg\subset \Eh$ with $\#(\edg)<\infty$, define $$\cfo(e)\coloneqq \int_e \cfo \quad \textnormal{and} \quad \cfo(\edg)\coloneqq \sum_{e\in \edg}\cfo(e).$$ 
	Then the minima 
	\begin{equation*}
	\min_{\substack{\edg\in \CRST(\gh)\\\cfo(\cycle(\edg))> 0}}\,\,\,\sum_{e\in \edg} \tqpd(e)\quad \textnormal{ and } \quad \min_{\substack{\edg\in \CRST(\gh)\\\cfo(\cycle(\edg))<0}}\,\,\,\sum_{e\in \edg} \tqpd(e)\
	\end{equation*}
exist.
Assume they satisfy the inequality 
	\begin{equation}\label{eq:min-rst-assumption}
	 \min_{\substack{\edg\in \CRST(\gh)\\\cfo(\cycle(\edg))> 0}}\,\,\,\sum_{e\in \edg} \tqpd(e) < \min_{\substack{\edg\in \CRST(\gh)\\\cfo(\cycle(\edg))<0}}\,\,\,\sum_{e\in \edg} \tqpd(e).
	\end{equation}
	Then the steady-state $[\cfo]$-flux \eqref{eq:flux-def} of the diffusion with generator $\dfte + \varepsilon \Delta$ satisfies $\fluxe([\cfo]) > 0$ for all sufficiently small $\varepsilon > 0$, and
	\begin{equation}\label{eq:flux-sde-mc-ld-expression}
	\lim_{\varepsilon\to 0}(-\varepsilon \ln \fluxe([\cfo])) = \left(\min_{\substack{\edg\in \CRST(\gh)\\\cfo(\cycle(\edg))> 0}}\,\,\,\sum_{e\in \edg} \tqpd(e)\right) - \left(\min_{\substack{\edg \in \RST(\gh)}}\sum_{e\in \edg}\tqpd(e)\right).
	\end{equation}
\end{restatable}

We conclude this section with an example (Ex.~\ref{ex:tilt-pot-circle}) which illustrates Theorem~\ref{th:flux-manifold-mc-CRST-ld} in the special case $\dim(M)=1$, proves (a stronger version of) Theorem~\ref{th:qualitative} in this special case, and compares with a known result from the literature.
The following two lemmas will be used in this example and in the sequel.
\begin{Lem}[cf. {\cite[p.~100]{freidlin2012random}}]\label{lem:S-alt-expression}
Let $\dft$ be a continuous vector field on a Riemannian manifold $M$. 
Let $\vp \in \Cont([T_1,T_2],M)$ satisfy $\af(\vp)<\infty$. 
Then
\begin{equation}
\af(\vp)=\frac{1}{4}\int_{T_1}^{T_2} \norm{\dot{\vp} + \dft(\vp)}^2dt -  \int_{T_1}^{T_2} \ip{\dot{\vp}}{\dft(\vp)} dt\label{eq:S-alt-expression}.
\end{equation}
\end{Lem}
\begin{proof}
Let $\bw$ be an arbitrary continuous vector field on $M$.
By expanding both sides below, we see that  
$$\norm{\dot{\vp}-\dft(\vp)}^2  = \norm{\dot{\vp}-\bw(\vp)}^2 + \norm{\dft(\vp)}^2-\norm{\bw(\vp)}^2 + 2\ip{\dot{\vp}}{\bw(\vp)-\dft(\vp)}.$$
Taking $\bw = -\dft$, dividing by $4$, and integrating yields \eqref{eq:S-alt-expression}.
\end{proof}

\begin{Lem}\label{lem:S-lower-bound}
Let $\dft$ be a continuous vector field on a Riemannian manifold $M$ such that $\dft = \cfo^\sharp$ is the metric dual of a continuous closed one-form.
Let $\vp\colon [T_1,T_2]\to M$ be continuous.
Then
$$\af(\vp)\geq \int_\vp (-\cfo)$$
with equality if and only if $\vp$ is a segment of an integral curve of $-\dft$.
\end{Lem}
\begin{proof}
This is immediate from Lem.~\ref{lem:S-alt-expression} and the fact that, since $\bv =\cfo^\sharp$, if $\af(\vp)<\infty$ then $$-\int_{T_1}^{T_2}\ip{\dot{\vp}}{\bv(\vp)}dt=\int_{\vp}(-\cfo).$$
\end{proof}

\begin{Ex}\label{ex:tilt-pot-circle}
\begin{figure}
	\centering
	\def\svgwidth{0.8\columnwidth}
	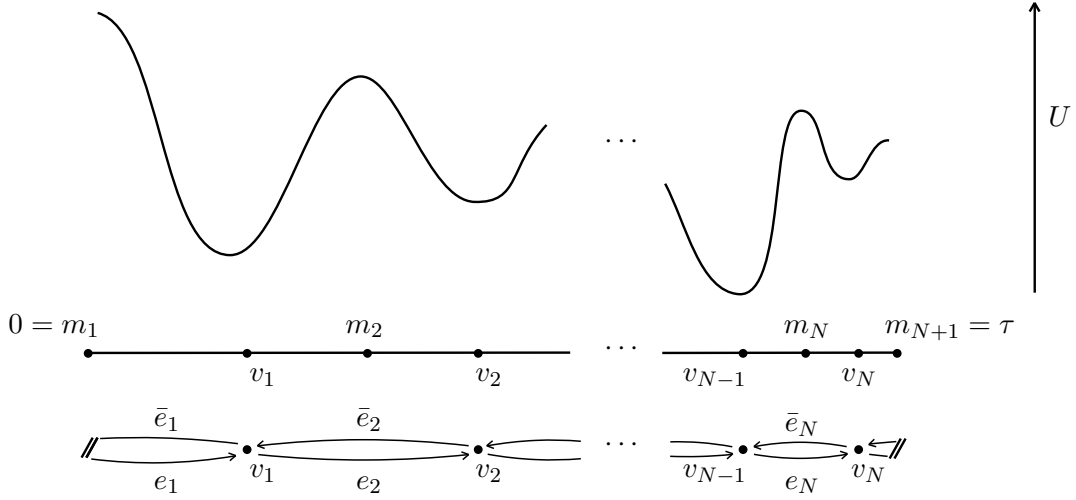  
	\caption{An illustration of Ex.~\ref{ex:tilt-pot-circle}.}\label{fig:tilt-pot}
\end{figure}
Let $\tau>0$, let $M$ be the circle $\sph^1$ viewed as $[0,\tau]$ with $0$ and $\tau$ identified, let $G$ be any Riemannian metric on $M = \sph^1$, and let $\dft$ be a $\Cont^1$ vector field on $M$ whose chain recurrent set consists of a finite number of hyperbolic zeros.
We may and do conflate $G$ with a smooth $\tau$-periodic function $\R\to (0,\infty)$ and $\dft$ with a  $\Cont^1$ $\tau$-periodic function $\dft\colon \R\to \R$.
Let $U\colon \R\to \R$ be any antiderivative of $-g\dft = -\dft^\flat$ (so that $U' = -g\dft$ and $\dft = -(1/g)U' = -(dU)^\sharp$), and note that $U$ is a $\Cont^2$ Morse function on $\R$.
After a translation of coordinates if necessary, we may assume that $U'(0) = U'(\tau) = 0$.

The Morse function $U|_{[0,\tau]}$ has some finite number $N \geq 1$ of local minimizers and $N+1$ local maximizers.
We order the local minimizers $v_1, v_2,\ldots, v_N$ and maximizers $0 = m_1, m_2,\ldots, m_{N+1} = \tau$ so that
$$m_1 < v_1 < m_2 < v_2 < \cdots < m_N < v_N < m_{N+1}.$$
See Fig.~\ref{fig:tilt-pot}.
The directed graph $\gh = (\vtx,\Eh,\src,\tgt)$ is defined by 
\begin{alignat*}{1}
\vtx&\coloneqq \{v_1,v_2,\ldots, v_N\} \qquad  \Eh\coloneqq  \{e_1,e_2,\ldots, e_N\} \cup \{\bar{e}_1,\bar{e}_2,\ldots, \bar{e}_N\} \cup \edg',
\end{alignat*} 
where the edges $e_i$ go to the right, the reversals $\bar{e}_i$ go to the left, and $\tgt$ and $\src$ respectively send the arrows $e_i$ and $\bar{e}_i$ to their tips and tails; $\edg'\subset \Eh$ is defined to be the infinite set of path homotopy classes $e$ that do not contain a path $\vp$ satisfying  $(\vp|_{\interior(\dom(\vp))})^{-1}(\vtx)=\varnothing$, and thus satisfy $\tqpd(e) = +\infty$ (cf. \eqref{eq:tqpd-def}).

For any $i$ we may choose a path $\vp\in \Cont_{e_i}([0,T_1],M)$ to first go from $\src(e_i)$ to $m_i$ while approximating an integral curve of $-\dft$, then go from $m_i$ to $\tgt(e_i)$ while approximating an integral curve of $\dft$.
The approximations can be made arbitrarily good by taking $T_1$ sufficiently large (cf. \cite[p.~143, Lem.~1.1]{freidlin2012random}), so \eqref{eq:action-functional} and Lem.~\ref{lem:S-alt-expression}, \ref{lem:S-lower-bound} imply that
\begin{equation}\label{eq:ex-1d-case-tqp}
\begin{split}
\forall i\in \{2,\ldots, N\}\colon \tqpd(e_i) &= U(m_i) - U(\src(e_i)) \quad \textnormal{and} \quad \tqpd(\bar{e}_i) = U(m_i) - U(\tgt(e_i))
\end{split}
\end{equation}
and 
\begin{equation}\label{eq:ex-1d-case-tqp-2}
\tqpd(e_1)  = U(\tau)-U(v_N), \qquad \quad\,\, \tqpd(\bar{e}_1) = U(0)-U(v_1),
\end{equation}
where the expressions with the reversals $\bar{e}_i$ follow from the same argument described above for the $e_i$.
(For more details, this also follows from Lem.~\ref{lem:qp-leq-g-W}.)

Let $\cfo = a(x)dx$ be any $\Cont^1$ closed one-form on $M = \sph^1$ which is closed but not exact.
Then there are only two cycle-rooted spanning trees $\edg,\bar{\edg}\in \CRST(\gh)$ for which $\cfo(\cycle(\slot))\neq 0$ and $\tqpd(\slot)$ has finite values on all edges: $\edg = \{e_1,\ldots, e_N\}$ and $\bar{\edg} = \{\bar{e}_1,\ldots, \bar{e}_N\}$.
(Note that $\cycle(\edg) = \edg$ and $\cycle(\bar{\edg}) = \bar{\edg}$.) 
From \eqref{eq:ex-1d-case-tqp} and \eqref{eq:ex-1d-case-tqp-2} we see that
\begin{equation}\label{eq:ex-1d-case-tqpd-difference}
\sum_{e\in \bar{\edg}} \tqpd(e) - \sum_{e\in \edg}\tqpd(e) = U(0) - U(\tau),
\end{equation}
so the assumption \eqref{eq:min-rst-assumption} of Theorem~\ref{th:flux-manifold-mc-CRST-ld} is satisfied if and only if
$U(\tau)\neq U(0)$ and $\textnormal{sign}(\int_0^\tau a(x)dx) = \textnormal{sign}(U(0)-U(\tau))$.
We henceforth assume this is the case.
By a reflection of $\R$ if necessary, we may and do henceforth assume that
\begin{equation}\label{eq:tilt-pot-condition}
\int_0^\tau a(x)dx > 0 \qquad \textnormal{and} \qquad U(0) > U(\tau).
\end{equation}
Theorem~\ref{th:flux-manifold-mc-CRST-ld} then implies that, if $\dfte$ is a smooth vector field on $M$ for each $\varepsilon > 0$ and $\dfte \to \dft$ uniformly as $\varepsilon \to 0$, the diffusion on $M$ with generator $\dfte + \varepsilon \Delta$ satisfies $\fluxe([\cfo]) > 0$ for all sufficiently small $\varepsilon > 0$ and
\begin{align}
\lim_{\varepsilon\to 0}(-\varepsilon \ln \fluxe([\cfo])) &= \sum_{e\in \edg} \tqpd(e) - \left(\min_{\substack{T \in \RST(\gh)}}\sum_{e\in T}\tqpd(e)\right) \nonumber\\
&= \max_{T\in \RST(\gh)} \left(\sum_{e\in \edg} \tqpd(e) - \sum_{e\in T}\tqpd(e)\right). \nonumber
\end{align}
If $T_*\in \RST(\gh)$ is any minimizer, $v_*$ is the root of $T_*$, and $m_* \in [v_*, v_*+\tau]$ corresponds to the local maximizer of $U$ in the unique edge in $\edg\setminus T_*$, then all of the following expressions follow from the above and \eqref{eq:ex-1d-case-tqp}, \eqref{eq:ex-1d-case-tqp-2}:
\begin{align}
\lim_{\varepsilon\to 0}(-\varepsilon \ln \fluxe([\cfo])) &= U(m_*)-U(v_*) \label{eq:ex-1d-case-flux-ld-0}\\
&= \max_{y\in [v_*,v_*+\tau]}U(y)-U(v_*) \label{eq:ex-1d-case-flux-ld-1}\\
&= \max_{\substack{x\in [0,\tau]\\y\in [x,x+\tau]}} \label{eq:ex-1d-case-flux-ld-2} U(y)-U(x).
\end{align}
We see that the result does not depend on the particular choice of closed one-form $\cfo = a(x)dx$, or even on the specific cohomology class of $\cfo$, as long as \eqref{eq:tilt-pot-condition} is satisfied.

In particular, the one-form $-U'(x)dx$ is closed but not exact since $U(0)>U(\tau)$, so we may take $\cfo = -U'(x)dx$.
In this case, $h_*$ as defined in \eqref{eq:h-star-def-morse} coincides with the right side of \eqref{eq:ex-1d-case-flux-ld-0}.
This proves Theorem~\ref{th:qualitative} in the case $\dim(M) = 1$.
In fact, it proves something stronger: the ``sufficiently close'' hypothesis of Theorem~\ref{th:qualitative} is not needed in this $1$-dimensional case (cf. Rem.~\ref{rem:th-qualitative-sufficiently-close}) as long as $\cfo^\sharp$ satisfies Assumption~\ref{assump:morse-smale} (which, in the case $\dim(M)=1$, is equivalent to our assumption that the chain recurrent set of $\dft$ consists of a finite number of hyperbolic zeros), and the hypothesis that the minimizer $T_*$ in \eqref{eq:morse-rst-minimizer} is unique is also not needed. 

Finally, we show that the above expressions for the large deviations of $\fluxe([\cfo])$ are consistent with a known result.
Consider the special case that $\dfte \equiv \dft = -U'$, the Riemannian metric $G$ is the standard Euclidean one, and $\int_{\sph^1}\cfo = \int_0^\tau a(x)dx = \tau$ (e.g., $\cfo = dx$).
The following result has appeared in the literature \cite[Eq.~18,~20]{reimann2002diffusion}:
\begin{equation}\label{eq:ex-1d-case-flux-analytical}
\fluxe([\cfo]) = \frac{\tau \varepsilon  (1-e^{-\frac{c}{\varepsilon}})}{\int_0^{\tau} \int_{x}^{\tau+x} e^{\frac{1}{\varepsilon}(U(y)-U(x))}\, dy\, dx }, 
\end{equation}
where $c\coloneqq U(0) - U(\tau) > 0$.
Eq.~\eqref{eq:ex-1d-case-flux-analytical} can be obtained by solving the stationary Fokker-Planck equation \eqref{eq:fokker-planck} by quadrature since \eqref{eq:fokker-planck} amounts to an ordinary differential equation with periodic boundary conditions in the case $M = \sph^1$.
Computing $\lim_{\varepsilon \to 0}(-\varepsilon\ln(\slot))$ of the right side of \eqref{eq:ex-1d-case-flux-analytical} using Laplace's method \cite[pp.~55--56]{freidlin2012random} yields \eqref{eq:ex-1d-case-flux-ld-2}, as claimed.
\end{Ex}

\section{Proof of Theorem~\ref{th:flux-manifold-mc-CRST-ld}}\label{sec:proof-th-flux-manifold-mc-CRST-ld}
In this section we prove Theorem~\ref{th:flux-manifold-mc-CRST-ld} in several steps.
We first establish some results concerning the quasipotential in \S \ref{sec:quasipotential}.
In \S \ref{sec:markov-chain-flux-expression} we define a discrete-time Markov chain intended to approximate the diffusion process, and we derive a Markov chain expression for the steady-state flux.
We then derive estimates on transition times and probabilities related to this approximating Markov chain in \S \ref{sec:transition-times-probabilities-estimates}.
Using these results, we complete the proof of Theorem~\ref{th:flux-manifold-mc-CRST-ld} in \S \ref{sec:finishing-the-proof-crst}.
\subsection{Quasipotential results}\label{sec:quasipotential}
Throughout this section $M$ is a closed Riemannian manifold.
The following definitions are used to formulate Prop.~\ref{prop:qp-int-conditions-iff-connecting-piecewise-orbit}.
\begin{Def}[omega-limit and alpha-limit sets {\cite[p.~29]{conley1978isolated}}]\label{def:omega-limit}
Let $\Phi\colon \R\times M\to M$ be the flow of a $\Cont^1$ vector field $\dft$.
Given $x\in M$, we define the \concept{$\omega$-limit set} and \concept{$\omega^*$-limit set} via  $$\omega(x)\coloneqq \bigcap_{T\geq 0}\cl(\Phi^{[T,\infty)}(x)) \qquad \omega^*(x)\coloneqq \bigcap_{T\geq 0}\cl(\Phi^{(-\infty, -T]}(x)).$$
Given any $\dft$-integral curve segment $\gamma$ of $\dft$ with connected domain, we define $\omega(\gamma)\coloneqq \omega(\gamma(t))$  and $\omega^*(\gamma)\coloneqq \omega^*(\gamma(t))$ for any $t\in \dom(\gamma)$; this does not depend on the choice of $t$ \cite[p.~12]{palis1980geometric}.
\end{Def}
\begin{Rem}
If $\dft$ is as in Def.~\ref{def:omega-limit} and the chain recurrent set $R(\dft)$ is finite, then $R(\dft) = \dft^{-1}(0)$ and every integral curve $\gamma\colon \R\to M$ of $\dft$ satisfies $$\omega(\gamma) = \lim_{t\to+\infty}\gamma(t)\in \dft^{-1}(0) \quad \textnormal{ and } \quad \omega^*(\gamma) = \lim_{t\to -\infty}\gamma(t)\in \dft^{-1}(0).$$ 
\end{Rem}

Up to reparametrization, the following definition is a minor extension of what is known in the literature by names such as ``piecewise flow-lines''  \cite[p.~3]{cohen1995morse}, ``broken orbits''  \cite[p.~138]{weber2006morse}, and ``broken tunnelings'' \cite[p.~231]{nicolaescu2011invitation}.
It is a minor extension only because we do not require the initial and final integral curves $\gamma_1$ and $\gamma_N$ in Def.~\ref{def:piecewise-integral-curves} to be maximal (we do not require that they are defined on all of $\R$).
\begin{Def}[piecewise $\dft$-integral curves]\label{def:piecewise-integral-curves}
Let $\dft$ be a $\Cont^1$ vector field such that $R(\dft)$ is finite.
Let $(\gamma_1,\gamma_2,\ldots, \gamma_N)$ be a finite sequence of $\dft$-integral curve segments.
Denote by $\dom(\gamma_j)$ the domain of $\gamma_j$.
We say that this sequence is a \concept{piecewise $\dft$-integral curve} if 
\begin{itemize}
	\item $\dom(\gamma_j)=\R$ for $j\not \in \{1,N\}$,
	\item $\dom(\gamma_1)= [t_1,\infty)\cap \R$ and $\dom(\gamma_N)=(-\infty,t_N]\cap \R$  for some $t_1, t_N\in \R\cup \{-\infty, +\infty\}$ if $N\geq 2$, and 
	\item $\omega(\gamma_j) = \omega^*(\gamma_{j+1})$ for all $j\in \{1,\ldots, N-1\}$. 
\end{itemize}
Given $x,y\in M$, we say that the sequence is a piecewise $\dft$-integral curve \concept{from $x$ to $y$} if either $\gamma_1(t_1) = x$ or $\omega^*(\gamma_1) = x$ and either $\gamma_N(t_N)=y$ or $\omega(\gamma_N)=y$.
\end{Def}

\begin{Def}\label{def:int-curve-homotopy-class}
Given any $a,b\in [-\infty,+\infty]$ with $a<b$ and any curve $\gamma\colon [a,b]\cap \R\to M$ such that $\lim_{t\to a}\gamma(t)$ and $\lim_{t\to b}\gamma(t)$ exist, $[\gamma]\in \Pi(M)$ is defined to be the path homotopy class of the unique continuous extension to $[0,1]$ of any reparametrization of $\gamma|_{(a,b)}$ with domain $(0,1)$.
In particular, if $\dft$ is a $\Cont^1$ vector field on $M$ such that $R(\dft)$ is finite and if $c = (\gamma_1,\gamma_2,\ldots,\gamma_N)$ is a piecewise $\dft$-integral curve, then each $\gamma_j$ satisfies these assumptions.
In this case, we define $[c]\coloneqq [\gamma_1][\gamma_2]\cdots[\gamma_N]$. 
If $[c]=e\in \Pi(M)$, we say that \concept{$e$ contains the piecewise integral curve} $c$. 
\end{Def}
The definition \eqref{eq:action-functional} of the action $\af(\slot)$ and Lem.~\ref{lem:S-alt-expression} make the following result intuitively plausible. 
We defer the proof to App.~\ref{app:proofs}.
\begin{restatable}[]{Prop}{PropQpIntIffPiecewise}\label{prop:qp-int-conditions-iff-connecting-piecewise-orbit}Let $\dft$ be a $\Cont^1$ vector field on a closed Riemannian manifold $M$. 
Assume that the chain recurrent set $R(\dft)$ is finite.
Then for any $e\in \Pi(M)$, $$\qpd(e) = 0 \quad \iff \quad \textnormal{e contains a piecewise $\dft$-integral curve}$$
and
$$\qpd(e) = \inf_{[\vp]=e}\left(-\int_{T_1}^{T_2}\ip{\dot{\vp}}{\dft(\vp)}dt\right) \quad \iff \quad \textnormal{e contains a piecewise $(-\dft)$-integral curve},$$
where the infimum is over absolutely continuous paths $\vp$ of the form $\vp\colon [T_1,T_2]\to M$ with square integrable derivative and satisfying $[\vp]=e$. 
In particular, if $\dft = \cfo^\sharp$ is the metric dual of a closed one-form $\cfo$, then
$$\qp_{\cfo^\sharp}(e) = \int_{e}(-\cfo) \quad \iff \quad \textnormal{e contains a piecewise $(-\cfo^\sharp)$-integral curve}.$$
\end{restatable}
We expect the following result is known.
However, since we were unable to find a proof of this exact statement in the literature, we provide a proof in App.~\ref{app:proofs}.
\begin{restatable}[]{Lem}{OrbitsBoundedLength}\label{lem:orbits-bounded-length}	Let $\dft$ be a $\Cont^1$ vector field on a closed Riemannian manifold $M$.
Assume that $R(\dft)$ consists of a finite number of hyperbolic zeros.
Then there exists $C>0$ such that, if $\gamma$ is any integral curve of $\dft$,
\begin{equation}
\textnormal{length}(\gamma)<C.
\end{equation}
\end{restatable}

\begin{Def}\label{def:length}
	Given $e\in \Pi(M)$, we define $$\length(e)\coloneqq \inf\{\length(\vp)\colon \vp\in \Cont_e([0,1],M) \textnormal{ and }\vp \textnormal{ is absolutely continuous}\}.$$
\end{Def}

\begin{Lem}\label{lem:univ-cover-qp-growth}
	Let $\dft$ be a $C^1$ vector field on a closed Riemannian manifold $M$.
	Assume that $R(\dft)$ consists of a finite number of hyperbolic zeros.
	Then there exists $m,c>0$ such that, for all $e\in \Pi(M)$,
	\begin{equation}\label{eq:univ-cover-qp-growth-1}
	\qpd(e) \geq m \cdot \length(e)-c.
	\end{equation}
	Equivalently, if $\pi\colon \tM \to M$ is the smooth universal cover equipped with the pullback metric and $\qp_{\tdft}$ is the quasipotential defined with respect to the lift $\tdft$ of $\dft$ to $\tM$, then for all $x,y\in \tM$,
	\begin{equation}\label{eq:univ-cover-qp-growth-2}
	\qptd(x,y) \geq m \cdot \dist{x}{y}-c.
	\end{equation}
	Furthermore, if $\cfo$ is any closed 1-form on $M$, then for all $e\in \Pi(M)$,
	\begin{equation}\label{eq:univ-cover-qp-growth-3}
	\qpd(e) \geq m \cdot \left|\int_{e}\cfo\right| - c
	\end{equation}
	where $\int_e\cfo \coloneqq \int_{\vp}\cfo$ for any continuous path $\vp$ with $[\vp] = e$.
\end{Lem}
\begin{proof}
	By Prop.~\ref{prop:qp-int-conditions-iff-connecting-piecewise-orbit}, $\qpd(e) = 0$ implies that $e$ contains a piecewise $\dft$-integral curve.
	By Lem.~\ref{lem:orbits-bounded-length}, there exists $K_0>0$ such that every integral curve of $\dft$ has length less than $K_0$.
	There are no heteroclinic cycles for $\dft$ since $R(\dft) = \dft^{-1}(0)$ is finite, so the length of every piecewise $\dft$-integral curve is less than $K_0 \cdot \#(\dft^{-1}(0)) \eqqcolon K_1$.
	Thus, for $e\in \Pi(M)$ with $\length(e) \geq 2 K_1 \eqqcolon K_2$, $\qpd(e) > 0$.
	
	On $\tM$, the latter condition is equivalent to the condition that $\qptd(x,y) > 0$ for all $x,y$ with $\dist{x}{y}\geq K_2$. 
	The pullback metric on $\tM$ is complete since $M$ is compact \cite[p.~146, Thm~2.8]{docarmo1992riemannian}, so for any $x\in \tM$ the closed metric ball $B_{K_2}(x)$ of radius $K_2$ centered at $x$ is compact.
	Continuity of $\qptd$ \cite[p.~143, Lem.~1.1]{freidlin2012random} then implies that, for each $x\in \tM$, there is $\varepsilon_x > 0$ such that $\qptd(y,z) > \varepsilon_x$ for all $y\in B_{K_2}(x)$ and $z\in \partial B_{2K_2}(x)$.
    It follows that $\qptd(y,z)> \varepsilon_x$ for all $y\in B_{K_2}(x)$ and all $z\not \in \interior(B_{2K_2}(x))$.
	
	By invariance of $\tdft$ and the pullback metric on $\tM$ with respect to deck transformations, for any $x\in M$ it follows that  $\qptd(y,z) > \varepsilon_x$ for all $\tx\in \pi^{-1}(\pi(x))$, $y\in B_{K_2}(\tx)$, and $z\not \in \interior(B_{2K_2}(\tx))$. 
	By compactness of $M$, there exist finitely many $x_1,\ldots, x_n\in \tM$ such that every fiber $\pi^{-1}(x)$ has nonempty intersection with $\bigcup_{i=1}^n B_{K_2}(x_i)$.
	Defining $\varepsilon\coloneqq \min\{\varepsilon_{x_1},\ldots, \varepsilon_{x_n}\}$ and $K_3\coloneqq 2 K_2$, it follows that 
	$$\forall x,y\in \tM\colon \dist{x}{y}\geq K_3 \implies \qptd(x,y)> \varepsilon,$$
	and hence
	$$\forall x,y\in \tM\colon \forall n\in \N\colon  \dist{x}{y}\geq n K_3 \implies \qptd(x,y)> n \varepsilon.$$
	The latter condition follows from the first since any path from $x$ to $y$ with $\dist{x}{y}\geq nK_3$ must pass through $\partial B_{K_3}(x), \partial B_{2K_3}(x),\ldots, \partial B_{nK_3}(x)$ and therefore must pass through a sequence of at least $n$ points $y_1,\ldots, y_n$  satisfying $\dist{y_i} {y_{i+1}}\geq K_3$.
	Denoting by $\llfloor r\rrfloor\leq r$ the integer part of $r\in \R$, it follows that
	$$\forall x,y\in \tM\colon \qptd(x,y)> \llfloor \frac{\dist{x}{y}}{K_3} \rrfloor \varepsilon \geq \left(\frac{\dist{x}{y}}{K_3} -1 \right)\varepsilon = \frac{\varepsilon}{K_3}\dist{x}{y}-\varepsilon.$$
	Defining $m_0\coloneqq \frac{\varepsilon}{K_3}$ and $c\coloneqq \varepsilon$ establishes \eqref{eq:univ-cover-qp-growth-1} and \eqref{eq:univ-cover-qp-growth-2} for any $m\in (0,m_0]$.
	Since $$\left|\int_\vp \cfo \right|\leq \left(\sup_{x\in M}\norm{\cfo_x}\right)\length(\vp)$$ for any absolutely continuous $\vp$ with $[\vp]=e$, taking the infimum over all such $\vp$ and using \eqref{eq:univ-cover-qp-growth-1} yields \eqref{eq:univ-cover-qp-growth-3} for any $m\in (0,m_1]$, where $m_1\coloneqq \left(\sup_{x\in M}\norm{\cfo_x}\right)^{-1}m_0$.
	Taking $m\coloneqq \min(m_0,m_1)$ completes the proof.
\end{proof}

\subsection{Markov chain expressions for the steady-state flux}\label{sec:markov-chain-flux-expression}
Throughout the remainder of \S \ref{sec:proof-th-flux-manifold-mc-CRST-ld},  $\dft$ is a $\Cont^1$ vector field on a closed connected Riemannian manifold $M$ whose chain recurrent set consists of a finite number of hyperbolic zeros, and $(X^\varepsilon_t, \Prob_x^\varepsilon)$ is the diffusion process with generator \eqref{eq:diff-gen}, where $\dfte$ is smooth for each $\varepsilon > 0$ and $\dfte\to \dft$ uniformly as $\varepsilon \to 0$.

Recall that  $\Ws(z)$ and $\Wu(z)$ respectively denote the stable and unstable manifolds of $z\in \dft^{-1}(0)$, and the (Morse) index of $z\in \dft^{-1}(0)$ is  $\ind(z)\coloneqq \dim(\Wu(z))$.
Recall also that $\vtx \subset \dft^{-1}(0)$ are those zeros with index $0$,  $\Eh\subset \Pi(M)$ are those path homotopy classes with source $\src(\slot)$ and target $\tgt(\slot)$ points in $\vtx$, and $\gh$ is the  directed graph $\gh\coloneqq (\vtx,\Eh,\src,\tgt)$.
Note that the vertex set $\vtx$ is finite but the edge set $\Eh$ is infinite.
 
Let $\kappa_0 > 0$ be sufficiently small that the closed metric balls $B_{\kappa_0}(v)$ of radius $\kappa_0$ centered at each $v\in \vtx$ are pairwise disjoint and geodesically convex \cite[Thm~6.17]{lee2018riemannian}, and define $C\coloneqq M\setminus \bigcup_{v\in \vtx}\interior(B_{\kappa_0}(v))$.
Following \cite[Sec.~6.2]{freidlin2012random}, we fix $\kappa_1\in(0,\kappa_0)$ and define $g_v\coloneqq B_{\kappa_1}(v)$ for $v\in \vtx$ and $g\coloneqq \bigcup_{v\in \vtx}g_v$.
We introduce the stopping times $\tau_0 \coloneqq 0$, $\sigma_n \coloneqq \inf \{t \geq \tau_n\colon X^\varepsilon_t\in C\}$, and $\tau_n\coloneqq \inf\{t \geq \sigma_{n-1}\colon X^\varepsilon_t \in \partial g\}$ and consider the Markov chains $Z_n^\varepsilon = X^\varepsilon_{\tau_n}$ with unique invariant measure \cite[p.~107, Lem.~4.6, p.~120]{khasminskii2012stochastic} $\nu^\varepsilon$ on $\partial g$.
Given $v\in \vtx$ and $x\in g$, let $P(x,v) = \Prob_x^\varepsilon(X^\varepsilon_{\tau_1}\in \partial g_v)$.

Given $e\in \Eh$ and $x\in g_{\src(e)}$, we let $P(x,e)$ denote the probability conditioned on $X^\varepsilon_0 = x$ that, when concatenated with short paths in $g$ from $\vtx$ to $x$ and from $X^\varepsilon_{\tau_1}$ to $\vtx$, the resulting path $\vp = X^\varepsilon_{[0,\tau_1]}$ satisfies $[\vp]=e$.
The property $[\vp]=e$ does not depend on the choice of short paths since each of the $g_v$ are simply connected.

Given a closed one-form $\cfo$ on $M$ and $e\in \Pi(M)$, recall the definition $\int_e \cfo\coloneqq \int_{\vp}\cfo$, where $\vp$ is any continuous path $\vp$ satisfying $[\vp]=e$.
The following two lemmas express the flux $\fluxe([\cfo])$ defined in \eqref{eq:flux-def} in terms of the data defined above.
\begin{Lem}\label{lem:flux-manifold-mc}
	Let $\cfo$ be a closed one-form on $M$.
	Given $e\in \Eh$, define $$\bar{P}_{\kappa_1}(e)\coloneqq \frac{1}{\nu^\varepsilon(\partial g_{\src(e)})}\int_{\partial g_{\src(e)}}\nu^\varepsilon(dy)P(y,e)\quad \textnormal{and} \quad  
	\cfo(e)\coloneqq \int_e \cfo.$$
	Then the following integral and sum are absolutely convergent and:
	\begin{equation}\label{eq:flux-manifold-mc}
	\begin{split}
	\fluxe([\cfo]) = \left(\int_{\partial g}\nu^\varepsilon(dy)\E_y^\varepsilon[\tau_1]\right)^{-1} \sum_{e\in \Eh}\nu^\varepsilon(\partial g_{\src(e)}) \bar{P}_{\kappa_1}(e) \cfo(e).
	\end{split}
	\end{equation}  
\end{Lem}
\begin{proof}
	Given any path $\vp\in \Cont([T_1,T_2],M)$ with initial and terminal points in $g$, we denote by $e(\vp)\in \Eh$ the class of the path defined by following the unique minimizing geodesic in $g$ from $\vtx$ to $\vp(T_1)$, then following $\vp$, then following the unique minimizing geodesic in $g$ from $\vp(T_2)$ to $\vtx$.
	Since for any $n$ the path homotopy class $$[X^\varepsilon_{[0,\tau_n]}]=[\gamma_1]e(X^\varepsilon_{[0,\tau_{1}]})\cdots e(X^\varepsilon_{[\tau_{n-1},\tau_{n}]})[\gamma_2],$$
	where $\gamma_1,\gamma_2$ are the unique minimizing geodesics in $g$ joining $X^\varepsilon_0$ to $\vtx$ and $\vtx$ to $X^\varepsilon_{\tau_n}$, and since $|\int_{\gamma}\cfo|$ is uniformly bounded for any minimizing geodesic $\gamma$ in $g$, it follows that $$\lim_{n\to\infty} \frac{1}{\tau_n}\int_{X^\varepsilon_{[0,\tau_n]}}\cfo = \lim_{n\to\infty}\frac{1}{\tau_n}\sum_{i=0}^{n-1}\int_{e(X^\varepsilon_{[\tau_i,\tau_{i+1}]})}\cfo.$$
	Using this observation and the fact that the $\tau_1$ (hence also $\tau_n$ with $n\geq 1$) are positive with probability $1$ since $X^\varepsilon$ has continuous sample paths with probability $1$, we compute (with additional justification after):
	\begin{equation*}
	\begin{split}
	 \lim_{n\to\infty} \left(\frac{1}{\tau_{n}}  \int_{X^\varepsilon_{[0,\tau_{n}]}}\cfo\right)
	&= \lim_{n\to\infty} \left(\frac{n}{(\tau_{1}-\tau_0)+\cdots + (\tau_{n}-\tau_{n-1})} \cdot \frac{1}{n}\sum_{i=0}^{n-1}  \int_{e(X^\varepsilon_{[\tau_{i},\tau_{i+1}]})}\cfo\right)\\
	&\overset{\textnormal{a.s.}}{=} \left(\int_{\partial g}\nu^\varepsilon(dy)\E_y^\varepsilon[\tau_1]\right)^{-1} \int_{\partial g} \nu^\varepsilon(dy) \E_y^\varepsilon\left[\int_{e(X^\varepsilon_{[0,\tau_{1}]})} \cfo\right]\\
	&=  \left(\int_{\partial g}\nu^\varepsilon(dy)\E_y^\varepsilon[\tau_1]\right)^{-1} \sum_{e\in \Eh} \cfo(e) \nu^\varepsilon(\partial g_{\src(e)})  \frac{1}{\nu^\varepsilon(\partial g_{\src(e)})}\int_{\partial g_{\src(e)}}\nu^\varepsilon(dy)P(y,e),
	\end{split}
	\end{equation*}
    the right side of which is the right side of \eqref{eq:flux-manifold-mc}.
	The third line follows from the second and the law of total expectation.
	If both $\tau_1$ and $\int_{e(X^\varepsilon_{[0,\tau_1]})}\cfo$ are absolutely integrable, then the second line follows from the first and the strong law of large numbers for strictly stationary stochastic processes (or the Birkhoff ergodic theorem) \cite[p.~465,~Thm~2.1]{doob1953stochastic} since, by the strong Markov property of $X^\varepsilon$, the sequences $\left(\tau_1-\tau_0, \tau_2-\tau_1, \ldots \right)$ and $\left(\int_{e(X^\varepsilon_{[0,\tau_1]})}\cfo, \int_{e(X^\varepsilon_{[\tau_1, \tau_2]})}\cfo,\ldots \right)$ are strictly stationary processes \cite[p.~94]{doob1953stochastic} when given the joint distributions induced by the invariant measure $\nu^\varepsilon$ on $\partial g$ and the transition probabilities $\Prob_x^\varepsilon(X^\varepsilon_{\tau_1}\in \slot)$.
	
	We now argue that both $\tau_1$ and $\int_{e(X^\varepsilon_{[0,\tau_1]})}\cfo$ are absolutely integrable.
    The absolute integrability claim for $\tau_1 = (\tau_1-\sigma_0)+\sigma_0$ follows since
    \begin{equation*}
    \E_y^\varepsilon[\tau_1-\sigma_0] = \begin{cases}
    f_1(y), & y \in C\\
    0, &y\not \in C
    \end{cases}
    \qquad \text{and} \qquad \E_y^\varepsilon[\sigma_0] = \begin{cases}
    f_2(y), & y\not \in \interior(C)\\
    0, & y\in \interior(C)
    \end{cases},
    \end{equation*}
    where $f_1$ and $f_2$ are respectively the unique smooth (hence bounded uniformly in $y$) solutions to certain elliptic PDE boundary value problems on the compact domains $C$ and $M\setminus \interior(C)$ \cite[p.~90, Cor.~3.2,~p.~120]{khasminskii2012stochastic}.
    Since $M$ is compact it follows from \cite[pp.~382--384]{ikeda1981stochastic}, the Cauchy-Schwarz inequality, and the It\^{o} isometry that there is $C_0 > 0$ such that $\E_x^\varepsilon\left|\int_{X^\varepsilon_{[0,\tau]}}\cfo\right|\leq (\E_x^\varepsilon[\tau] + \sqrt{\E_x^\varepsilon[\tau]})C_0$ for any $x\in M$ and stopping time $\tau$.
    Since there is $C_1 > 0$ such that $\left|\int_{e(X^\varepsilon_{[0,\tau_1]})}\cfo \right| \leq \left|\int_{X^\varepsilon_{[0,\tau_1]}}\cfo\right| + C_1$, we see that the uniform boundedness of $\E_y^\varepsilon[\tau_1]$ implies the uniform boundedness of $\E_y^\varepsilon \left|\int_{e(X^\varepsilon_{[0,\tau_1]})}\cfo \right|$.
    In particular, $\int_{e(X^\varepsilon_{[0,\tau_1]})}\cfo $ is absolutely integrable.

	To complete the proof, we note that
    \begin{equation*}
    \begin{split}
    \sum_{e\in \Eh}\left|\nu^\varepsilon(\partial g_{\src(e)}) \bar{P}_{\kappa_1}(e) \cfo(e)\right| &= \sum_{e\in \Eh}\nu^\varepsilon(\partial g_{\src(e)}) \bar{P}_{\kappa_1}(e) |\cfo(e)| =  \sum_{e\in \Eh}|\cfo(e)| \int_{\partial g_{\src(e)}}\nu^\varepsilon(dy)P(y,e)\\
    &= \int_{\partial g}\nu^\varepsilon(dy)\E_y^\varepsilon \left|\int_{e(X^\varepsilon_{[0,\tau_{1}]})} \cfo \right|
    \end{split}
    \end{equation*}		
    by the law of total expectation, so the uniform boundedness of $\E_y^\varepsilon \left|\int_{e(X^\varepsilon_{[0,\tau_{1}]})} \cfo \right|$ established above and compactness of $\partial g$ imply that the sum in \eqref{eq:flux-manifold-mc} is absolutely convergent.
\end{proof}

\begin{Lem}\label{lem:flux-manifold-mc-CRST}
	Let $\cfo$ be a closed one-form on $M$. For $e\in \Eh$ let $\cfo(e)$, $\bar{P}_{\kappa_1}(e)$ be defined as in Lem.~\ref{lem:flux-manifold-mc}.
	Given $\edg\subset \Eh$ with $\#(\edg)<\infty$, define $\cfo(\edg)\coloneqq \sum_{e\in \edg}\cfo(e)$ and $\bar{\pi}(\edg)\coloneqq \prod_{e\in \edg}\bar{P}_{\kappa_1}(e)$. 
	Then the following integral and both sums are absolutely convergent and: 
	\begin{equation}\label{eq:flux-manifold-mc-CRST}
	\fluxe([\cfo]) = \left(\int_{\partial g}\nu^\varepsilon(dy)\E_y^\varepsilon[\tau_1]\right)^{-1} \left(\sum_{\edg\in \RST(\gh)}\bar{\pi}(\edg)\right)^{-1}\sum_{\edg \in \CRST(\gh)} \bar{\pi}(\edg)\cfo(\cycle(\edg)).
	\end{equation}
\end{Lem}
\begin{proof}
	Let 
	$\gph_c = (\vtx,\edg_c)$ denote the finite, complete directed graph on $\vtx$ (exactly one edge for each ordered pair of vertices) with transition probabilities given by $\bar{P}_{\kappa_1}(v,w)\coloneqq \sum_{e\in \src^{-1}(v)\cap \tgt^{-1}(w)}\bar{P}_{\kappa_1}(e)$ for $v,w\in \vtx$.
	Given $\edg\subset \edg_c$, we similarly define $\bar{\pi}(\edg)\coloneqq \prod_{e\in \edg}\bar{P}_{\kappa_1}(\src(e),\tgt(e))$.
	Using the definitions we find for all $v\in \vtx$ that
	\begin{equation*}
	\begin{split}
	\sum_{w\in \vtx }\nu^\varepsilon(\partial g_w)\bar{P}_{\kappa_1}(w,v)&= \sum_{w\in \vtx }\int_{\partial g_w}\nu^\varepsilon(dy)P(y,v) = \sum_{w\in \vtx }\int_{\partial g_w}\nu^\varepsilon(dy)\Prob_{y}(X^\varepsilon_{\tau_1}\in \partial g_v) \\
	&= \int_{\partial g}\nu^\varepsilon(dy)\Prob_{y}(X^\varepsilon_{\tau_1}\in \partial g_v) = \nu^\varepsilon(\partial g_v),
	\end{split}
	\end{equation*}
	since $\nu^\varepsilon$ is the invariant measure for the Markov chain $Z_n^\varepsilon = X^\varepsilon_{\tau_n}$.
	Hence $\mu^\varepsilon(v)\coloneqq \nu^\varepsilon(\partial g_v)$ coincides with the unique invariant measure of the irreducible and aperiodic Markov chain on $\gph_c$ with transition probabilities $\bar{P}_{\kappa_1}(\slot,\slot)$.
	Since $\gph_c$ is finite, we may thus apply the Markov chain tree formula \cite{pitman2018tree} to deduce that
	$$\nu^\varepsilon(\partial g_v) = \left(\sum_{\edg\in \RST(\gph_c)}\bar{\pi}(\edg)\right)^{-1} \sum_{\edg\in \RST(\gph_c;v)}\bar{\pi}(\edg)$$
	for each $v\in \vtx$.
	Substituting this expression into \eqref{eq:flux-manifold-mc} of Lem.~\ref{lem:flux-manifold-mc} yields
	\begin{equation*}
	\begin{split}
	\fluxe([\cfo]) &= \left(\int_{\partial g}\nu^\varepsilon(dy)\E_y^\varepsilon[\tau_1]\right)^{-1} \left(\sum_{\edg\in \RST(\gph_c)}\bar{\pi}(\edg)\right)^{-1}\sum_{e\in \Eh}\bar{P}_{\kappa_1}(e)\cfo(e)\sum_{\edg\in \RST(\gph_c;\src(e))}\bar{\pi}(\edg),
	\end{split}
	\end{equation*}
	and Lem.~\ref{lem:flux-manifold-mc} implies that the integral and second sum are absolutely convergent.
	In terms of the path-homotopical directed graph $\gh=(\vtx,\Eh,\src,\tgt)$, it follows from the definitions that $\sum_{\edg\in \RST(\gph_c)}\bar{\pi}(\edg) = \sum_{\edg\in \RST(\gh)}\bar{\pi}(\edg)$ and  $\sum_{\edg\in \RST(\gph_c;\src(e))}\bar{\pi}(\edg) = \sum_{\edg\in \RST(\gh;\src(e))}\bar{\pi}(\edg)$, so
	\begin{equation*}
	\begin{split}
	\fluxe([\cfo]) &=  \left(\int_{\partial g}\nu^\varepsilon(dy)\E_y^\varepsilon[\tau_1]\right)^{-1} \left(\sum_{\edg\in \RST(\gh)}\bar{\pi}(\edg)\right)^{-1}\sum_{e\in \Eh}\bar{P}_{\kappa_1}(e)\cfo(e)\sum_{\edg\in \RST(\gh;\src(e))}\bar{\pi}(\edg)\\
	&=  \left(\int_{\partial g}\nu^\varepsilon(dy)\E_y^\varepsilon[\tau_1]\right)^{-1}\left(\sum_{\edg\in \RST(\gh)}\bar{\pi}(\edg)\right)^{-1}\sum_{\edg\in \CRST(\gh)}\bar{\pi}(\edg)\cfo(\cycle(\edg))
	\end{split}
	\end{equation*}
	as desired.
\end{proof}

\subsection{Estimates on transition times and probabilities}\label{sec:transition-times-probabilities-estimates}
We aim to use Lem.~\ref{lem:flux-manifold-mc-CRST} to estimate $\fluxe([\cfo])$.
To do this, we will first estimate the terms $\E_y^\varepsilon[\tau_1]$ and $\bar{P}_{\kappa_1}(e)$ (hence also $\bar{\pi}(\edg)$) appearing in \eqref{eq:flux-manifold-mc-CRST}.
We begin with $\E_y^\varepsilon[\tau_1]$ in the following result.
In the proofs in this section we use the notation $\Cont_{x,y}([0,T],M)\subset \Cont([0,T],M)$ for those continuous paths going from $x\in M$ to $y\in M$.
\begin{Lem}\label{lem:tau-1-estimates}
	For every $\delta > 0$ there is $\bar{\kappa}>0$ such that, for all $0 < \kappa_1 < \kappa_0 < \bar{\kappa}$,  there is $\varepsilon_0 > 0$ such that, for all $0 < \varepsilon <\varepsilon_0$ and all $x\in g$,
	\begin{equation}\label{eq:tau-1-bound}
	e^{-\frac{1}{\varepsilon}\delta} \leq \E_x^\varepsilon[\tau_1] \leq e^{\frac{1}{\varepsilon}\delta}.
	\end{equation}	
\end{Lem}
\begin{proof}
	Fix $\delta > 0$.
	From \cite[p.~143, Lem.~1.1]{freidlin2012random} it follows that, if $\bar{\kappa} > \kappa_0 > \kappa_1$ is sufficiently small, then for any $v\in \vtx$, $x\in  B_{\kappa_0}(v)$, and $y\in \partial B_{\kappa_0 + (\kappa_0-\kappa_1)}(v)$ there exists a constant path $c_{x,T_0}\in \Cont_{x,x}([0,T_0],M)$ and a short path $\vp\in \Cont_{x,y}([0,T_1],M)$ such that $\af_{T_0}(c_{x,T_0}),\af_{T_1}(\vp)< \delta/4$.

	Defining $\distT{\gamma}{\psi}{T}\coloneqq \max_{t\in [0,T]}\dist{\gamma(t)}{\psi(t)}$, it follows
	from \cite[p.~74; p.~135, Thm~3.2]{freidlin2012random} that there exists $\varepsilon_0$ independent of $v,x,y$ such that $\Prob_x^\varepsilon(\distT{X^\varepsilon}{c_{x_0,T_0}}{T_0}< (\kappa_0-\kappa_1)/2)\geq e^{-\frac{1}{\varepsilon}\delta/2}$ and $\Prob_x^\varepsilon(\distT{X^\varepsilon}{\vp}{T_1}< (\kappa_0-\kappa_1)/2)\geq e^{-\frac{1}{\varepsilon}\delta/2}$ for all $0 < \varepsilon < \varepsilon_0$.
	Hence
	\begin{equation}\label{eq:sigma-geq-prob-lower-bound}
    \forall x\in g_v\colon  \Prob_x^\varepsilon(\sigma_0\geq T_0) \geq \Prob_x^\varepsilon(\distT{X^\varepsilon}{c_{x,T_0}}{T_0}< (\kappa_0-\kappa_1)/2) \geq e^{-\frac{1}{\varepsilon}\delta/2}
	\end{equation}
	and 
	\begin{equation}\label{eq:sigma-leq-prob-lower-bound}
	\forall x\in B_{\kappa_0}(v)\colon \Prob_x^\varepsilon(\sigma_0< T_1)\geq \Prob_x^\varepsilon(\distT{X^\varepsilon}{\vp}{T_1}< (\kappa_0-\kappa_1)/2) \geq e^{-\frac{1}{\varepsilon}\delta/2}.
	\end{equation}
	From \eqref{eq:sigma-geq-prob-lower-bound} and $\tau_1 \geq \sigma_0$ we obtain (if $e^{-\frac{1}{\varepsilon_0}\delta/2}<T_0$) the desired lower bound 
	\begin{equation}\label{eq:sigma-mean-lower-bound}
	\E_x^\varepsilon[\tau_1]\geq \E_x^\varepsilon[\sigma_0]\geq T_0\cdot \Prob_x^\varepsilon(\sigma_0\geq T_0) > e^{-\frac{1}{\varepsilon}\delta}.
	\end{equation}
	From \eqref{eq:sigma-leq-prob-lower-bound} and the strong Markov property, for $n\in \N$ we obtain
	$$\Prob_x^\varepsilon(\sigma_0 \geq nT_1)\leq (1-e^{-\frac{1}{\varepsilon}\delta/2})^n,$$
	hence (cf. \cite[p.~148, Lem.~1.7]{freidlin2012random})
	\begin{equation}\label{eq:sigma-mean-upper-bound}\begin{split}
	\E_x^\varepsilon[\sigma_0] &\leq  T_1\sum_{n=0}^\infty (n+1)\Prob_x^\varepsilon((n+1)T_1\geq \sigma_0\geq nT_1) =  T_1\sum_{n=0}^\infty \Prob_x^\varepsilon(\sigma_0\geq nT_1) \\
	&\leq T_1\sum_{n=0}^\infty (1-e^{-\frac{1}{\varepsilon}\delta/2})^n = T_1 e^{\frac{1}{\varepsilon}\delta/2}.
	\end{split}
	\end{equation}
	By \cite[p.~173, Lem.~5.1]{freidlin2012random}, \cite[p.~176, Lem.~5.3, Thm~5.3]{freidlin2012random}, $$\lim_{\varepsilon\to 0}\varepsilon \ln \E_x^\varepsilon[\tau_1 - \sigma_0] = 0$$
	uniformly in $x\in C$ since $C$ contains no stable invariant sets for $\dft$.
	Hence after a further $(v,x,y)$-independent shrinking of $\varepsilon_0$ if necessary,  $e^{-\frac{1}{\varepsilon}\delta/2} < \E_x^\varepsilon[\tau_1-\sigma_0] < e^{\frac{1}{\varepsilon}\delta/2}$ for all $0 < \varepsilon < \varepsilon_0$.
	From this and \eqref{eq:sigma-mean-upper-bound} we obtain $\E_x^\varepsilon[\tau_1] = \E_x^\varepsilon[\sigma_0] + \E_x^\varepsilon[\tau_1-\sigma_0]< (1+T_1)e^{\frac{1}{\varepsilon}\delta/2} \leq e^{\frac{1}{\varepsilon}\delta}$, where the latter inequality holds if $1+T_1 \leq e^{\frac{1}{\varepsilon_0}\delta}$.
	This establishes \eqref{eq:tau-1-bound}.
\end{proof}

The following result provides estimates on the transition probabilities $P(x,e)$ defined in \S \ref{sec:markov-chain-flux-expression}, where $e\in \Eh$ and $x\in g_{\src(e)}$.
\begin{Lem}\label{lem:FW-transition-refinement}
	For any $\delta, N > 0$ there is $\bar{\kappa}>0$ such that, for all $0 < \kappa_1 < \kappa_0 < \bar{\kappa}$,  there is $\varepsilon_0 > 0$ such that, for all $0 < \varepsilon <\varepsilon_0$, all $v\in \vtx$, and all $x\in g_v$:
	\begin{equation}\label{eq:FW-transition-P-refinement-decay}
	\forall m\in \N_{\geq 1}\colon \sum_{\{e\in \src^{-1}(v)\colon \tqpd(e) \geq mN\}} P(x,e)\leq e^{-\frac{1}{\varepsilon}m(N-\delta)}
	\end{equation}
	and, for all $e\in \src^{-1}(v)$ satisfying $\tqpd(e)\leq N$:
	\begin{equation}\label{eq:FW-transition-refinement-P}
	 e^{-\frac{1}{\varepsilon}(\tqpd(e)+\delta)} \leq P(x,e)\leq e^{-\frac{1}{\varepsilon}(\tqpd(e)-\delta)}.
	\end{equation}
\end{Lem}

\begin{proof}
    In the case that $\dim(M) = 1$ (so that $M$ is diffeomorphic to the circle), the statement of the lemma follows straightforwardly from \cite[Ch.~6, Thm~5.1]{freidlin2012random}, so we may and do henceforth assume that $\dim(M) \geq 2$.
    Recall that in this case $\qpd = \tqpd$ (Lem.~\ref{lem:tqp-equals-qp-transversality}), and this allows us to take advantage of continuity of $\qpd$ \cite[p.~143, Lem.~1.1]{freidlin2012random}.

	We work on the universal cover $\pi\colon \tM\to M$ equipped with the pullback metric, which is complete since $M$ is compact \cite[p.~146, Thm~2.8]{docarmo1992riemannian}, and we denote by $\tdft$ the unique lift of $\dft$ to $\tM$ ($\pi_*\tdft=\dft$).
	We work with the quasipotential $\qptd(\slot,\slot) = \tqptd(\slot,\slot)$ as in the proof of Lem.~\ref{lem:univ-cover-qp-growth}.
	Given $A,B\subset \tM$ and $x\in \tM$, we define the sets
	\begin{equation}\label{eq:qptd-images}
	\begin{split}
	\qptd(A,B)\coloneqq \{\qptd(a,b)\colon a\in A, b\in B\}, \quad  \qptd(x,A)\coloneqq \qptd(\{x\},A), \quad  \qptd(A,x)\coloneqq \qptd(A,\{x\}).
	\end{split}
	\end{equation}
	We also consider the lift $(\tilde{X}^\varepsilon_t, \tProbe_x)$ of the diffusion process $(X^\varepsilon_t,\Prob^\varepsilon_x)$ to $\tM$.
	We define $\tvtx\coloneqq \pi^{-1}(\vtx)$, $\tilde{g}\coloneqq \pi^{-1}(g)$, and $\tilde{g}_{v}$ to be the connected component of $\pi^{-1}(g_{\pi(v)})$ containing $v\in \tvtx$.
	Note that $\tilde{g}_v = B_{\kappa_1}(v)$, where $\kappa_1$ is yet to be specified.
	We consider the Markov chain $\tilde{Z}_n = \tilde{X}^\varepsilon_{\tau_n}$ on $\partial \tilde{g}$, where the stopping times $\sigma_n$, $\tau_n$ are defined as before.
	Given $v,w\in \tvtx$ and $x\in \tilde{g}_v$, we let $P(x,w) = \tProbe_x(\tilde{X}^\varepsilon_{\tau_1}\in \partial \tilde{g}_w)$.
	We denote by $H = \Aut(\pi)$ the group of deck transformations which acts by isometries on $\tM$.

    Fix $N,\delta > 0$ and  $N_0 > \max(3,3\delta/N)$ (to aid intuition, $N_0$ may be arbitrarily large).
    Roughly speaking, we would like to establish \eqref{eq:FW-transition-P-refinement-decay} and \eqref{eq:FW-transition-refinement-P} by employing \cite[Ch.~6, Thm~5.1]{freidlin2012random} to estimate transition probabilities within the compact $\qptd(x,\slot)$-sublevel sets, but these sublevel sets are not smooth (as the cited theorem requires).
    Thus, we will consider ``smooth approximate $\qptd(x,\slot)$-sublevel sets''.
    More precisely, Sard's theorem and Lem.~\ref{lem:univ-cover-qp-growth} imply that for each $x\in \tM$ there exists a compact, connected, smooth, codimension-$0$ submanifold $D_x$ with boundary and $N_x >N_0$ such that (cf. \eqref{eq:qptd-images}):\footnote{Proof: for each $x\in \tM$ and $h > 0$, define the open sublevel set $U_{x,h}\coloneqq \{y\in \tM\colon \qptd(x,y)< h\}$.
    It follows from Lem.~\ref{lem:univ-cover-qp-growth} and completeness of the pullback metric on $\tM$ that each $U_{x,N}$ is precompact, so the set $U_{x,N} \cap \tdft^{-1}(0)$ is finite since $\tdft^{-1}(0)=\pi^{-1}(\dft^{-1}(0))$ is  discrete. 
    Thus, there is $N_x > N_0$ such that $D_x\cap \tdft^{-1}(0)\subset [0,N-3\delta/N_x]$.
    For each $x\in \tM$, let $f_x\colon \tM\to [0,1]$ be a $\Cont^\infty$ function satisfying $f_x^{-1}(1) = \cl(U_{x,N-2\delta/N_x})$ and $\textnormal{supp}(f_x)\subset U_{x,N-\delta/N_x}$.
    By Sard's theorem, for each $x\in \tM$ there exists a regular value $c_x\in (0,1)$ of $f$. 
    Defining $D_x$ to be the connected component of $f_x^{-1}([c_x,1])$ containing $x$, since $\partial D_x = f^{-1}(c_x)$ the collection $(D_x)_{x\in \tM}$ satisfies the required conditions.}
    \begin{align}
    \qptd(x,D_x) &\subset [0,N-\delta/N_x] \nonumber\\
    \qptd(x,\partial D_x)&\subset [N-2\delta/N_x,N-\delta/N_x] \label{eq:qp-z-pDz-cont}\\
    \qptd(x,D_x \cap \tdft^{-1}(0))&\subset [0,N-3\delta/N_x] \label{eq:qp-z-zeros-leq}.
    \end{align}
    In particular, it follows from \eqref{eq:qp-z-pDz-cont} and \eqref{eq:qp-z-zeros-leq} that $\partial D_x \cap \tdft^{-1}(0) = \varnothing$.

    Since $M$ is compact and $\qptd(\slot,\slot)$ is continuous \cite[p.~143, Lem.~1.1]{freidlin2012random} and $H$-invariant, it follows that $\qptd(\slot,\slot)$ is uniformly continuous.
    Since also the pullback metric is $H$-invariant, there is $r_0 > 0$ such that $B_{r_0}(v) \cap B_{r_0}(w)=\varnothing$ for all $v,w\in \tvtx$ and
    \begin{equation}\label{eq:FW-trans-r0-choice}
    \forall x,y\in \tM\colon \dist{x}{y}\leq r_0\implies \qptd(x,y)< \min(\delta/(2N_0), N-3\delta/N_0).
    \end{equation}
    Using \eqref{eq:qp-z-pDz-cont}, it follows in particular that $B_{r_0}(x)\subset \interior(D_x)$ for all $x\in \tM$.
    For each $v\in \vtx$, fix a single representative $\tv\in \pi^{-1}(v)$ and let $S\subset \tvtx$ be the finite set of representatives.
	By compactness of $M$, there exists a finite subset $I_0\subset \tM$ with $S\subset I_0$ such that 
	$\bigcup_{z\in I_0}B_{r_0}(\pi(z)) = M$ and hence also $H \cdot \bigcup_{z\in I_0}B_{r_0}(z) = \tM$.
    Defining $I\coloneqq H\cdot I_0\supset \tvtx$, it follows that 
	\begin{equation}\label{eq:union-r0-balls-tm}
	\bigcup_{z\in I}B_{r_0}(z)=\tM.
	\end{equation}    
    
    Since the set $I_0$ is finite, so is $N_1\coloneqq \max_{z\in I_0} N_z > N_0$.
    We choose $\kappa_0\in (0,r_0)$ small enough that (using uniform continuity of $\qptd(\slot,\slot)$)
    \begin{equation}\label{eq:b-kappa-V-small}
    \forall x,y\in \tM\colon \dist{x}{y}\leq \kappa_0 \implies \qptd(x,y) < \delta/(2N_1).
    \end{equation}
    Since \eqref{eq:qp-z-pDz-cont} and \eqref{eq:qp-z-zeros-leq} imply that (cf. \eqref{eq:qptd-images})$$\qptd(D_z\cap \tdft^{-1}(0),\partial D_z) \geq (N-2\delta/N_z)-(N-3\delta/N_z) = \delta/N_z \geq \delta/N_1$$ for all $z\in I_0$, it follows that $\partial D_z \cap \left(\bigcup_{y\in D_z \cap \tdft^{-1}(0)}B_{\kappa_0}(y)\right)=\varnothing$ for all $z\in I_0$.
    Since $H$ acts on $\tM$ by isometries and since $\qptd$ is $H$-invariant, it follows that
\begin{equation}\label{eq:b-kappa-disjoint-Dz}
\forall z\in I\colon \partial D_z \cap \left(\bigcup_{y\in D_z \cap \tdft^{-1}(0)}B_{\kappa_0}(y)\right)=\varnothing. 
\end{equation}    
    
For each $z\in I_0$ and $w\in  \tvtx \cap D_{z}$, fix $\vp_{z,w}\in \Cont_{z,w}([0,T_{z,w}],D_{z})$ satisfying $\vp_{z,w}((0,T_{z,w})) \cap \tdft^{-1}(0) = \varnothing$ and
\begin{equation}\label{eq:FW-trans-af-vp-bound}
\af_{T_{z,w}}(\vp_{z,w}) < \qptd(z,w) + \delta/N_z.
\end{equation}
This is possible since a path $\vp_{z,w}$ from $z$ to $w$ such that $\vp_{z,w}((0,T_{z,w}))$ is disjoint from $\tdft^{-1}(0)$ and $\af_{T_{z,w}}(\vp_{z,w})< \qptd(z,w) + \delta/N_z$ exists by Lem.~\ref{lem:tqp-equals-qp-transversality} since $\dim(M)\geq 2$, and such a path must necessarily be contained in $D_{z}$ by \eqref{eq:qp-z-pDz-cont} and \eqref{eq:qp-z-zeros-leq}.
Since there are only finitely many such paths (since $\#(I_0)<\infty$), if necessary we may further shrink $\kappa_0 > 0$  to ensure that the image of $\vp_{z,w}$ is disjoint from $\bigcup_{y\in D_z \cap \tdft^{-1}(0)}B_{\kappa_0}(y)$ for all $z\in I_0$.
We define $\vp_{hz,hw}\coloneqq h\vp_{z,w}$ for each $z\in I_0$, $w\in  \tvtx \cap D_{z}$, and $h\in H$.
Since everything in sight is $H$-invariant, this yields a collection of paths $\vp_{z,w}$ satisfying the same properties for all $z \in I$ and $w\in  \tvtx \cap D_{z}$.

Now fix any $\kappa_1\in (0,\kappa_0)$ and recall that $\tilde{g}_{v} = B_{\kappa_1}(v)$ for $v\in \tvtx$ and $\tilde{g} = \bigcup_{v\in \tvtx}\tilde{g}_v$.
For each $z\in I$ we define $\hat{D}_{z}\coloneqq D_{z}\setminus \interior(\tilde{g})$ and
\begin{align}
\zeta_{z}&\coloneqq \inf\{t\geq \sigma_0\colon \tilde{X}^\varepsilon_{t}\in \partial \hat{D}_{z}\}\\
\forall x,y\in \hat{D}_{z}\colon \qptdz(x,y)&\coloneqq \inf \{\af_T(\vp)\colon T > 0, \vp \in \Cont_{x,y}([0,T],\hat{D}_{z})\}.
\end{align}
Eq.~ \eqref{eq:b-kappa-disjoint-Dz} implies that each $\hat{D}_z$ is a compact, smooth, codimension-$0$ submanifold with boundary since each $D_z$ is, and this will shortly enable us to apply \cite[Ch.~6,~Thm~5.1]{freidlin2012random}. 

Note that $\tilde{X}^\varepsilon_{\sigma_0}\in \pi^{-1}(C)$, each $\hat{D}_z$ contains no stable invariant sets for $\tdft$ by construction, and $\qptdz(\slot,\slot) \geq \qptd(\slot,\slot)$.
It thus follows from \eqref{eq:FW-trans-r0-choice} and \eqref{eq:b-kappa-V-small}, the strong Markov property, \cite[Ch.~6, Thm~5.1]{freidlin2012random}, and $N_0 < N_z$ (for all $z$) that there is $\varepsilon_1 > 0$ such that, for all $\varepsilon \in (0,\varepsilon_1)$, the following inequalities hold for all $z\in I_0$.
By invariance of the transition probabilities with respect to the action of $H = \Aut(\pi)$ by isometries, they also hold for all $z\in I$:
\begin{align}
\forall x\in  B_{r_0}(z)\colon \tProbe_x(\tilde{X}^\varepsilon_{\zeta_{z}}\in \partial D_{z})&\leq e^{-\frac{1}{\varepsilon}(N-3\delta/N_0)}\label{eq:prob-exit-ub}\\
\forall x\in  B_{r_0}(z)\colon \forall w\in \tvtx \cap D_{z}\colon e^{-\frac{1}{\varepsilon}(\qptdz(x,w)+2\delta/N_0)} \leq \tProbe_x(\tilde{X}^\varepsilon_{\zeta_{z}}\in \partial \tilde{g}_{w}) &\leq  e^{-\frac{1}{\varepsilon}(\qptd(x,w)-2\delta/N_0)} \label{eq:prob-transit-in-D-ub}.
\end{align}
Since $I\supset \tvtx$, we may substitute any $v\in \tvtx$ for $z$ in \eqref{eq:prob-exit-ub} and \eqref{eq:prob-transit-in-D-ub}.
Hence the lower bound in \eqref{eq:FW-transition-refinement-P} follows from the lower bound in  \eqref{eq:prob-transit-in-D-ub}, $\tProbe_x(\tilde{X}^\varepsilon_{\zeta_{z}}\in \partial \tilde{g}_{w})\leq P(x,w)$, $r_0>\kappa_0> \kappa_1$, $N_0 > 3$, and the fact that $\qptdz(x,w)\leq \af_{T_{\tv,w}}(\vp_{\tv,w})< \qptd(x,w) + \delta/N_0$ by \eqref{eq:FW-trans-af-vp-bound} and $N_z > N_0$.
Next, we estimate
\begin{equation}\label{eq:FW-pxe-final-short-estimate}
\begin{split}
P(x,w) &\leq  \tProbe_x(\tilde{X}^\varepsilon_{\zeta_{v}}\in \partial \tilde{g}_{w}) + \tProbe_x(\tilde{X}^\varepsilon_{\zeta_{v}}\in \partial D_v) \leq e^{-\frac{1}{\varepsilon}(\qptd(x,w)-2\delta/N_0)} + e^{-\frac{1}{\varepsilon}(N-3\delta/N_0)}\\
&\leq e^{-\frac{1}{\varepsilon}(\qptd(x,w)-3\delta/N_0)} 
\end{split}
\end{equation}	
for all $v\in \tvtx\subset I$ and $x\in \tilde{g}_{v}\subset B_{r_0}(v)$ by \eqref{eq:prob-exit-ub} and \eqref{eq:prob-transit-in-D-ub}, with the last inequality holding for all $\varepsilon \in (0,\varepsilon_0)$ if $\varepsilon_0\in (0, \varepsilon_1)$ is sufficiently small.
Since $N_0 > 3$ and $r_0>\kappa_0 >  \kappa_1$, \eqref{eq:FW-pxe-final-short-estimate} implies the upper bound in \eqref{eq:FW-transition-refinement-P}.
		
Finally, let $f\colon \tM\to I$ be any Borel measurable ``selection function'' satisfying $f(z)=z$ for all $z\in I$ and, for all $y\in \tM$, $f(y)=z$ for some $z\in I$ satisfying $y\in B_{r_0}(z)$.
Such a map exists since $\tM = \bigcup_{z\in I}B_{r_0}(z)$ by \eqref{eq:union-r0-balls-tm}.
Define the sequence $(s_n)_{n\in \N}$ of stopping times by setting $s_0 \coloneqq 0$ and inductively defining $s_{n+1}\coloneqq s_n + \zeta_{f(\tilde{X}^\varepsilon_{s_n})}$. 
Then for any $v\in \tvtx$ and $x\in \tilde{g}_{v}$, \eqref{eq:qp-z-pDz-cont} and \eqref{eq:FW-trans-r0-choice} imply that
\begin{equation}
\Prob_x^\varepsilon(\tqptd(x,X^\varepsilon_{\tau_1})\geq mN)\leq \prod_{i=0}^{m-1} \Prob_{\tilde{X}^\varepsilon_{s_i}}(\tilde{X}^\varepsilon_{s_{i+1}}\in \partial D_{f(\tilde{X}^\varepsilon_{s_i})})\leq e^{-\frac{1}{\varepsilon}m(N-3\delta/N_0)},
\end{equation}
where the last equality follows from \eqref{eq:prob-exit-ub} (substituting $v\in \tvtx$ for $z\in I\supset \tvtx$).
Since $N_0>3$, this implies \eqref{eq:FW-transition-P-refinement-decay} and completes the proof.
\end{proof}

\subsection{Finishing the proof}\label{sec:finishing-the-proof-crst}
We now prove Theorem~\ref{th:flux-manifold-mc-CRST-ld}.
For convenience we restate the theorem.
\ThmFWGeneral*
\begin{proof}
We first note that the minimizers in \eqref{eq:min-rst-assumption} (and maximizers below) exist by Lem.~\ref{lem:univ-cover-qp-growth} and discreteness of $\dft^{-1}(0)$.

We will use the following additional notation: for $n\geq 0$ we define $\Eh^n\coloneqq \{e\in \Eh\colon \tqpd(e)\geq n\}$, $ R_n \coloneqq \{\edg\in \RST(\gh) \colon \edg\cap \Eh^n \neq \varnothing\}$,  $C_n\coloneqq \{\edg\in \CRST(\gh)\colon \edg\cap \Eh^n \neq \varnothing\}$, and $2^{\Eh}_n \coloneqq \{\edg\in 2^{\Eh}\colon \#(\edg)\leq \#(\vtx) \textnormal{ and }\edg\cap \Eh^n \neq \varnothing\}$.
Note that $R_n = 2^{\Eh}_n\cap \RST(\gh)$ and $C_n = 2^{\Eh}_n\cap \CRST(\gh)$  since $\#(\edg)=\#(\vtx)-1$ for every $\edg\in \RST(\gh)$ and $\#(\edg)=\#(\vtx)$ for every $\edg\in \CRST(\gh)$.

Fix any $\delta > 0$.
Let $m, c> 0$ be as in \eqref{eq:univ-cover-qp-growth-3} of Lem.~\ref{lem:univ-cover-qp-growth} and choose $N > 0$ such that $N - \delta$ is strictly larger than both of the minimums in \eqref{eq:flux-sde-mc-ld-expression}.
By Lem.~\ref{lem:tau-1-estimates} and Lem.~\ref{lem:FW-transition-refinement} there is $\bar{\kappa}>0$ such that, for all $0 < \kappa_1 < \kappa_0 < \bar{\kappa}$, there is $\varepsilon_6 > 0$ such that, for all $0<\varepsilon<\varepsilon_6$, \eqref{eq:tau-1-bound}, \eqref{eq:FW-transition-P-refinement-decay} and \eqref{eq:FW-transition-refinement-P} hold with $\delta/6$ replacing $\delta$.
Moreover, using \eqref{eq:FW-transition-refinement-P}, we may assume that $\varepsilon_6$ is small enough that, for all $0<\varepsilon < \varepsilon_6$, 
\begin{equation}\label{eq:bar-pi-bounds}
e^{-\frac{1}{\varepsilon}(\delta/6+ \sum_{e\in \edg}\tqpd(e))}  \leq \bar{\pi}(\edg)\leq  e^{-\frac{1}{\varepsilon}(-\delta/6+ \sum_{e\in \edg}\tqpd(e))} 
\end{equation}
for any finite subset $\edg\subset \Eh$ satisfying $\#(\edg)\leq \#(\vtx)$ and $\tqpd(e)\leq N$ for all $e\in \edg$.

We first use  \eqref{eq:tau-1-bound} to obtain from \eqref{eq:flux-manifold-mc-CRST} of Lem.~\ref{lem:flux-manifold-mc-CRST} that
\begin{equation}\label{eq:flux-manifold-mc-CRST-ld-1}
\begin{split}
e^{-\frac{1}{\varepsilon}\delta/6}\left(\sum_{\edg\in \RST(\gh)}\bar{\pi}(\edg)\right)^{-1}&\sum_{\edg \in \CRST(\gh)} \bar{\pi}(\edg)\cfo(\cycle(\edg)) \leq \fluxe([\cfo])\\
&\leq e^{\frac{1}{\varepsilon}\delta/6}\left(\sum_{\edg\in \RST(\gh)}\bar{\pi}(\edg)\right)^{-1}\sum_{\edg \in \CRST(\gh)} \bar{\pi}(\edg)\cfo(\cycle(\edg)),
\end{split}
\end{equation}
where $\bar{\pi}(\slot)$ is as defined in Lem.~\ref{lem:flux-manifold-mc-CRST}.

Next, by the definition $\bar{P}_{\kappa_1}(e)\coloneqq \frac{1}{\nu^\varepsilon(\partial g_{\src(e)})}\int_{\partial g_{\src(e)}}\nu^\varepsilon(dy)P(y,e)$  and the Fubini-Tonelli theorem, for any subset $\edg\subset \Eh$ it follows that 
\begin{equation}\label{eq:sum-p-bar-expression}
\sum_{e\in \edg}\bar{P}_{\kappa_1}(e) = \sum_{v\in \vtx}\frac{1}{\nu^\varepsilon(\partial g_v)}\int_{\partial g_v}\nu^\varepsilon(dy)\sum_{e\in \src^{-1}(v)\cap \edg}P(y,e).
\end{equation}
If we set $\edg = \Eh$ in \eqref{eq:sum-p-bar-expression} the sum inside the integral becomes equal to $1$, so we obtain
\begin{equation}\label{eq:sum-p-bar-all}
\sum_{e\in \Eh}\bar{P}_{\kappa_1}(e) = \#(\vtx).
\end{equation}
If we set $\edg = \Eh^{kN}$ in \eqref{eq:sum-p-bar-expression} for $k\in \N$, then by \eqref{eq:FW-transition-P-refinement-decay} of Lem.~\ref{lem:FW-transition-refinement} (with $\delta/6$ replacing $\delta$) the sum inside the integral in \eqref{eq:sum-p-bar-expression} is bounded above by $e^{-\frac{1}{\varepsilon}k(N-\delta/6)}$: 
\begin{equation}\label{eq:sum-p-bar-em}
\sum_{e\in \Eh^{kN}}\bar{P}_{\kappa_1}(e)\leq \#(\vtx) e^{-\frac{1}{\varepsilon}k(N-\delta/6)}.
\end{equation}
Let $\varepsilon_5\in (0,\varepsilon_6)$ be small enough that $\sum_{\ell=1}^{\#(\vtx)}\#(\vtx)^{\ell}\leq e^{\frac{1}{\varepsilon_5}(\delta/5-\delta/6)}$.
For any $k\in \N$, we compute
\begin{equation*}
\sum_{\edg\in 2^{\Eh}_{kN}}\bar{\pi}(\edg) = \sum_{\ell=1}^{\#(\vtx)}\sum_{\substack{e_1\in \Eh^{kN}\\e_2,\ldots,e_\ell\in\Eh}} \bar{P}_{\kappa_1}(e_1)\cdots \bar{P}_{\kappa_1}(e_\ell) = \left(\sum_{e\in \Eh^{kN}}\bar{P}_{\kappa_1}(e)\right)\sum_{\ell=1}^{\#(\vtx)} \left(\sum_{e\in \Eh}\bar{P}_{\kappa_1}(e)\right)^{\ell-1}
\end{equation*}
using the Fubini-Tonelli theorem.
Substituting \eqref{eq:sum-p-bar-all} and \eqref{eq:sum-p-bar-em} into the above yields, for all $0<\varepsilon<\varepsilon_5$ and $k\in \N$,
\begin{equation}\label{eq:sum-rst-n-bound}
\sum_{\edg\in 2^{\Eh}_{kN}}\bar{\pi}(\edg)\leq e^{-\frac{1}{\varepsilon}k(N-\delta/6)}\sum_{\ell=1}^{\#(\vtx)}\#(\vtx)^{\ell}\leq e^{-\frac{1}{\varepsilon}k(N-\delta/5)},
\end{equation}
where the second inequality follows from our choice of $\varepsilon_5$.

We now bound the sum over $\edg\in \RST(\gh) = R_0$ in \eqref{eq:flux-manifold-mc-CRST-ld-1}.
By the bounds in \eqref{eq:bar-pi-bounds} and since $N-\delta> \min_{\edg\in \RST(\gh)}\sum_{e\in \edg}\tqpd(e)$, there exists $\varepsilon_4\in (0,\varepsilon_5)$ such that, for all $0<\varepsilon < \varepsilon_4$,
\begin{equation}\label{eq:sum-bar-pi-R0-RN-bounds}
\max_{\edg\in \RST(\gh)} e^{-\frac{1}{\varepsilon}(\delta/5+ \sum_{e\in \edg}\tqpd(e))} \leq \sum_{\edg\in R_0\setminus R_N}\bar{\pi}(\edg)\leq \max_{\edg\in \RST(\gh)} e^{-\frac{1}{\varepsilon}(-\delta/5+ \sum_{e\in \edg}\tqpd(e))}.
\end{equation}
Since $\RST(\gh) = (R_0 \setminus R_N) \cup R_N$ and $R_N = 2^{\Eh}_N\cap \RST(\gh)$, we have that $(R_0\setminus R_N) \subset \RST(\gh)\subset (R_0\setminus R_N) \cup 2^{\Eh}_N$. 
Therefore, combining \eqref{eq:sum-bar-pi-R0-RN-bounds} with \eqref{eq:sum-rst-n-bound} (with $k = 1$ in the latter) yields the existence of $\varepsilon_3\in (0,\varepsilon_4)$ such that, for all $\varepsilon \in (0,\varepsilon_3)$,
\begin{equation}\label{eq:sum-bar-pi-RST-bounds}
\max_{\edg\in \RST(\gh)}e^{-\frac{1}{\varepsilon}(+\delta/4+\sum_{e\in \edg}\tqpd(e))} \leq \sum_{\edg\in \RST(\gh)}\bar{\pi}(\edg)\leq \max_{\edg\in \RST(\gh)}e^{-\frac{1}{\varepsilon}(-\delta/4+\sum_{e\in \edg}\tqpd(e))}.
\end{equation}
These are the desired bounds on $\sum_{\edg\in \RST(\gh)}\bar{\pi}(\edg)$.

We now seek to establish bounds on $\sum_{\edg\in \CRST(\gh)}\bar{\pi}(\edg)\cfo(\cycle(\edg))$. 
Since $C_0\setminus C_N$ is a finite set, the assumption \eqref{eq:min-rst-assumption} and our choice of $N$ imply the existence of $\varepsilon_2\in (0,\varepsilon_3)$ such that, for all $\varepsilon \in (0,\varepsilon_2)$,
\begin{equation}\label{eq:C0-CN-bounds}
\begin{split}
\max_{\substack{\edg\in C_0\setminus C_N\\\cfo(\cycle(\edg))> 0}}\,\,\,e^{-\frac{1}{\varepsilon}(\delta/5+\sum_{e\in \edg}\tqpd(e))} &\leq \sum_{\edg\in C_0\setminus C_N}\bar{\pi}(\edg)\cfo(\cycle(\edg)) \\
&\leq \max_{\substack{\edg\in C_0\setminus C_N\\\cfo(\cycle(\edg))> 0}}\,\,\,e^{-\frac{1}{\varepsilon}(-\delta/5+\sum_{e\in \edg}\tqpd(e))}.
\end{split}
\end{equation}
Now suppose we can find $\varepsilon_1\in (0,\varepsilon_2)$ such that, for all $\varepsilon \in (0,\varepsilon_1)$,
\begin{equation}\label{eq:CN-bound-desired}
\left|\sum_{\edg\in C_N}\bar{\pi}(\edg)\cfo(\cycle(\edg))\right|\leq e^{-\frac{1}{\varepsilon}(N-\delta/4)}.
\end{equation}
Then since $\CRST(\gh) = (C_0\setminus C_N) \cup C_N$, \eqref{eq:C0-CN-bounds}, \eqref{eq:CN-bound-desired}, and our choice of $N$ would imply the existence of $\varepsilon_0\in (0,\varepsilon_1)$ such that, for all $\varepsilon \in (0,\varepsilon_0)$, 
\begin{equation}\label{eq:CRST-bounds-desired}
\begin{split}
\max_{\substack{\edg\in \CRST(\gh)\\\cfo(\cycle(\edg))> 0}}\,\,\,e^{-\frac{1}{\varepsilon}(\delta/3+\sum_{e\in \edg}\tqpd(e))} &\leq \sum_{\edg\in \CRST(\gh)}\bar{\pi}(\edg)\cfo(\cycle(\edg)) \\
&\leq \max_{\substack{\edg\in \CRST(\gh)\\\cfo(\cycle(\edg))> 0}}\,\,\,e^{-\frac{1}{\varepsilon}(-\delta/3+\sum_{e\in \edg}\tqpd(e))}.
\end{split}
\end{equation}
From \eqref{eq:flux-manifold-mc-CRST-ld-1} we see this would imply $\fluxe([\cfo])> 0$ for all $\varepsilon\in (0,\varepsilon_0)$, and substituting the bounds \eqref{eq:sum-bar-pi-RST-bounds} and \eqref{eq:CRST-bounds-desired} into \eqref{eq:flux-manifold-mc-CRST-ld-1} would yield that, for all $\varepsilon \in (0,\varepsilon_0)$, 
\begin{equation}\label{eq:flux-sde-mc-ld-delta-bound}
\begin{split}
\left|-\varepsilon\ln\fluxe([\cfo]) - \left(\min_{\substack{\edg\in \CRST(\gh)\\\cfo(\cycle(\edg))> 0}}\,\,\,\sum_{e\in \edg} \tqpd(e)\right) + \left(\min_{\substack{\edg \in \RST(\gh)}}\sum_{e\in \edg}\tqpd(e)\right)\right|\leq \delta.
\end{split}
\end{equation}
Since $\delta > 0$ was arbitrary, \eqref{eq:flux-sde-mc-ld-delta-bound} would imply the desired Eq.~\eqref{eq:flux-sde-mc-ld-expression}.

To complete the proof it therefore remains only to show that there exists $\varepsilon_1\in(0,\varepsilon_2)$ such that \eqref{eq:CN-bound-desired} holds for all $\varepsilon\in (0,\varepsilon_1)$.
For any $k\in \N$ and $\edg \in C_{kN}\setminus C_{(k+1)N}$, Eq.~\eqref{eq:univ-cover-qp-growth-3} of Lem.~\ref{lem:univ-cover-qp-growth} implies that $$|\cfo(\cycle(\edg))|\leq \#(\vtx)\max_{e\in \edg}|\cfo(e)|\leq \frac{(k+1)N+c}{m}\#(\vtx),$$
where $m,c>0$ are the constants defined in Lem.~\ref{lem:univ-cover-qp-growth}.
Therefore,
\begin{align*}
\left|\sum_{\edg\in C_N}\bar{\pi}(\edg)\cfo(\cycle(\edg))\right|&= \left|\sum_{k=1}^\infty\,\,\sum_{\edg\in C_{kN}\setminus C_{(k+1)N}}\bar{\pi}(\edg)\cfo(\cycle(\edg))\right|\\
&\leq \frac{\#(\vtx)}{m}\left(c\sum_{k=1}^\infty\sum_{\edg\in 2^{\Eh}_{kN}} \bar{\pi}(\edg) + N\sum_{k=1}^\infty (k+1)\sum_{\edg\in 2^{\Eh}_{kN}}\bar{\pi}(\edg)\right)\\
&\leq \frac{\#(\vtx)}{m}\left(c\sum_{k=1}^\infty  e^{-\frac{1}{\varepsilon}k(N-\delta/5)}  + N\sum_{k=1}^\infty ke^{-\frac{1}{\varepsilon}k(N-\delta/5)}\right),
\end{align*}
where the final line follows from \eqref{eq:sum-rst-n-bound}.
The first sum in the last line is equal to $\frac{e^{-\frac{1}{\varepsilon}(N-\delta/5)}}{1-e^{-\frac{1}{\varepsilon}(N-\delta/5)}}$, and the second sum is equal to the derivative $\frac{e^{-\frac{1}{\varepsilon}(N-\delta/5)}}{(1-e^{-\frac{1}{\varepsilon}(N-\delta/5)})^2}$ of the first with respect to $-\frac{1}{\varepsilon}(N-\delta/5)$, so both sums are logarithmically equivalent to $e^{-\frac{1}{\varepsilon}(N-\delta/5)}$ as $\varepsilon \to 0$.
Thus, there indeed exists $\varepsilon_1\in (0,\varepsilon_2)$ such that \eqref{eq:CN-bound-desired} holds for all $\varepsilon\in (0,\varepsilon_1)$.
This completes the proof.
\end{proof}

\section{Proofs of Theorems~\ref{th:qualitative}, \ref{th:qualitative-measure} and Prop.~\ref{prop:one-minimizer}}\label{sec:proofs-morse}
In this section we prove Theorems~\ref{th:qualitative} and \ref{th:qualitative-measure} and Prop.~\ref{prop:one-minimizer}.
Properties of the quasipotential specific to the case that $\dft = \cfo^\sharp$ is dual to a closed one-form are established in \S \ref{sec:quasipot-closed-form}.
Results concerning continuous dependence with respect to $\dft$ of the quasipotential and of minimizing rooted spanning trees are established in \S \ref{sec:quasipotential-continuous-dependence} and \S\ref{sec:spanning-tree-continuous-dependence}, respectively. 
Using these tools and also Theorem~\ref{th:flux-manifold-mc-CRST-ld}, we complete the proof of Theorem~\ref{th:qualitative} in \S \ref{sec:finishing-morse-flux-proof}, and we also prove Theorem~\ref{th:qualitative-measure} and Prop.~\ref{prop:one-minimizer} in the same section.
\subsection{Properties of the quasipotential for tilted Morse-Smale potentials}\label{sec:quasipot-closed-form}
In the statements below, given $\dft = \cfo^\sharp$ satisfying Assumption~\ref{assump:morse-smale} and $e\in \Ed$, let $[e]\in \Pi(M)$ be the path homotopy class of any orientation-preserving parametrization of $e$.
Define $\Ws(e)$ and $\Wu(e)$ to be the stable and unstable manifolds of the unique index-$1$ zero of $\dft$ in $q(e)$, and note that $\Wu(e)=q(e)\in \Eu$.
If $\cfo = -dU$ is exact, then we additionally define $U(e)$ and $U(q(e))$ to be the value of $U$ at the unique index-$1$ critical point of $U$ in $q(e)$.

\begin{Lem}\label{lem:saddle-lowest-wrt-beta}
Let $M$ be a closed connected Riemannian manifold.
Assume that $\dft = \cfo^\sharp$ is dual to a $\Cont^1$ closed one-form and satisfies Assumption~\ref{assump:morse-smale}.
Fix $e\in \Ed$, define $W$ to be any neighborhood of $q(e)$ contained in $\Ws(\src(e))\cup \Ws(\tgt(e))\cup \Ws(e)$, and let $\vp\colon[T_1,T_2]\to W$ be any continuous path satisfying $[\vp]=[e]$.
Define $t_*\coloneqq \inf \{t\in [T_1,T_2]\colon \vp(t)\in \Ws(e)\}$ and let $z\in q(e)\in \Eu$ be the unique index-$1$ zero of $\dft$ in $q(e)\subset M$.
Then $$\int_{\vp|_{[T_1,t_*]}}(-\cfo) \geq g(e)$$
with equality if and only if $\vp(t_*)=z$.
\end{Lem}
\begin{proof}
Let $\pi\colon \tM\to M$ be the universal cover of $M$ equipped with the pullback metric, and let $f\in \Cont^2(\tM)$ be a primitive of $\pi^*(-\cfo)$ so that $\pi^*\cfo = -df$.
Fix $v\in \pi^{-1}(\src(e))$ and let $\tilde{\vp}\colon [T_1,T_2]\to \tM$ be the unique lift of $\vp$ to $\tM$ satisfying $\tilde{\vp}(T_1)=v$. 
Define $\tx\coloneqq \tilde{\vp}(t_*)$ and $\tilde{w}\coloneqq \tilde{\vp}(T_2)$.
Let $\tilde{z}$ be the unique index-$1$ critical point of $f$ such that $\tilde{x}\in \Ws(\tilde{z})$ belongs to the stable manifold of $\tilde{z}$ with respect to the flow of $-\nabla f$.

Then the claim of the lemma is equivalent to the statement that $f(\tilde{x})\geq f(\tilde{z})$ with equality if and only if $\tilde{x} = \tilde{z}$.
But this follows since $\tilde{x}\in \Ws(\tilde{z})$, a  stable manifold for an equilibrium of the flow of $-\nabla f$, and the unique minimum value of $f$ on such a stable manifold is always the value of $f$ at the equilibrium.
\end{proof}

\begin{Lem}\label{lem:qp-leq-g-W}
Let $M$ be a closed connected Riemannian manifold.
Assume that $\dft = \cfo^\sharp$ is dual to a $\Cont^1$ closed one-form and satisfies Assumption~\ref{assump:morse-smale}.
Fix $e\in \Ed$, define $W$ to be any neighborhood of $q(e)$ contained in $\Ws(\src(e))\cup \Ws(\tgt(e))\cup \Ws(e)$, and define
\begin{equation}\label{eq:qp-W}
\qpd^W([e])\coloneqq \inf \{\af(\vp)\colon [\vp]=[e], \image(\vp)\subset W\},
\end{equation}
where $\image(\vp)$ is the image of $\vp$ and the infimum is over all continuous paths of the form $\vp\colon [T_1,T_2]\to W$ with $[\vp]=[e]$.
Then
$$g(e) = \qpd^W([e])\geq \tqpd([e])\geq \qpd([e]).$$
\end{Lem}
\begin{proof}
Set $v=\src(e), w=\tgt(e)$, and let $z\in q(e)$ be the unique index-$1$ critical point of $\dft$ contained in $q(e)$.
Fix $\varepsilon > 0$.
By \cite[p.~143, Lem.~1.1]{freidlin2012random} there are arbitrarily small neighborhoods $B_v, B_z, B_w\subset W$ respectively of $v,z,w$ such that, for any $x,y$ contained in the same neighborhood, there is a smooth path $\vp$ from $x$ to $y$ in $W$ satisfying $\af(\vp)< \varepsilon/3$.
Let $\vp_v$ be such a path from $v$ to a nearby point $v'$ in the interior of $q(e)$, $\vp_w$ be such a path to $w$ from a nearby point $w'$ in the interior of $q(e)$, and $\vp_z$ be such a path from points $z', z''$ in the interior of $q(e)$ near $z$ such that $z'$ is in the component of $q(e)\setminus \{z\}$ containing $v'$ and $z''$ is in the component containing $w'$. 
Let $\vp_{vz}$ be the segment of the integral curve of $-\dft$ going from $v'$ to $z'$, and let $\vp_{zw}$ be the segment of the integral curve of $\dft$ going from $z''$ to $w'$.
Let $\vp$ be the concatenation $$\vp = \vp_{v}\cdot \vp_{vz}\cdot \vp_{z}\cdot \vp_{zw}\cdot \vp_w$$
of the paths. 
Then
\begin{align*}
\af(\vp)&=\af(\vp_v)+\af(\vp_{vz})+\af(\vp_z)+\af(\vp_{zw})+\af(\vp_w)\\
&\leq \af(\vp_{vz})+\af(\vp_{zw})+\varepsilon\\
&= \af(\vp_{vz})+\varepsilon,
\end{align*}
where $\af(\vp_{zw})=0$ by the definition of $\af(\slot)$ in \eqref{eq:action-functional} and the fact that $\dot{\vp}_{zw}\equiv \bv(\vp_{zw})$.
Since $\vp_{vz}$ is a segment of an integral curve of $\dft$, Lem.~\ref{lem:S-lower-bound} implies that $\af(\vp_{vz}) = \int_{\vp_{vz}}(-\cfo)$, and this integral converges to $g(e)$  as the diameters of the neighborhoods $B_v, B_z$ tend to zero.
Hence by choosing these neighborhoods small enough we may ensure that $\af(\vp_{vz})< g(e)+ \varepsilon$ and $\af(\vp)< g(e)+2\varepsilon$.
Since $\varepsilon$ was arbitrary and $\qpd^W([e])\leq \af(\vp)$, this implies that $$\qpd^W([e])\leq g(e).$$

To prove the reverse inequality, let $\vp\colon [T_1,T_2]\to W$ be any continuous path in $W$ with $[\vp]=[e]$ and define $t_*\coloneqq \inf\{t\in [T_1,T_2]\colon \vp(t)\in \Ws(z)\}$. 
Lem.~\ref{lem:S-lower-bound} and \ref{lem:saddle-lowest-wrt-beta} imply that 
\begin{equation}\label{eq:qp-leq-g-W-S-lb}
\af(\vp)\geq \int_{\vp|_{[T_1,t_*]}}(-\cfo) \geq g(e),
\end{equation}
and taking the infimum over all such paths $\vp$ yields the desired inequality
$$\qpd^W([e])\geq g(e).$$
The inequality $\qpd^W([e])\geq \tqpd([e])$ follows since $W$ contains no index-$0$ zeros except $\src(e)$ and $\tgt(e)$, and the inequality $\tqpd([e])\geq \qpd([e])$ always holds.
\end{proof}

Let $N\subset M$ be a cooriented, codimension-$1$, $\Cont^1$ embedded submanifold (not necessarily compact). 
Given a $\Cont^1$ path $\vp \colon [T_1,T_2]\to M$ which is transverse to $N$, recall (for Lem.~\ref{lem:cycle-path-qpd-lower-bound}) that the \concept{oriented intersection number} $I(N,\vp)$ of $N$ with $\vp$ is defined by
\begin{equation}\label{eq:intersection-number}
I(N,\vp)\coloneqq \sum_{t\in \vp^{-1}(N)} \varepsilon_t,
\end{equation}
where $\varepsilon_t = \pm1$ if the coorientation of $N$ induced by $\dot{\vp}(t)$ is $\pm$ the given coorientation of $N$.
(Note that $\vp(T_1)$ and/or $\vp(T_2)$ may belong to $N$, and such boundary points do contribute to the sum in \eqref{eq:intersection-number}.)

\begin{Lem}\label{lem:cycle-path-qpd-lower-bound}
Let $M$ be a closed connected Riemannian manifold.
Assume that $\dft = \cfo^\sharp$ is dual to a closed but not exact $\Cont^1$ one-form, satisfies Assumption~\ref{assump:morse-smale}, and is such that the minimizer $T_*$ in \eqref{eq:morse-rst-minimizer} is unique.
Equip
\begin{align}
N\coloneqq \bigcup_{\substack{e\in \Ed\\ h(e)<h(\bar{e})}}\Ws(e)
\end{align}
with the coorientation induced by the directed edges $e\in \Ed$ satisfying $h(e) < h(\bar{e})$.
Then there exists $k>0$ such that the following holds.
For any $\Cont^1$ path $\vp\colon [T_1,T_2] \to M$ satisfying (i) $[\vp]\in \Eh$, (ii) $\vp$ is transverse to $\bigcup_{z\in \dft^{-1}(0)}\Ws(z)$, (iii) $I(N,\vp) < 0$, and (iv)
\begin{align}\label{eq:t1-t2-def-eq}
t_1&\coloneqq \inf\{t\geq T_1\colon \vp(t) \in N\}
< t_2\coloneqq \inf\{t\geq t_1\colon I(N,\vp|_{[T_1, t]}) < 0\}
\end{align}
it follows that, if $e_1\in \Ed$ is such that $\vp(t_1)\in \Ws(e_1)$ and $h(e_1)< h(\bar{e}_1)$, then
\begin{equation}\label{eq:lem-cycle-path-qpd-ineq-0}
h(e_1) + k < h(\vp(0)) + \af(\vp|_{[T_1, t_2]}).
\end{equation}
\end{Lem}
\begin{proof}
Let $\vp$, $t_1$, and $t_2$ be as in the statement, $\pi\colon \tM\to M$ be the universal cover equipped with the pullback metric, $\tdft$ be the lift of $\dft$ to $\tM$, and $f\in \Cont^2(\tM)$ satisfy $df = \pi^*(-\cfo)$ and $f(\vp(0)) = h(\vp(0))$.
Define $v_1\coloneqq \vp(T_1)\in \vtx$, fix $w_1\in \pi^{-1}(v_1)$, and let $\tilde{\vp}$ be the unique lift of $\vp$ satisfying $\tilde{\vp}(T_1) = w_1$.
Let $z_1, z_2\in \tdft^{-1}(0)$ be such that $\tilde{\vp}(t_i)$ belongs to the stable manifold $\Ws(z_i)$ of $z_i$ with respect to $\tdft$. 
Conditions (ii--iv) from the statement imply that $z_1 \neq z_2$.

Lem.~\ref{lem:S-lower-bound} implies the first inequality below, where $\af^{\tdft}$ and $\qptd(\slot,\slot)$ are respectively the action and quasipotential associated to $\tdft$:
\begin{equation}\label{eq:lem-cycle-path-qpd-ineq-1}
\begin{split}
h(\vp(0)) + \af(\vp|_{[T_1,t_2]}) & = f(\tilde{\vp}(0)) + \af^{\tdft}(\tilde{\vp}|_{[T_1,t_1]}) + \af^{\tdft}(\tilde{\vp}|_{[t_1,t_2]})\geq f(\tilde{\vp}(t_1)) + \af^{\tdft}(\tilde{\vp}|_{[t_1,t_2]})\\
&\geq f(\tilde{\vp}(t_1)) + \qptd(\tilde{\vp}(t_1),\tilde{\vp}(t_2))= f(\tilde{\vp}(t_1)) + \qpd([\vp|_{[t_1,t_2]}])\\  
&> f(\tilde{\vp}(t_1)) \geq h(e_1).
\end{split}
\end{equation} 
The second equality follows since $\qptd(x,y) = \qpd([\pi \circ \psi])$ for any continuous path $\psi$ in $\tM$ from $x$ to $y$, and the strict inequality follows from Prop.~\ref{prop:qp-int-conditions-iff-connecting-piecewise-orbit} since $z_1\neq z_2$ implies that there is no piecewise $\dft$-integral curve in $[\vp|_{[t_1,t_2]}]$.
From the continuity of $\qptd$ \cite[p.~143,~Lem.~1.1]{freidlin2012random}, invariance of $\qptd$ under deck transformations, Lem.~\ref{lem:univ-cover-qp-growth}, and \eqref{eq:lem-cycle-path-qpd-ineq-1} it follows that the restriction of the map $(x,y)\mapsto f(x) + \qptd(x,y)$ to $$\bigcup_{z\in \pi^{-1}(\pi(z_1))} \Ws(z)\times (\pi^{-1}(N)\setminus \Ws(z))$$
attains a minimum $h(e_1) + k_{\pi(z_1)} > h(e_1)$. 
Defining $k\coloneqq \min\{k_{z}\colon z\in \dft^{-1}(0) \textnormal{ and } \ind(z) = 1\}$ and using the fact that $\dft^{-1}(0)$ is a finite set (hence $k>0$), this and \eqref{eq:lem-cycle-path-qpd-ineq-1} imply \eqref{eq:lem-cycle-path-qpd-ineq-0}.
\end{proof}

\begin{Rem}
The equality in \eqref{eq:C-AB-def} below holds from taking $X=M$, $Y = \bigcup_{a\in A}\Ws(a)$, and $Z = \bigcup_{b\in B}\Ws(b)$ in the following general fact about topological spaces.
If (i) $X = \cl(Y) \cup \cl(Z)$ where (ii) $Y\cap \cl(Z) = \varnothing$ and (iii) $\cl(Y) \cap Z = \varnothing$, then $\partial (\cl(Y)) = \partial (\cl(Z))$.
This in turn follows since (i) implies  $\partial(\cl(Y))\subset \cl(Z)$ while (ii) implies $\partial(\cl(Y))\subset \cl(Y) \subset X\setminus \interior(\cl(Z))$, so $\partial(\cl(Y))\subset \cl(Z)\setminus \interior(\cl(Z))= \partial(\cl(Z))$, and the reverse argument using (iii) instead of (ii) yields the reverse inclusion $\partial(\cl(Y))\supset \partial(\cl(Z))$.
\end{Rem}
\begin{Lem}\label{lem:exact-form-graph-cond}
Let $M$ be a closed connected Riemannian manifold.
Assume that $\dft = -\nabla U$ is dual to a $\Cont^1$ exact one-form and is Morse-Smale.
Fix $v\in \vtx$ and let $T_v\in \RST(\gh;v)$ be any (possibly nonunique) minimizer
\begin{equation}\label{eq:lem-exact-form-graph-cond-FW-min}
T_v \in \arg\min_{T\in \RST(\gh;v)}\sum_{e\in T}\qpd(e).
\end{equation}
Then for every $e\in T_v$ the following hold, where $A,B\subset V$ are the (undirected) vertex components of $T_v\setminus \{e\}$,
\begin{equation}\label{eq:C-AB-def}
C_{\{A,B\}}\coloneqq \partial \left[\cl \bigcup_{a\in A}\Ws(a)\right] =  \partial \left[\cl \bigcup_{b\in B}\Ws(b)\right],
\end{equation}
and $H_{\{A,B\}} \coloneqq \min_{y\in C_{\{A,B\}}}U(y)$:
\begin{enumerate}
\item\label{item:e-eq-e'} there exists $e'\in \Ed$ such that $e = [e']$,
\item\label{item:qpd-eq-g} $\qpd(e) = g(e')=H_{\{A,B\}}-U(\src(e'))=H_{\{A,B\}}-U(\src(e))$, and
\item\label{item:HAB-strict-lower-bound} if $x\in \Ws(\src(e'))\setminus q(e')$, then $$\qpd(x,\tgt(e))> H_{\{A,B\}}-U(x) = g(e') + U(\src(e'))-U(x).$$
\end{enumerate}
\end{Lem}
\begin{proof}
First consider \emph{any} partition $V = A\cup B$, define $C_{\{A,B\}}$ according to \eqref{eq:C-AB-def}, and define the ``barrier height'' $H_{\{A,B\}}\coloneqq \min_{y\in C_{\{A,B\}}}U(y)$.
Note that, by the Morse-Smale assumption, $H_{\{A,B\}}$ is the smallest value of $U(e)$ where $e\in \Eu$ ranges over those undirected Morse edges having one end in $A$ and one end in $B$; such edges exist by the Morse-Smale assumption and connectedness of $M$.
Lem.~\ref{lem:S-lower-bound} implies that, for  any continuous path $\vp\colon [T_1,T_2]\to M$ from $\cl (\bigcup_{a\in A}\Ws(a))$ to $\cl (\bigcup_{b\in B}\Ws(b))$,  
$$U(\vp(T_1))+\af(\vp)\geq \max_{t\in [T_1,T_2]}U(\vp(t))\geq H_{\{A,B\}}.$$
Taking the infimum over all such paths $\vp$ additionally satisfying $\vp(T_2)\in C_{A,B}$, followed by the same reasoning with the roles of $A$ and $B$ reversed yields, for all $x\in M$ and $y\in C_{\{A,B\}}$, 
\begin{equation}\label{eq:U-qpd-x-y-lower-bound}
U(x)+ \qpd(x,y)\geq H_{\{A,B\}}.
\end{equation}
Since $H_{\{A,B\}}$ is the lowest value of $U$ on $C_{\{A,B\}}$, the Morse-Smale assumption implies that any $y\in C_{\{A,B\}}$ satisfying $U(y) = H_{\{A,B\}}$ must satisfy $y\in q(e')$ for some $e'\in \Ed$.
The Morse-Smale assumption further implies that the only piecewise $(-\dft)$-integral curves terminating at such a $y\in q(e')$ must be contained in $q(e')$. 
Thus, Prop.~\ref{prop:qp-int-conditions-iff-connecting-piecewise-orbit} implies that equality holds in \eqref{eq:U-qpd-x-y-lower-bound} if and only if there is $e'\in \Ed$ with $x,y\in q(e')$ and $U(\src(e')) + g(e') = H_{\{A,B\}}$.

The same reasoning used to derive the equality characterization of \eqref{eq:U-qpd-x-y-lower-bound} implies that, for $e\in \Eh$ with ends in distinct elements of the partition $\{A,B\}$,
\begin{equation}\label{eq:U-qpd-e-lower-bound}
U(\src(e))+ \qpd(e)\geq H_{\{A,B\}}
\end{equation}
with equality if and only if there exists $e'\in \Ed$ satisfying $e =[e']$ and $U(\src(e))+g(e) = H_{\{A,B\}}$.
It follows that equality holds for $e\in \Eh$ if and only if equality holds for its reversal $\bar{e}\in \Eh$.

In the rest of the proof we use the following notation.
Let $\ST(\gh)\supset \RST(\gh)$ denote those subgraphs in $2^{\Eh}$ which are spanning trees in the undirected sense, i.e., those subgraphs which can be turned into rooted spanning trees by reversing the orientations of some edges.

We now prove the lemma using the notation introduced in its statement.
Assume that at least one of \ref{item:e-eq-e'} or  \ref{item:qpd-eq-g} fails for some $e_0\in T_v$ and let $A,B\subset V$ be the vertex components of $T_v\setminus \{e_0\}$.
In particular, failure of either \ref{item:e-eq-e'} or \ref{item:qpd-eq-g}, together with \eqref{eq:U-qpd-e-lower-bound} and its equality characterization, imply the \emph{strict} inequality 
\begin{equation}\label{eq:U-strict-lower-bound}
U(\src(e_0))+\qpd(e_0) > H_{\{A,B\}}.
\end{equation}
Fix an undirected edge $e_{\{A,B\}}\in \Eu$ with one end in $A$ and one in $B$ satisfying $U(e_{\{A,B\}})=H_{\{A,B\}}$.
Let $e_1\in q^{-1}(e_{\{A,B\}})\subset \Ed$ be either of the two directed copies of $e_{\{A,B\}}$.  
Since 
\begin{equation}\label{eq:lem-exact-form-double-eq}
\begin{split}
U(\src(e_1))+g(e_1) = H_{\{A,B\}} = U(\src(\bar{e}_1)) + g(\bar{e}_1),
\end{split}
\end{equation}  
\eqref{eq:U-strict-lower-bound} and the equality characterization of \eqref{eq:U-qpd-e-lower-bound} imply that
\begin{equation}\label{eq:lem-exact-form-double-ineq}
\begin{split}
U(\src(e_0))+\qpd(e_0)&> U(\src(e_1))+g(e_1) = U(\src(e_1))+\qpd([e_1])\\
U(\src(\bar{e}_0))+\qpd(\bar{e}_0)&> U(\src(\bar{e}_1))+g(\bar{e}_1) = U(\src(\bar{e}_1))+\qpd([\bar{e}_1]).
\end{split}
\end{equation}
The first line of \eqref{eq:lem-exact-form-double-ineq} implies that
\begin{align*}
\sum_{e\in T_v}U(\src(e))+\qpd(e) &>\sum_{e\in (T_v\setminus \{e_0\})\cup \{[e_1]\}}U(\src(e))+\qpd(e).
\end{align*}
Thus, $\edg' \coloneqq (T_v\setminus \{e_0\})\cup \{[e_1]\}\in \ST(\gh)$ satisfies $\sum_{e\in \edg'}U(\src(e)) + \qpd(e)< \sum_{e\in T_v}U(\src(e))+\qpd(e)$.

Repeating the above procedure for any edge $e_0\in \edg'$ for which at least one of \ref{item:e-eq-e'} or  \ref{item:qpd-eq-g} fails, we obtain by finite induction some $\edg\in \ST(\gh)$ for which 
\begin{equation}\label{eq:lem-exact-adjusted-cost-ineq}
\sum_{e\in T_v}U(\src(e))+\qpd(e) > \sum_{e\in \edg} U(\src(e))+ \qpd(e),
\end{equation}
and such that every edge in $\edg$ is of the form $[e_1]$ for some $e_1\in \Ed$ satisfying \eqref{eq:lem-exact-form-double-eq} and \eqref{eq:lem-exact-form-double-ineq}.
Since $\edg\in \ST(\gh)$, there is a unique $T\in \RST(\gh;v)$ which is obtained from reversing orientations of some edges in $\edg$.
Since satisfaction of \eqref{eq:lem-exact-form-double-eq} and \eqref{eq:lem-exact-form-double-ineq} is invariant under edge reversal, every edge in $T$ is of the form $[e_1]$ where $e_1\in \Ed$ satisfies \eqref{eq:lem-exact-form-double-eq} and \eqref{eq:lem-exact-form-double-ineq}.
Defining $C_v\coloneqq U(v)-\sum_{v'\in \vtx}U(v')$, we obtain (with justification after):
\begin{equation}\label{eq:lem-exact-punchline}
\begin{split}
\sum_{e\in T_v}\qpd(e) &= C_v + \sum_{e\in T_v}U(\src(e)) + \qpd(e)> C_v + \sum_{e\in \edg}U(\src(e))+\qpd(e)\\
&= C_v + \sum_{e\in T}U(\src(e)) + \qpd(e) = \sum_{e\in T}\qpd(e).
\end{split}
\end{equation}
The first equality is immediate from the definition of $C_v$ and the fact that $T_v\in \RST(\gh;v)$, the inequality follows from \eqref{eq:lem-exact-adjusted-cost-ineq}, the second equality follows from the fact that each $[e_1]\in T$ satisfies \eqref{eq:lem-exact-form-double-eq} and \eqref{eq:lem-exact-form-double-ineq}, and the final equality is immediate from the definition of $C_v$ and the fact that $T\in \RST(\gh;v)$.
The strict inequality in \eqref{eq:lem-exact-punchline} contradicts the assumption that $T_v\in \RST(\gh;v)$ minimizes the sum in \eqref{eq:lem-exact-form-graph-cond-FW-min}, so every edge $e\in T_v$ must satisfy both conditions \ref{item:e-eq-e'} and \ref{item:qpd-eq-g}.
Together with the equality characterization of \eqref{eq:U-qpd-x-y-lower-bound}, this establishes \ref{item:HAB-strict-lower-bound} and completes the proof.
\end{proof}

\subsection{A continuity property of the quasipotential}\label{sec:quasipotential-continuous-dependence}
For the proof of Theorem~\ref{th:qualitative} we need to consider the dependence of the quasipotential $\qpd(\slot)$ on the vector field $\dft$.
We are not aware of any existing results concerning this question of parameter dependence, so we state the result we need in this section and defer the technical proof to App.~\ref{app:proofs}.

\begin{Rem}[Smooth manifold structure on $\Pi(M)$]\label{rem:manifold-structure-fundamental-groupoid}
In formulating Prop.~\ref{prop:qp-continuity} and the results of \S \ref{sec:spanning-tree-continuous-dependence}, we use the topology on $\Pi(M)$ described as follows.
When given the discrete topology, the deck transformation group $\Aut(\pi)$ of the universal cover $\pi\colon \tM\to M$ of a smooth manifold $M$ acts smoothly, freely, and properly on $\tM$, so the diagonal action of $\Aut(\pi)$ on $\tM\times \tM$ (defined by $h\cdot (x,y) = (h\cdot x, h\cdot y)$ for $h\in \Aut(\pi)$ and $x,y\in \tM$) is also smooth, free, and proper.
Thus, the quotient manifold theorem implies that $(\tM\times \tM)/\Aut(\pi)$ has a unique smooth manifold structure of dimension $2\cdot\dim(M)$ such that the quotient map $\tM\times \tM \to (\tM\times \tM)/\Aut(\pi)$ is a smooth covering map \cite[Thm~21.13]{lee2013smooth}.
Next, the lifting properties of covering maps \cite[Prop.~A.77]{lee2013smooth} imply that the map $F\colon \tM\times \tM\to \Pi(M)$ sending $(x,y)$ to $[\pi\circ \vp]$ is surjective, where $\vp$ is any continuous path from $x$ to $y$, and that $F(x_1,y_1)=F(x_2,y_2)$ if and only if $ (x_1,y_1)= h\cdot (x_2,y_2)$ for some $h\in \Aut(\pi)$.
Thus, $F$ descends to a bijection $\bar{F}\colon (\tM\times \tM)/\Aut(\pi)\to \Pi(M)$, so we may define a topology and smooth manifold structure on $\Pi(M)$ by declaring that $\bar{F}$ is a diffeomorphism.
\end{Rem}

\newcommand{\cN}{\mathcal{N}}

\begin{restatable}[]{Prop}{PropQpContMS}\label{prop:qp-continuity}
Denote by $\vf^1(M)$ the space of $\Cont^1$ vector fields on the closed Riemannian manifold $M$ equipped with the $\Cont^1$ topology.
Let $\Pi(M)$ have the topology induced by its bijection with the smooth manifold $(\tM \times \tM)/\Aut(\pi)$, where $\pi\colon \tM\to M$ is the universal cover and the deck transformation group $\Aut(\pi)$ acts diagonally on $\tM\times \tM$. 
Let $\dft_0\in \vf^1(M)$ be a Morse-Smale vector field without nonstationary periodic orbits.
Then for any $e_0\in \Pi(M)$, the map
$$(\dft,e)\in \vf^1(M)\times \Pi(M) \mapsto \qpd(e)\in [0,+\infty) \quad \textnormal{is continuous at $(\dft_0,e_0)$.}$$
\end{restatable}

\begin{Rem}
We do not know if lower semicontinuity at $(\dft_0,e_0)$ holds if $\dft_0\in \vf^1(M)$ is arbitrary.
\end{Rem}

\subsection{Continuity properties of the minimal spanning trees}\label{sec:spanning-tree-continuous-dependence}
In this section we study the minimizing path-homotopical and Morse rooted spanning trees of vector fields dual to closed one-forms close to a generic exact one-form. 
These results will be used in the proof of Theorem~\ref{th:qualitative}.

\begin{Lem}\label{lem:near-assump-2-implies-assump-2}
Assume that $\dft$ is a $\Cont^1$ vector field on a closed connected Riemannian manifold $M$ and satisfies Assumption~\ref{assump:morse-smale}.
Then $\dft$ has a neighborhood $\cN_{\dft}\subset \vf^1(M)$ in the space of $\Cont^1$ vector fields $\vf^1(M)$ equipped with the $\Cont^1$ topology such that, for every $\bu\in \cN_{\dft}$, $\bu$ satisfies Assumption~\ref{assump:morse-smale}.
Moreover, for each index-$0$ zero $v$ of $\dft$ and any sequence $(\bu_n)\subset \cN_{\dft}$ converging to $\dft$ in $\vf^1(M)$, there is a unique sequence $(v_n)$ of index-$0$ zeros of $\bu_n$ close to $v$ with $v_n\to v$.     
\end{Lem}
\begin{proof}
This follows from (i) the $C^1$-openness \cite[Thm~3.5]{palis1968ms} and structural stability \cite[Thm~5.2]{palis1968structural} of Morse-Smale vector fields without nonstationary periodic orbits and (ii) the implicit function theorem (the zeros of $\dft$ are hyperbolic).
\end{proof}

\begin{Lem}\label{lem:converging-FW-trees}
Assume that $\dft$ is a $\Cont^1$ vector field on a closed connected Riemannian manifold $M$ and satisfies Assumption~\ref{assump:morse-smale}.
Let $\cN_{\dft}$ be as in Lem.~\ref{lem:near-assump-2-implies-assump-2} and $\dft_n\in \cN_{\dft}$ be a sequence of $\Cont^1$ vector fields converging to $\dft$ in the $\Cont^1$ topology, with associated path-homotopical graphs $\gh^{n} = (\Eh^{n},\vtx^{n},\src,\tgt)$.
Fix $v\in \vtx$ and let $v_{n}\in \vtx^{n}$ be the unique sequence (Lem.~\ref{lem:near-assump-2-implies-assump-2}) close to $v$ with $v_n\to v$.
Fix any (possibly nonunique)
\begin{equation}\label{eq:lem-converging-FW-trees-FW-min}
T_{v_n}^{n} \in \arg\min_{T\in \RST(\gh^{n};v_n)}\sum_{e\in T}\qp_{\dft_n}(e).
\end{equation}
Then 
\begin{equation}\label{eq:lem-converging-FW-trees}
T_{v_n}^{n} \to \left( \arg\min_{T\in \RST(\gh;v)} \sum_{e\in T}\qpd(e)\right)\subset [\Pi(M)]^{\#(\vtx)-1}
\end{equation}
in the topology on the Cartesian product $[\Pi(M)]^{\#(\vtx)-1}$ induced by the bijection of $\Pi(M)$ with the smooth manifold $(\tM \times \tM)/\Aut(\pi)$, where $\Aut(\pi)$ is the deck transformation group of the universal cover $\pi\colon \tM\to M$ acting diagonally on $\tM\times \tM$ (Rem.~\ref{rem:manifold-structure-fundamental-groupoid}).
\end{Lem}
\begin{proof}
Define
\begin{equation}\label{eq:lem-FW-min}
C\coloneqq \min_{T\in \RST(\gh;v)}\sum_{e\in T}\qpd(e),
\end{equation}
and fix $\varepsilon \in (0,1)$.
It follows from Prop.~\ref{prop:qp-continuity} that, for all sufficiently large $n$,
\begin{equation}\label{eq:lem-converging-FW-trees-u-ub}
\min_{T\in \RST(\gh^{n};v_n)}\sum_{e\in T}\qp_{\dft_n}(e) < C + \varepsilon.
\end{equation}

There is an embedding $M\hookrightarrow \Pi(M)$ which sends each $x\in M$ to the path homotopy class of the constant path at $x$; let $M_0\subset \Pi(M)$ denote the image of $M$ via this embedding.  
Given $\kappa > 0$, define the neighborhood
\begin{equation*}
\begin{split}
W_\kappa\coloneqq \{e\in \Pi(M)\colon \length(e)\leq \kappa\}
&= \{(x,y)\in K\times \tM \colon \dist{x}{y}\leq \kappa\}/\Aut(\pi)
\end{split}
\end{equation*}
of $M_0$, where $K\subset \tM$ is any compact set satisfying $\pi(K)=M$.
The pullback metric on $\tM$ is complete since $M$ is compact \cite[p.~146, Thm~2.8]{docarmo1992riemannian}, so the second expression shows that $W_{\kappa}$ is compact.

Lem.~\ref{lem:univ-cover-qp-growth} implies that there is $\kappa > 0$ such that
\begin{equation}\label{eq:lem-converging-FW-trees-lb-bndry}
\qpd|_{\partial W_{\kappa}} > C+1,
\end{equation}
so compactness of $\partial W_{\kappa}$ and Prop.~\ref{prop:qp-continuity} imply that, for all sufficiently large $n$,
\begin{equation}\label{eq:qpvn-lower-bound}
\qp_{\dft_n}|_{\partial W_{\kappa}} > C+1.
\end{equation}
Since $W_\kappa$ is a neighborhood of $M_0$ it follows that, for any continuous path $\vp\in \Cont_e([0,T],M)$ representing any $e\in \Pi(M)\setminus W_{\kappa}$, the continuous path $t\in [0,T]\mapsto [\vp|_{[0,t]}]\in \Pi(M)$ from $M_0$ to $e$ must pass through $\partial W_\kappa$.
Hence \eqref{eq:lem-converging-FW-trees-lb-bndry} and \eqref{eq:qpvn-lower-bound} imply that, for all sufficiently large $n$,
\begin{equation}\label{eq:lem-converging-FW-trees-lb}
\qpd|_{\Pi(M)\setminus W_\kappa} > C+1 \qquad \textnormal{and} \qquad \qp_{\dft_n}|_{\Pi(M)\setminus W_\kappa} > C+1.
\end{equation}
From this, \eqref{eq:lem-FW-min}, and \eqref{eq:lem-converging-FW-trees-u-ub} it follows that every edge of every minimizer in $\RST(\gh^{n};v_n)$ and in $\RST(\gh;v)$ of the sums in \eqref{eq:lem-converging-FW-trees-FW-min} and \eqref{eq:lem-FW-min}, respectively, must belong to $W\coloneqq W_\kappa$, so these minimizing trees must belong to the compact Cartesian product $W^{N}\subset [\Pi(M)]^{N}$ of $W$ with itself $N\coloneqq \#(\vtx)-1$ times.

It follows that any convergent subsequence $(T_{v_{n_k}}^{n_k})\subset \RST(\gh^n;v_n)$ converges to some $\bar{T}\in W^N$, so
\begin{equation}\label{eq:lem-converging-FW-final}
\begin{split} 
\sum_{e\in \bar{T}}\qpd(e) &= \lim_{n\to \infty} \sum_{e\in T_{v_{n_k}}^{n_k}}\qp_{\dft_{n_k}}(e) = \lim_{n\to \infty} \min_{T\in \RST(\gh^{n_k};v_{n_k})}\sum_{e\in T}\qp_{\dft_{n_k}}(e) \leq C\\
& = \min_{T\in \RST(\gh;v)}\sum_{e\in T}\qpd(e).
\end{split}
\end{equation}
The first equality follows from Prop.~\ref{prop:qp-continuity}, the second equality is immediate from the definition of $T_{v_{n_k}}^{n_k}$, the inequality follows from the arbitrariness of $\varepsilon$ in \eqref{eq:lem-converging-FW-trees-u-ub}, and the final equality follows from \eqref{eq:lem-FW-min}.
Since $\vtx^n\to \vtx$ by the implicit function theorem and the assumption that the finitely many zeros of $\dft$ are hyperbolic, it follows that $\bar{T}\in \RST(\gh;v)$, so \eqref{eq:lem-converging-FW-final} implies that $\bar{T} \in A\coloneqq \arg\min_{T\in \RST(\gh;v)} \sum_{e\in T}\qpd(e)\subset W^N$.
Since $W^N$ is compact and since \eqref{eq:lem-converging-FW-final} shows that every convergent subsequence of $(T^n_{v_n})\subset W^N$ converges to $A$, it follows that $T^n_{v_n}\to A$.
This establishes \eqref{eq:lem-converging-FW-trees} and completes the proof.
\end{proof}

\begin{Lem}\label{lem:converging-FW-trees-W}
Assume that $\dft = \cfo^\sharp = -\nabla U$ is dual to a $\Cont^1$ exact one-form on a closed connected Riemannian manifold $M$ and satisfies Assumption~\ref{assump:morse-smale}.
Let $\cN_{\dft}$ be as in Lem.~\ref{lem:near-assump-2-implies-assump-2} and $\dft_n = \cfo_n^\sharp\in \cN_{\dft}$ be a sequence of vector fields dual to $\Cont^1$ closed one-forms converging to $\dft$ in the $\Cont^1$ topology.
Let $\gh^{n} = (\Eh^{n},\vtx^{n},\src,\tgt)$ and $\compd^n = (\Ed^n,\vtx^n, \src, \tgt)$ denote the associated path-homotopical and directed Morse graphs, respectively.
Fix $v\in \vtx$ and let $v_{n}\in \vtx^{n}$ be the unique sequence of Lem.~\ref{lem:near-assump-2-implies-assump-2} with $v_n\to v$.
For each $n$ fix
\begin{equation}\label{eq:lem-converging-FW-trees-FW-min-W}
T_{v_n}^{n} \in \arg \min_{\edg\in \RST(\gh^{n};v_n)}\sum_{e\in \edg}\qp_{\dft_n}(e).
\end{equation}
Then for all sufficiently large $n$ and $e_n\in T^n_{v_n}$, $e_n = [e_n']$ for some $e_n'\in \Ed^n$ and
$$\qp_{\dft_n}(e_n) = g(e_n').$$
\end{Lem}
\begin{proof}
Lem.~\ref{lem:converging-FW-trees} implies that $T^n_{v_n}$ converges to the set
\begin{equation}\label{eq:lem-converging-FW-trees-W-1}
\arg\min_{T\in \RST(\gh;v)} \sum_{e\in T}\qpd(e)
\end{equation}
in the topology described in that lemma, and Lem.~\ref{lem:exact-form-graph-cond} (or Lem.~\ref{lem:univ-cover-qp-growth}) implies that the set of minimizers in \eqref{eq:lem-converging-FW-trees-W-1} is finite.
Thus, we may partition the sequence $(T^n_{v_n})$ into subsequences converging to individual minimizers in \eqref{eq:lem-converging-FW-trees-W-1}, so we may and do henceforth assume that $T^n_{v_n}$ converges to some $T_v$ in \eqref{eq:lem-converging-FW-trees-W-1}.
Every edge in $T^n_{v_n}$ is then arbitrarily close (in the topology described in Rem.~\ref{rem:manifold-structure-fundamental-groupoid}) to a unique edge in $T_v$ for sufficiently large $n$, so we may and do henceforth assume that $e_n\to e\in T_v$.
Note that Lem.~\ref{lem:exact-form-graph-cond} implies the existence of $e'\in \Ed$ with $e=[e']$.
As in \S \ref{sec:quasipot-closed-form}, we abuse notation and denote by $\Ws(e')$ the stable manifold of the unique index-$1$ zero of $\dft$ in $q(e')$.

Let $W\subset M$ be a neighborhood of $q(e')$ with $\cl(W) \subset \Ws(\src(e'))\cup \Ws(\tgt(e'))\cup \Ws(e')$ and define $\qpd^W$, $\qp_{\dft_n}^W$ as in Lem.~\ref{lem:qp-leq-g-W}.
Since $\cl(q(e'))$ is an asymptotically stable invariant set for the flow of $\dft$, by taking $\cl(W)$ to be a sublevel set of a $\Cont^\infty$ Lyapunov function \cite{wilson1969smoothing,fathi2019smoothing} for $\cl(q(e'))$ we may assume that $\cl(W)$ is a smooth codimension-$0$ submanifold with boundary and that $\dft$ is strictly inward pointing at $\partial W$; since $\dft_n\to \dft$ it follows that
\begin{equation}\label{eq:inward-pointing}
\textnormal{$\dft_n$ is strictly inward pointing at $\partial W$ for all sufficiently large $n$.}
\end{equation} 
Since $\vtx^n\to \vtx$ by Lem.~\ref{lem:near-assump-2-implies-assump-2}, the properties of the topology on $\Pi(M)$ (Rem.~\ref{rem:manifold-structure-fundamental-groupoid}) imply that, if $e^1_n,e^2_n\in \Eh^n$ are two sequences of edges converging to the same edge in $\Eh$, then $e^1_n = e^2_n$ for all sufficiently large $n$. 
Since the $1$-dimensional unstable manifolds of $\dft_n$ converge uniformly to those of $\dft$ by the implicit function and stable manifold theorems \cite[p.~75,~Thm~6.2]{palis1980geometric}, there is a sequence $e_n'\in \Ed^n$ with $e_n'\to e'$ and hence $[e_n']\to [e']=e$, so the preceding sentence implies that $e_n = [e_n']$ for all sufficiently large $n$.
Let $z_n\in q(e_n')$ denote the unique index-$1$ zero of $\dft_n$ in $q(e_n')$.
For all sufficiently large $n$ we have $\{\src(e_n'), \tgt(e_n'), z_n\} \subset \interior(W)$ while the other zeros of $\dft_n$ are disjoint from $\cl(W)$, so \eqref{eq:inward-pointing} implies that, for all sufficiently large $n$, $$\cl(W)\subset \Ws(\src(e_n))\cup \Ws(\tgt(e_n))\cup \Ws(e_n).$$
Thus, Lem.~\ref{lem:qp-leq-g-W} implies that $\qp_{\dft_n}^W(e_n)=g(e_n')$ for large $n$, so to complete the proof it suffices to show that $\qp_{\dft_n}(e_n)=\qp_{\dft_n}^W(e_n)$ for all sufficiently large $n$.

Suppose not and define $A_n\coloneqq (\partial W)\cap \Ws(\src(e_n))$.
Then there exists a sequence $f_n\in \Eh^n$ such that $\src(f_n) = \src(e_n)$, $\tgt(f_n)\in A$, $\length(f_n)\leq \textnormal{diam}(W)$ (Def.~\ref{def:length}), and
\begin{equation}\label{eq:qp-fn-ineq}
\qp_{\dft_n}(e_n) = \qp_{\dft_n}(f_n) + \qp_{\dft_n}(\bar{f}_ne_n),
\end{equation}
where $\bar{f}_n$ denotes the reversal of $f_n$.
Since $\length(f_n)\leq \textnormal{diam}(W)$ for all $n$, after passing to a subsequence we may assume there is $f\in \src^{-1}(\src(e))$ satisfying 
\begin{equation}\label{eq:tgt-f-location}
\tgt(f)\in (\partial W) \cap \cl(\Ws(\src(e))
\end{equation}
such that $f_n\to f$ in the topology on $\Pi(M)$ described in Rem.~\ref{rem:manifold-structure-fundamental-groupoid}.
Prop.~\ref{prop:qp-continuity} and \eqref{eq:qp-fn-ineq} imply that 
\begin{equation}\label{eq:qpd-converged-equality}
\qpd(e) = \qpd(f) + \qpd(\bar{f}e).
\end{equation}
Set $x\coloneqq \src(e)$, $y\coloneqq \tgt(f)$, and define $z$ to be the unique index-$1$ zero of $\dft$ in $q(e')$.
Lem.~\ref{lem:S-lower-bound} implies  that $\qpd(f)\geq U(y)-U(x)$, and Lem.~\ref{lem:exact-form-graph-cond} implies that $\qpd(e)=g(e')$ and that the strict inequality $\qpd(\bar{f}e)> U(z)-U(y)$ holds.
Substituting these relations in \eqref{eq:qpd-converged-equality} yields
\begin{equation*}
g(e') = \qpd(f) + \qpd(\bar{f}e) > U(z)-U(x) = g(e'),
\end{equation*}
where we have used $g(e') = U(z)-U(x)$.
Since the strict inequality $g(e')>g(e')$ is a contradiction, this completes the proof.
\end{proof}

 For the following, recall that $U(e)$ for $e\in \Eu$ is defined to be the value of $U$ at the unique index-$1$ zero of $-\nabla U$ in $e\in \Eu$, and if $e\in \Ed$ then $U(e)\coloneqq U(q(e))$.
\begin{Lem}\label{lem:convering-morse-trees-reversal-invariant}
Assume that $\dft = \cfo^\sharp = -\nabla U$ is dual to a $\Cont^1$ exact one-form on a closed connected Riemannian manifold $M$ and satisfies Assumption~\ref{assump:morse-smale}.
Let $\cN_{\dft}$ be as in Lem.~\ref{lem:near-assump-2-implies-assump-2} and $\dft_n = \cfo_n^\sharp\in \cN_{\dft}$ be a sequence of $\Cont^1$ vector fields dual to closed one-forms converging to $\dft$ in the $\Cont^1$ topology, with associated directed Morse graphs $\compd^n = (\Ed^n,\vtx^n, \src, \tgt)$.
Assume that $U$ has a unique global minimizer and that there is a unique minimizer
\begin{equation}\label{eq:converging-morse-trees-reversal-invariant-1}
T\in \arg \min_{\edg\in \ST(\comp)}\sum_{e\in \edg}U(e).
\end{equation}
Then for every $v\in \vtx$ there is a unique minimizer
\begin{equation}\label{eq:converging-morse-trees-reversal-invariant-1pt5}
T_v\in \arg\min_{\edg\in \RST(\compd;v)}\sum_{e\in \edg}g(e),
\end{equation}
the equality $q(T_v) = T$ holds, and there is a unique minimizer $T_*\in \arg \min_{\edg\in \RST(\compd)}\sum_{e\in \edg}g(e)$.
Furthermore, for all sufficiently large $n$ there exists $T^n\in \ST(\comp^n)$ such that, for all $v_n\in \vtx^n$, there is a unique minimizer
\begin{equation}\label{eq:convering-morse-trees-reversal-invariant-2}
T^n_{v_n} \in \arg \min_{\min \edg\in \RST(\compd^n;v_n)}\sum_{e\in \edg}g(e),
\end{equation}
the equality $q(T^n_{v_n})=T^n$ holds, and there is a unique minimizer $T_*^n\in \arg \min_{\edg\in \RST(\compd^n)}\sum_{e\in \edg}g(e)$.
\end{Lem}
\begin{Rem}\label{rem:all-min-morse-trees-obtained-by-rev-edges}
In other words, for all sufficiently large $n$ and any $v_n, w_n\in \vtx^n$, the Morse minimizing rooted spanning trees in both $\RST(\compd^n;v_n)$ and $\RST(\compd^n;w_n)$ are unique and are obtained from one another by simply reversing some edges.
Moreover, for all sufficiently large $n$ there is a unique overall minimizer $T_*^n\in \RST(\compd^n)$ from among the $T_{v_n}^n$.
\end{Rem}

\begin{proof}
Fix any $v\in \vtx$, define $C\coloneqq -\sum_{v'\in \vtx}U(v')$, and note that any $\edg\in \RST(\compd;v)$ satisfies
\begin{align}\label{eq:convering-morse-trees-reversal-invariant-3}
\sum_{e\in \edg}g(e) =  C + U(v) +\sum_{e\in \edg}U(e) = C + U(v) +  \sum_{e\in q(\edg)}U(e)
\end{align}
since $g(e)=U(e)-U(\src(e))$ as $\cfo=-dU$, $C+U(v) = -\sum_{e\in \edg}U(\src(e))$, and $U(q(e))\coloneqq U(e)$.
Since the edges of the minimizer $T$ in \eqref{eq:converging-morse-trees-reversal-invariant-1} admit a unique choice of orientations producing a tree in $\RST(\compd;v)$, it follows from \eqref{eq:convering-morse-trees-reversal-invariant-3} and the uniqueness of the minimizer in \eqref{eq:converging-morse-trees-reversal-invariant-1} that the minimizer $T_v$ in \eqref{eq:converging-morse-trees-reversal-invariant-1pt5} is unique and satisfies $q(T_v) = T$, and since $U$ has a unique global minimizer it also follows that there is a unique minimizer $T_*\in \arg \min_{\edg\in \RST(\compd)}\sum_{e\in \edg}g(e)$.

Fix $v,w\in \vtx$ and any $v_n, w_n\in \vtx^n$ with $v_n\to v$ and $w_n\to w$.
The Morse graphs for $\dft_n$ converge to that of $\dft$ as $n\to \infty$ by the implicit function and stable manifold theorems \cite[p.~75,~Thm~6.2]{palis1980geometric}, so each $\dft_n$ has the same finite number of Morse rooted spanning trees for all sufficiently large $n$.
Since also the gain $g(\slot)$ of an edge depends jointly continuously on the edge and the vector field, it follows that the minimizers $T^n_{v_n}$ and $T^n_{w_n}$ in \eqref{eq:convering-morse-trees-reversal-invariant-2} are unique and $q(T^n_{v_n})=q(T^n_{w_n})$ for all large $n$ (since $q(T_w) = T = q(T_v)$).
By similar reasoning, there is a unique minimizer $T_*^n\in \arg \min_{\edg\in \RST(\compd^n)}\sum_{e\in \edg}g(e)$ for all sufficiently large $n$. 
Thus, defining $T^n\coloneqq q(T^n_{v_n})$ completes the proof.
\end{proof}

\subsection{Finishing the proofs}\label{sec:finishing-morse-flux-proof}
We now prove Theorem~\ref{th:qualitative}.
For convenience we restate the theorem.
\ThmQualitative*
\begin{proof}
This theorem was proved for the case $\dim(M) = 1$ in Ex.~\ref{ex:tilt-pot-circle}, so we may and do henceforth assume that $\dim(M) \geq 2$.
Recall that in this case $\qpd = \tqpd$ (Lem.~\ref{lem:tqp-equals-qp-transversality}), and this allows us to take advantage of the preliminary results concerning $\qpd$ proved in \S \ref{sec:spanning-tree-continuous-dependence} without additional fuss. 

Note that $\qpd([e])\leq g(e)$ for any $e\in \Ed$ by Lem.~\ref{lem:qp-leq-g-W}.
From this and Lem.~\ref{lem:converging-FW-trees-W} and \ref{lem:convering-morse-trees-reversal-invariant} (and the fact that the $\Cont^1$ topology is metrizable since $M$ is compact \cite[p.~62]{hirsch1976differential}) it follows that, if $\cfo$ is sufficiently close to $-dU$ in the $\Cont^1$ topology,
\begin{equation}\label{eq:m-th-pf-1}
\min_{T'\in \RST(\gh;v)}\sum_{e\in T'}\qpd(e) = \min_{T'\in \RST(\compd;v)}\sum_{e\in T'}g(e)
\end{equation}
for each $v\in \vtx$.
The same lemmas additionally imply that the minimizers $T^{\Pi}_{v}\in \RST(\gh;v)$ and $T^{m}_{v}\in \RST(\compd;v)$ are unique, the minimizer $T_* \in \arg\min_{\edg \in \RST(\compd)}\sum_{e\in \edg}g(e)$ is unique, and that 
\begin{equation*}
\textnormal{each $e\in T^{m}_v$ satisfies $[e]\in T^{\Pi}_v$ and $\qpd([e])=g(e)$.}
\end{equation*}
Defining $K_v\coloneqq \sum_{v'\in \vtx\setminus \{v\}}h(v')$, for later use we note that the latter fact implies that 
\begin{equation}\label{lem:finishing-morse-crst-bridge}
\begin{split}
\min_{\edg\in \RST(\gh)}\sum_{e\in \edg}h(\src(e))+\qpd(e) &= \min_{v\in \vtx} K_v + \sum_{e\in T^{\Pi}_v}\qpd(e) = \min_{v\in V}K_v + \sum_{e\in T^{m}_v}g(e)\\
&= \min_{\edg\in \RST(\compd)}\sum_{e\in \edg} h(e).
\end{split}
\end{equation}
Lem.~\ref{lem:convering-morse-trees-reversal-invariant} further implies the existence of $T\in \ST(\comp)$ such that
\begin{equation}\label{eq:morse-trees-reverse-each-other}
\textnormal{
 $q(T_v^m) = T$ for all $v\in \vtx$,}
\end{equation}
so that all minimizing Morse rooted spanning trees are obtained from one another by reversing some edges (Rem.~\ref{rem:all-min-morse-trees-obtained-by-rev-edges}).
It follows that, for any $v,w\in \vtx$, 
\begin{equation}\label{eq:finish-proof-morse-tree-diffs}
\sum_{e\in T^{\Pi}_v}\qpd(e) - \sum_{e\in T^{\Pi}_w}\qpd(e) = \sum_{e\in T^{m}_v}g(e) - \sum_{e\in T^{m}_w}g(e) = h(v)-h(w).
\end{equation}

Next, we equip $$N\coloneqq \bigcup_{\substack{e\in \Ed\\ h(e)< h(\bar{e})}}\Ws(e)$$
with the coorientation induced by the directed edges $e$ satisfying $h(e)< h(\bar{e})$, so the oriented intersection numbers $I(N, \vp)$ of $N$ with smooth paths $\vp$ transverse to $N$ are well-defined according to \eqref{eq:intersection-number}.
Let $k > 0$ be as in Lem.~\ref{lem:cycle-path-qpd-lower-bound} and fix $\varepsilon > 0$.
Fix $\sigma \in \{-1,+1\}$ and $\edg_0\in \CRST(\gh)$ satisfying $\textnormal{sign}(\cfo(\cycle(\edg_0))) = \sigma \neq 0$.
For each $e\in \cycle(\edg_0)$, let $\vp_e$ be a smooth path such that $[\vp_e]=e$, $\af(\vp_e)< \qpd(e) + \varepsilon$, and $\vp_e$ is transverse to $\bigcup_{z\in \dft^{-1}(0)}\Ws(z)$  (Lem.~\ref{lem:tqp-equals-qp-transversality}).
From the transversality condition it follows that the image of each $\vp_e$ does not intersect the stable manifold of any $z\in \dft^{-1}(0)$ with $\ind(z)\geq 2$ (since then $\textnormal{codim}(\Ws(z))\geq 2$).
Since $\textnormal{sign}(\cfo(\cycle(\edg_0))) = \sigma$, there exists $e_0\in \cycle(\edg_0)$ such that $\textnormal{sign}(I(N,\vp_{e_0}))  = \sigma$.
Let $[T_1,T_2]=\dom(\vp_{e_0})$ and define $t_*\coloneqq \inf\{t\geq T_1\colon \vp_{e_0}(t_*)\in N\}$, $\bar{h}_*\coloneqq \min\{h(\bar{e})\colon e\in \Ed \textnormal{ and } h(e) < h(\bar{e})\}$ (hence $0 < h_* < \bar{h}_*$), and $k_\sigma \geq 0$ by \begin{equation}\label{eq:ksig-def}
k_{\sigma}\coloneqq \begin{cases}
\min(k,\bar{h}_*-h_*), & \sigma < 0\\
0, & \sigma > 0
\end{cases}.
\end{equation}

Lem.~\ref{lem:cycle-path-qpd-lower-bound} implies that
\begin{equation}\label{eq:h-e-star-ineq}
\begin{split}
k_\sigma + h_* \leq  h(\src(e_0)) + \af(\vp_{e_0}) < h(\src(e_0)) + \qpd(e_0) + \varepsilon.
\end{split}
\end{equation}
Using \eqref{eq:h-e-star-ineq} and defining $C\coloneqq -\sum_{v\in \vtx}h(v)$, we compute
\begin{equation*}
\begin{split} 
\varepsilon + \sum_{e\in \edg_0}\qpd(e) &= \varepsilon +  C + \sum_{e\in \edg_0}h(\src(e)) + \qpd(e) \geq C+ k_\sigma + h_* + \sum_{e\in \edg_0\setminus \{e_0\}}h(\src(e)) + \qpd(e)\\
&\geq C + k_\sigma + h_* + \min_{\edg\in \RST(\gh)}\sum_{e\in \edg} h(\src(e))+\qpd(e)\\
&= C + k_\sigma + h_* + \min_{\edg\in \RST(\compd)}\sum_{e\in \edg}h(e),
\end{split} 
\end{equation*}
where the final equality follows from \eqref{lem:finishing-morse-crst-bridge}.
Let $e_*\in \Ed$ satisfy $h_* = h(e_*) < h(\bar{e}_*)$.
It follows from \eqref{eq:morse-trees-reverse-each-other} that $\min_{\edg\in \RST(\compd)}\sum_{e\in \edg}h(e) = \sum_{e\in T^{m}_v}h(e)$ for any $v\in \vtx$, so
\begin{equation}\label{eq:epsilon-sandwich}
\begin{split} 
\varepsilon + \sum_{e\in \edg_0}\qpd(e) &\geq C + k_\sigma + h_* + \min_{\edg\in \RST(\compd)}\sum_{e\in \edg}h(e)\\
&= C+k_\sigma + h_* + \sum_{e\in T^{m}_{\src(e_*)}}h(e)\\
&= k_\sigma + \sum_{e\in T^{m}_{\src(e_*)}\cup\{e_*\}}g(e)\\
&\geq k_\sigma + \sum_{e\in T^{m}_{\src(e_*)}\cup\{e_*\}}\qpd([e])\\
&\geq \sum_{e\in T^{m}_{\src(e_*)}\cup\{e_*\}}\qpd([e]),
\end{split}
\end{equation}
where the penultimate inequality follows from Lem.~\ref{lem:qp-leq-g-W}, and the final inequality follows from $k_{\sigma}\geq 0$.
Note that $\edg_1\coloneqq \{[e]\colon e\in T^{m}_{\src(e_*)}\cup\{e_*\}\}\in \CRST(\gh)$ and $\cfo(\cycle(\edg_1)) >0$ by Rem.~\ref{rem:h-star-alt-def} since $h(e_*)< h(\bar{e}_*)$. 
Since $\varepsilon > 0$ was arbitrary and $\edg_0\in \CRST(\gh)$ was an arbitrary cycle-rooted spanning tree satisfying $\sigma = \textnormal{sign}(\cfo(\cycle(\edg_0)))$, it follows from this and \eqref{eq:epsilon-sandwich} that
\begin{equation}\label{eq:final-morse-proof-almost-done}
\begin{split}
\min_{\substack{\edg\in \CRST(\gh)\\ \cfo(\cycle(\edg))<0}}\sum_{e\in \edg}\qpd(e) &\geq  k_{-1} + \min_{\substack{\edg\in \CRST(\gh)\\ \cfo(\cycle(\edg))>0}}\sum_{e\in \edg}\qpd(e)\\
& > \min_{\substack{\edg\in \CRST(\gh)\\ \cfo(\cycle(\edg))>0}}\sum_{e\in \edg}\qpd(e)
= g(e_*) +  \sum_{e\in T^m_{\src(e_*)}}g(e),
\end{split}
\end{equation}
where the strict inequality follows from  $k_{-1} > 0$.

It now follows from Theorem \ref{th:flux-manifold-mc-CRST-ld} and \eqref{eq:final-morse-proof-almost-done}, \eqref{eq:m-th-pf-1}, and \eqref{eq:finish-proof-morse-tree-diffs} (in that order) that 
\begin{equation*}
\begin{split}
\lim_{\varepsilon\to 0}(-\varepsilon \ln \fluxe([\cfo])) &= \left(\min_{\substack{\edg\in \CRST(\gh)\\\cfo(\cycle(\edg))> 0}}\sum_{e\in \edg}\qpd(e)\right) - \left(\min_{\substack{\edg \in \RST(\gh)}}\sum_{e\in \edg}\qpd(e)\right)\\
&= g(e_*) +  \sum_{e\in T^m_{\src(e_*)}}g(e) - \left(\min_{\substack{\edg \in \RST(\compd)}}\sum_{e\in \edg}g(e)\right)\\
&= g(e_*) + h(\src(e_*)) - h(v_*) = h_*,
\end{split}
\end{equation*}
where the final equality follows since $h_* = h(e_*) = g(e_*) + h(\src(e_*))$ by the definitions of $h(\slot)$ and $e_*$, and since the root $v_*$ of $T_*$ satisfies $h(v_*)=0$.
This completes the proof.
\end{proof}

We now prove Theorem~\ref{th:qualitative-measure}.
For convenience we restate the theorem.
For the statement, recall that $B_r(x)$ denotes the metric ball of radius $r \geq 0$ centered at $x\in M$.
\ThmQualitativeMeasure*
\begin{proof}
Since the hypotheses of Theorem~\ref{th:qualitative-measure} are identical to those of Theorem~\ref{th:qualitative}, for brevity we will assume all notation and statements established in the proof of Theorem~\ref{th:qualitative} up to and including \eqref{eq:finish-proof-morse-tree-diffs}.
The uniqueness of the minimizer $T_*\in \arg\min_{\edg\in \RST(\compd)}\sum_{e\in \edg}g(e)$ together with \cite[p.~167,~Thm~4.2]{freidlin2012random} immediately imply the weak convergence of $\mu_\varepsilon$ to $\delta_{v_*}$  as $\varepsilon \to 0$.

It follows from \cite[Ch.~6, Thm~4.1, Lem.~4.1]{freidlin2012random} that for any $\delta > 0$ there is $k > 0$ such that, for any $v\in \vtx$ and $\varepsilon, r\in (0,k)$,
\begin{equation*}
\begin{split}
-\delta + \left( \min_{\edg\in \RST(\gh;v)}\sum_{e\in \edg}\qpd(e)\right) &- \left( \min_{\edg\in \RST(\gh)}\sum_{e\in \edg}\qpd(e)\right) < -\varepsilon \ln \left(\int_{B_r(v)}\dme(x) dx\right)\\
  &< \delta + \left( \min_{\edg\in \RST(\gh;v)}\sum_{e\in \edg}\qpd(e)\right) - \left( \min_{\edg\in \RST(\gh)}\sum_{e\in \edg}\qpd(e)\right) .
 \end{split}
\end{equation*}
This and \eqref{eq:finish-proof-morse-tree-diffs} imply that 
\begin{equation}\label{eq:th-qualitative-measure-main-fw}
\begin{split}
-\delta + h(v) - h(v_*) &< -\varepsilon \ln \left(\int_{B_r(v)}\dme(x) dx\right) < \delta + h(v) - h(v_*) .
 \end{split}
\end{equation}
Since $h(v_*)=0$ by Def.~\ref{def:morse-heights} of $h(\slot)$, taking $\exp(-\frac{1}{\varepsilon}\slot)$ of \eqref{eq:th-qualitative-measure-main-fw} establishes \eqref{eq:th-qualitative-measure-main} for all $\varepsilon, r\in (0,k)$ and completes the proof. 
\end{proof}

We now prove Prop.~\ref{prop:one-minimizer}.
For convenience we restate the proposition.
\PropOneMinimizer*
\begin{proof}
This proposition was proved for the case $\dim(M) = 1$ in Ex.~\ref{ex:tilt-pot-circle}, so we may and do henceforth assume that $\dim(M) \geq 2$.
Recall that in this case $\qpd = \tqpd$ (Lem.~\ref{lem:tqp-equals-qp-transversality}). 
Since $\dft$ has precisely one index-$0$ zero $v_*\in \vtx$, $\RST(\gh)$ contains only the trivial rooted spanning tree with no edges, and every $\edg\in \CRST(\gh)$ consists of a single edge $\edg\in \Ed$. 
Note that every $e\in \Ed$ satisfies $\src(e)=\tgt(e)=v_*$ since there is only one index-$0$ zero.
In the following, given $e\in \Ed$ we use the notation $\cfo(e)\coloneqq \int_e \cfo$ of Theorem~\ref{th:flux-manifold-mc-CRST-ld}.

As in the proof of Theorem~\ref{th:qualitative}, we equip $$N\coloneqq \bigcup_{\substack{e\in \Ed\\ h(e)< h(\bar{e})}}\Ws(e)$$
with the coorientation induced by the directed edges $e$ satisfying $h(e)< h(\bar{e})$, so the oriented intersection numbers $I(N, \vp)$ of $N$ with smooth paths $\vp$ transverse to $N$ are well-defined according to \eqref{eq:intersection-number}.
Let $k > 0$ be as in Lem.~\ref{lem:cycle-path-qpd-lower-bound} and fix $\varepsilon > 0$.
Fix $\sigma \in \{-1,+1\}$ and $\{e_0\}\in \CRST(\gh)$ satisfying $\textnormal{sign}(\cfo(e_0)) = \sigma \neq 0$.
Let $\vp_{e_0}$ be a smooth path with $[\vp_{e_0}]=e_0$, $\af(\vp_{e_0})< \qpd(e_0) + \varepsilon$, and such that $\vp_{e_0}$ is transverse to $\bigcup_{z\in \dft^{-1}(0)}\Ws(z)$ (Lem.~\ref{lem:tqp-equals-qp-transversality}).

Since $\textnormal{sign}(\cfo(e_0))) = \sigma$ it follows that $\textnormal{sign}(I(N,\vp_{e_0}))  = \sigma$.
Let $[T_1,T_2]=\dom(\vp_{e_0})$ and define $t_*\coloneqq \inf\{t\geq T_1\colon \vp_{e_0}(t_*)\in N\}$, $\bar{h}_*\coloneqq \min\{h(\bar{e})\colon e\in \Ed \textnormal{ and } h(e) < h(\bar{e})\}$ (hence $0 < h_* < \bar{h}_*$), and $k_\sigma \geq 0$ by
\begin{equation}\label{eq:ksig-def-2}
k_{\sigma}\coloneqq \begin{cases}
\min(k,\bar{h}_*-h_*), & \sigma < 0\\
0, & \sigma > 0
\end{cases}.
\end{equation} 
Since $h(\src(e_0)) = h(v_*)=0$, Lem.~\ref{lem:cycle-path-qpd-lower-bound} implies that
\begin{equation}\label{eq:h-e-star-ineq-2}
\begin{split}
k_\sigma + h_* \leq \af(\vp_{e_0}) < \qpd(e_0) + \varepsilon.
\end{split}
\end{equation}
Since $\varepsilon>0$ was arbitrary and $k_{-1}>k_{+1}=0$, this and Lem.~\ref{lem:qp-leq-g-W} imply that
\begin{equation}\label{eq:one-min-morse-final}
\min_{\substack{\edg\in \CRST(\gh)\\ \cfo(\cycle(\edg)) < 0}}\sum_{e\in \edg} \qpd(e) \geq k_{-1} + \min_{\substack{\edg\in \CRST(\gh)\\ \cfo(\cycle(\edg)) > 0}}\sum_{e\in \edg} \qpd(e) >  \min_{\substack{\edg\in \CRST(\gh)\\ \cfo(\cycle(\edg)) > 0}}\sum_{e\in \edg} \qpd(e) =  h_*.
\end{equation}
Thus, \eqref{eq:one-min-morse-final} and Theorem~\ref{th:flux-manifold-mc-CRST-ld} imply the desired equality
\begin{equation*}
\begin{split}
\lim_{\varepsilon\to 0}(-\varepsilon \ln \fluxe([\cfo])) &= \left(\min_{\substack{\edg\in \CRST(\gh)\\\cfo(\cycle(\edg))> 0}}\sum_{e\in \edg}\qpd(e)\right) -\underbrace{\left(\min_{\substack{\edg \in \RST(\gh)}}\sum_{e\in \edg}\qpd(e)\right)}_{0} =  h_*.
\end{split}
\end{equation*}
\end{proof}

\section{Counterexamples}\label{sec:counterexamples}
In Rem.~\ref{rem:th-qualitative-sufficiently-close} we pointed out the following two natural questions:
\begin{Quest}\label{quest:weaken-hypotheses-flux}
Does the conclusion of Theorem~\ref{th:qualitative} remain true if the hypothesis that $\cfo$ is sufficiently close to $-dU$ in the $C^1$ topology is replaced with the weaker hypothesis that $\cfo^\sharp$ satisfies Assumption~\ref{assump:morse-smale} and the minimizer $T_*$ in \eqref{eq:morse-rst-minimizer} is unique?
\end{Quest}
\begin{Quest}\label{quest:weaken-hypotheses-measure}
Does the conclusion of Theorem~\ref{th:qualitative-measure} remain true if the same hypothesis replacement in Question~\ref{quest:weaken-hypotheses-flux} is made?
\end{Quest}
In the present section we demonstrate that the answers to Questions~\ref{quest:weaken-hypotheses-flux} and \ref{quest:weaken-hypotheses-measure} are both negative for multiple reasons, even if certain natural conditions are imposed on the Morse graph $\compd$ and gains $g(\slot)$, although we saw in Ex.~\ref{ex:tilt-pot-circle} that Question~\ref{quest:weaken-hypotheses-flux} has a positive answer in the special case that $\dim(M)=1$.
(However, we show in Ex.~\ref{ex:cex-1d-measure} that the answer to Question~\ref{quest:weaken-hypotheses-measure} is negative even when $\dim(M)=1$.)
The various claims concerning $\qpd(\slot)$ made in these examples can be justified using Lem.~\ref{lem:S-lower-bound}, Lem.~\ref{lem:qp-leq-g-W}, and/or  Prop.~\ref{prop:qp-int-conditions-iff-connecting-piecewise-orbit}, but we omit some details for the sake of brevity.

The following setup is shared by both Ex.~\ref{ex:cex-2d-flux-fw-morse-agree} and \ref{ex:cex-2d-flux-all-min-trees-coincide-as-undirected} which refer to Fig.~\ref{fig:cex-2d-flux-fw-morse-agree} and \ref{fig:cex-2d-flux-all-min-trees-coincide-as-undirected}, respectively.
The ``tilted potential'' $\tilde{U}\in \Cont^\infty(\R^2)$ is a Morse function satisfying $d\tilde{U}(x+k)=d\tilde{U}(x)$ for all $x\in \R^2$ and $k\in \Z^2$, so the exact one-form $-d\tilde{U}$ descends to a closed one-form $\cfo$ on $M = \tor^2$ identified with $\R^2/\Z^2$.
We give $M$ the Euclidean metric inherited from the standard one on $\R^2$, and we define the smooth vector field $\dft = \cfo^\sharp$ on $M$ dual to $\cfo$ via this metric.
The vector field $\dft = \cfo^\sharp$ satisfies Assumption~\ref{assump:morse-smale} in both Ex.~\ref{ex:cex-2d-flux-fw-morse-agree} and \ref{ex:cex-2d-flux-all-min-trees-coincide-as-undirected}.
In both of these examples we consider the steady-state $[\cfo]$-flux $\fluxe([\cfo])$ of the diffusion ($X^\varepsilon_t, \Prob^\varepsilon_x)$ on $M$ with generator $\dfte + \varepsilon \Delta$, where $\varepsilon > 0$ and the smooth vector fields $\dfte\to \dft$ uniformly as $\varepsilon\to 0$.

\begin{Ex}\label{ex:cex-2d-flux-fw-morse-agree}
\begin{figure}
	\centering
	\def\svgwidth{1.0\columnwidth}
	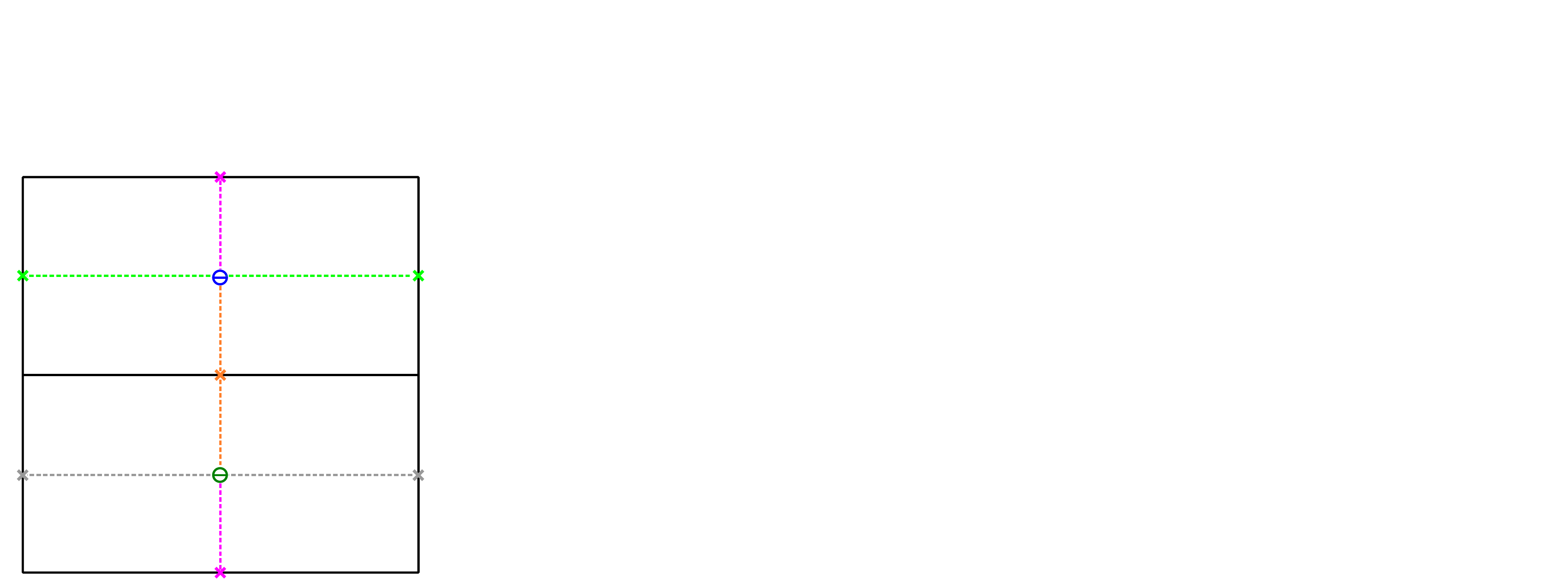
	\caption{An illustration of Ex.~\ref{ex:cex-2d-flux-fw-morse-agree}. Critical points of $\tilde{U}$ of index $0$, $1$, and $2$ are respectively indicated by the symbols $\ominus$, $\times$, and $\oplus$ (analogous to negative, ``mixed'', and positive charge in electromagnetism).  
	Critical points are labeled with the numerical values of $\tilde{U}$ attained there. One-dimensional stable and unstable manifolds of the flow of $-\nabla \tilde{U}$ are indicated by solid and dotted curves, respectively.}\label{fig:cex-2d-flux-fw-morse-agree}
\end{figure}
We assume the notation of Fig.~\ref{fig:cex-2d-flux-fw-morse-agree} and the description in its caption.
Note that the de Rham cohomology class $[\cfo]\in \Hdr^1(M)$ satisfies $[\cfo] = [dx] - 194[dy]$, where $(x,y)$ are the coordinates on $M$ induced from the standard coordinates on $\R^2$. 
Shown on the far left of Fig.~\ref{fig:cex-2d-flux-fw-morse-agree} are all critical points of $\tilde{U}|_{[0,1]^2}$ and the (un)stable manifolds of the index-$1$ critical points as indicated, followed by the undirected Morse graph $\comp = (\vtx, \Eu, \st)$ on the right.
Shown next to the right in Fig.~\ref{fig:cex-2d-flux-fw-morse-agree} are edges $e_1, e_2, e_3$ in the directed Morse graph $\compd = (\vtx, \Ed, \src, \tgt)$.
These edges are such that 
\begin{equation}\label{eq:cex-2d-flux-fw-morse-agree-min-rst-crst}
\begin{split}
\min_{T\in \RST(\gh)}\sum_{e\in T}\qpd(e) &= \qpd([e_1]) = \min_{T\in \RST(\compd)} \sum_{e\in T}g(e) = g(e_1) = 6\\
\min_{\substack{\edg\in \CRST(\gh)\\ \cfo(\cycle(\edg))>0}}\sum_{e\in \edg}\qpd(e) &= \qpd([e_2]) + \qpd([e_3]) = \min_{\substack{\edg\in \CRST(\compd)\\ \cfo(\cycle(\edg))>0}} \sum_{e\in T}g(e) \\
&= g(e_2) + g(e_3) = 7 + 1 = 8,
\end{split}
\end{equation}
where $[e]\in \Eh$ denotes the path homotopy class of an orientation preserving parametrization of $e\in \Ed$, and all four of these minimizers are unique.
Similarly,  $$\min_{\substack{\edg\in \CRST(\gh)\\ \cfo(\cycle(\edg))<0}}\sum_{e\in \edg}\qpd(e) = \min_{\substack{\edg\in \CRST(\compd)\\ \cfo(\cycle(\edg))<0}} \sum_{e\in T}g(e) = g(e_2) + g(\bar{e}_3) = 9 > 8 = \min_{\substack{\edg\in \CRST(\gh)\\ \cfo(\cycle(\edg))>0}}\sum_{e\in \edg}\qpd(e).$$
From this, \eqref{eq:cex-2d-flux-fw-morse-agree-min-rst-crst}, Theorem~\ref{th:flux-manifold-mc-CRST-ld}, and Lem.~\ref{lem:tqp-equals-qp-transversality} it follows that the steady-state $[\cfo]$-flux $\fluxe([\cfo]) = \fluxe([dx])$ is positive for sufficiently small $\varepsilon > 0$ and 
\begin{equation}\label{eq:cex-2d-flux-fw-morse-agree-flux-asym}
\lim_{\varepsilon\to 0}(-\varepsilon \ln \fluxe([\cfo])) = 8-6 =  2.
\end{equation}
On the other hand, since $T_* = \{e_1\}$ and $v_* = \tgt(e_1) = v_1$, if $e_4\in \Ed$ is the edge emanating from $v_1$ toward the right, we see that $h_* = h(e_4) = 5 > 2$.
Thus, the asymptotics in \eqref{eq:cex-2d-flux-fw-morse-agree-flux-asym} do not match those in the conclusion of Theorem~\ref{th:qualitative}, so the answer to Question~\ref{quest:weaken-hypotheses-flux} is negative. 

One reason for this disagreement is that, ignoring orientations, the minimizing cycle-rooted spanning tree is not obtained by adding a single edge to $T_*$, a property seen to follow from the hypotheses of Theorem~\ref{th:qualitative} during the course of its proof in \S \ref{sec:finishing-morse-flux-proof}.
It is possible for this to occur in the present example because $\cfo$ is not sufficiently $\Cont^1$-close to an exact one-form, as Theorem~\ref{th:qualitative} requires; the ``tilt'' of $\tilde{U}$ is too large.
\end{Ex}
Ex.~\ref{ex:cex-2d-flux-fw-morse-agree} provided a negative answer to Question~\ref{quest:weaken-hypotheses-flux}, and it also motivates the following two related questions.
Question~\ref{quest:do-morse-fw-min-rst-agree} is motivated by the observation that the minimizing Morse and path-homotopical rooted spanning trees ``agreed'' in Ex.~\ref{ex:cex-2d-flux-fw-morse-agree}. 
Question~\ref{quest:reformulated-q1} is a reformulation of Question~\ref{quest:weaken-hypotheses-flux} motivated by the fact that all minimizing Morse rooted spanning trees agree as undirected graphs if $\cfo$ is sufficiently $\Cont^1$-close to $-dU$, where $U$ satisfies the relevant hypotheses of Theorem~\ref{th:qualitative} (Lem.~\ref{lem:convering-morse-trees-reversal-invariant} and Rem.~\ref{rem:all-min-morse-trees-obtained-by-rev-edges}).
\begin{Quest}\label{quest:do-morse-fw-min-rst-agree}
Assume that $\dft = \cfo^\sharp$ satisfies Assumption~\ref{assump:morse-smale}.
If the respective minimizers $T^m_v\in \RST(\compd;v)$ and $T^\Pi_v \in \RST(\gh;v)$ of $\sum g(\slot)$ and $\sum \qpd(\slot)$ are unique for all $v\in \vtx$, do they always agree in the sense that $[e]\in T^\Pi_v$ and $\qpd([e])=g(e)$ for every $e\in T^m_v$? 
\end{Quest}
\begin{Quest}[Reformulation of Question~\ref{quest:weaken-hypotheses-flux}]\label{quest:reformulated-q1}
Does the conclusion of Theorem~\ref{th:qualitative} remain true if the hypothesis that $\cfo$ is sufficiently close to $-dU$ in the $C^1$ topology is replaced with the weaker set of hypotheses that $\cfo^\sharp$ satisfies Assumption~\ref{assump:morse-smale}, the minimizer $T_*$ in \eqref{eq:morse-rst-minimizer} is unique, and for any $v,w\in \vtx$ the minimizers $T_v \in \RST(\compd;v)$ and $T_w\in \RST(\compd;w)$ of $\sum g(\slot)$ are unique and coincide as undirected graphs (in the sense that $q(T_v)=q(T_w)$; cf. Rem.~\ref{rem:all-min-morse-trees-obtained-by-rev-edges} and Def.~\ref{def:morse-graph-directed})?
\end{Quest}
In Ex.~\ref{ex:cex-2d-flux-all-min-trees-coincide-as-undirected} we show that the answers to Questions~\ref{quest:do-morse-fw-min-rst-agree} and \ref{quest:reformulated-q1} are both negative.
\begin{Ex}\label{ex:cex-2d-flux-all-min-trees-coincide-as-undirected}
\begin{figure}
	\centering
	\def\svgwidth{1.0\columnwidth}
	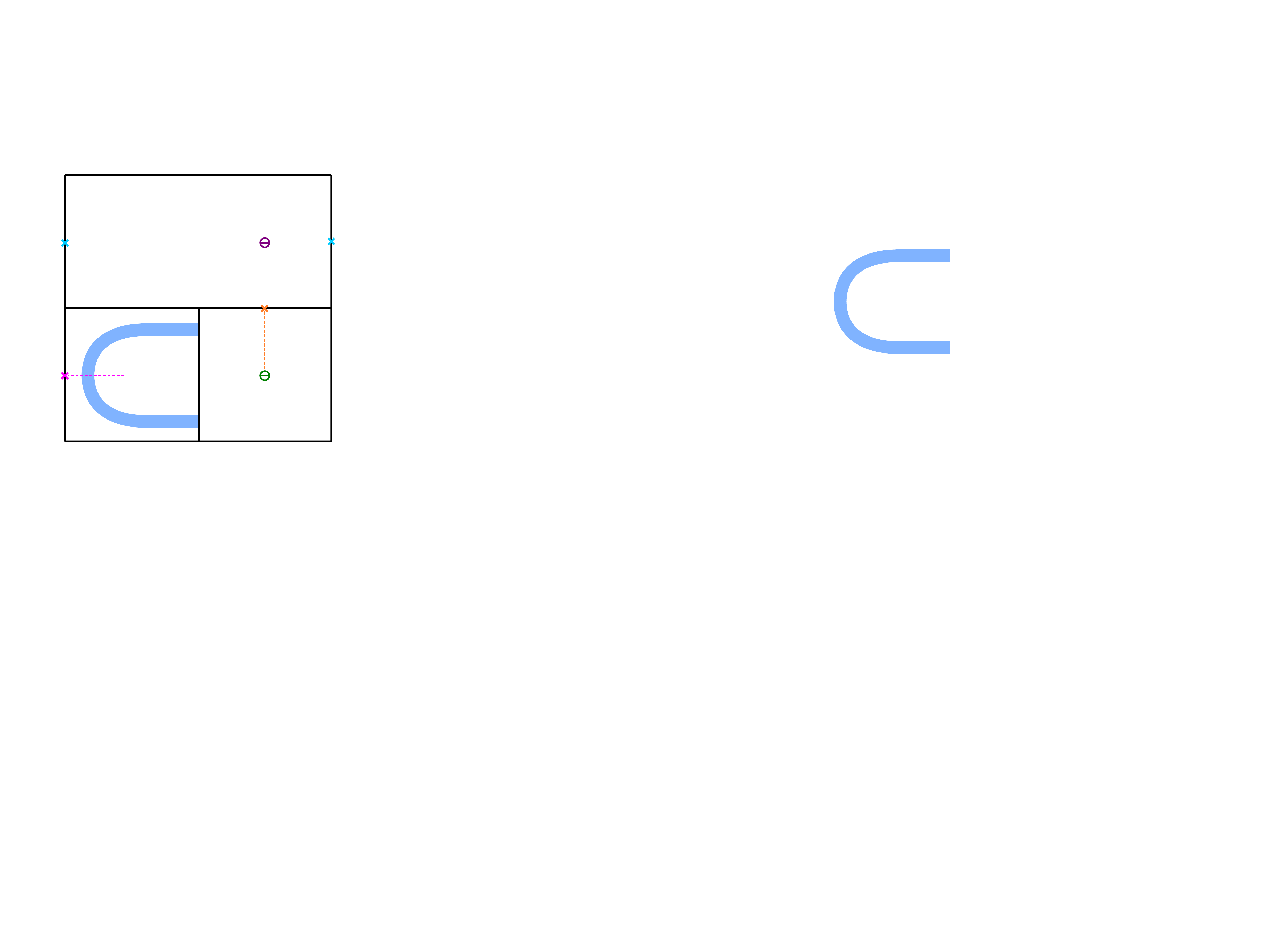
	\caption{An illustration of Ex.~\ref{ex:cex-2d-flux-all-min-trees-coincide-as-undirected}. Critical points of $\tilde{U}$ of index $0$, $1$, and $2$ are respectively indicated by the symbols $\ominus$, $\times$, and $\oplus$ (analogous to negative, ``mixed'', and positive charge in electromagnetism).  
	Critical points are labeled with the numerical values of $\tilde{U}$ attained there. One-dimensional stable and unstable manifolds of the flow of $-\nabla \tilde{U}$ are indicated by solid and dotted curves, respectively. The values of $\tilde{U}$ are all approximately equal to $999$ on the horseshoe-shaped ``almost-plateau'' shown on the left (we need an almost-plateau rather than an actual flat plateau in order for $\tilde{U}$  and hence also $\cfo$ to be Morse.)}\label{fig:cex-2d-flux-all-min-trees-coincide-as-undirected}
\end{figure}
We assume the notation of Fig.~\ref{fig:cex-2d-flux-all-min-trees-coincide-as-undirected} and the description in its caption.
Note that the de Rham cohomology class $[\cfo]\in \Hdr^1(M)$ satisfies $[\cfo] = 100[dx]$, where $(x,y)$ are the coordinates on $M$ induced from the standard coordinates on $\R^2$. 
Shown on the far left of Fig.~\ref{fig:cex-2d-flux-all-min-trees-coincide-as-undirected} are all critical points of $\tilde{U}|_{[0,1]^2}$ and the (un)stable manifolds of the index-$1$ critical points as indicated.
The horseshoe-shaped ``almost-plateau'' region $H \subset (0,1)^2$ contains no critical points of $\tilde{U}$, and is such that
\begin{equation}\label{eq:delta-plateau}
\sup_{(x,y)\in H}\norm{\nabla \tilde{U}(x,y)} + |\tilde{U}(x,y)-999| < \delta_1,
\end{equation}
where $\delta_1 > 0$ can be made arbitrarily small by modifications of $\tilde{U}$ supported on a small neighborhood of $H$.
Fixing $\delta_2 > 0$, we also choose paths $\vp_j$ approximating the edges of the unique minimizers in $\RST(\gh;v_i)$ of $\sum \qpd(\slot)$ and the unique minimizers in $\CRST(\gh)$ satisfying $\cfo(\cycle(\slot)) < 0$ and $\cfo(\cycle(\slot))> 0$ such that the sum $\sum \qpd(e)$ over any of these minimizers is $\delta_2$-approximated by the sum of the actions $\af(\vp_j)$ of the corresponding paths.
From the rightmost two panels at the top of Fig.~\ref{fig:cex-2d-flux-all-min-trees-coincide-as-undirected} we see already that the answer to Question~\ref{quest:do-morse-fw-min-rst-agree} is negative.

We may assume the edges of the minimizer in $\CRST(\gh)$ satisfying $\cfo(\cycle(\slot)) < 0$ are $\delta_2$-approximated by $\af(\vp_1) + \af(\vp_4) + \af(\vp_6)$, where $\vp_6$ is a suitable orientation-preserving parametrization of $\bar{e}_3$ going $v_1$ to $v_3$. 
From this, the bottom right panel of Fig.~\ref{fig:cex-2d-flux-all-min-trees-coincide-as-undirected}, the fact that
\begin{equation*}
\begin{split}
|\af(\vp_3)+\af(\vp_5)+\af(\vp_4) -( 5 + 900 + 1099+10^6)| &< \delta_2\\
|\af(\vp_1)+\af(\vp_6)+\af(\vp_4) -(100 + 905 + 1099+10^6)| &< \delta_2,
\end{split}
\end{equation*}
and $5+900 < 100 + 905$ we see that the hypothesis \eqref{eq:min-rst-assumption} of Theorem~\ref{th:flux-manifold-mc-CRST-ld} holds if $\delta_2$ is sufficiently small.
By shrinking $\delta_1$ in \eqref{eq:delta-plateau}, we may assume that 
\begin{equation*}
|\af(\vp_2) - 1000| < \delta_2
\end{equation*}
since \eqref{eq:finite-af-iffs-l2} and the fact that $\norm{\dft} < \delta_1$ on $H$ imply that the contribution to $\af(\vp_2)$ of the portion of $\vp_2$ in $H$ can be made arbitrarily small by  shrinking $\delta_1$ and making $\vp_2$ travel sufficiently slowly through $H$.

It follows from Fig.~\ref{fig:cex-2d-flux-all-min-trees-coincide-as-undirected}, the fact that $\sum \qpd(\slot)$ over its unique minimizer in $\RST(\gh)$ (which has root $v_2$) is $\delta_2$-approximated by $\af(\vp_2) + \af(\vp_3)$, Theorem~\ref{th:flux-manifold-mc-CRST-ld}, Lem.~\ref{lem:tqp-equals-qp-transversality}, and the triangle inequality that
\begin{equation}\label{eq:cex-2d-flux-all-min-trees-coincide-as-undirected-thm4}
\begin{split}
&\left|\lim_{\varepsilon\to 0}(-\varepsilon \ln \fluxe([\cfo]))- (5+900+1099+10^6) + (1000 + 5)\right|< 5\delta_2.
\end{split}
\end{equation}
On the other hand, since $T_* = \{e_1,e_2\}$ with root $v_* = v_2$, we see from Fig.~\ref{fig:cex-2d-flux-all-min-trees-coincide-as-undirected} that $$h_* = h(e_3) = 900+10^6 < 999 + 10^6 = (5+900+1099+10^6) - (1000+5),$$
so the small-noise asymptotics of $\fluxe([\cfo]) =  \fluxe(100[dx])$ in \eqref{eq:cex-2d-flux-all-min-trees-coincide-as-undirected-thm4} do not match those in the conclusion of Theorem~\ref{th:qualitative} if $\delta_2$ is small enough, in which case $h_*$ is smaller than the right side of the conclusion \eqref{eq:flux-sde-mc-ld-expression} of Theorem~\ref{th:flux-manifold-mc-CRST-ld}.
The reason for this disagreement is that the sums $\sum g(\slot)$ and $\sum \qpd(\slot)$ over the minimizers in $\RST(\compd)$ and $\RST(\gh)$ do not coincide, a property guaranteed if $\cfo$ is sufficiently close to a generic exact one-form (Lem.~\ref{lem:converging-FW-trees-W}) and used in the proof of Theorem~\ref{th:qualitative}.
As in Ex.~\ref{ex:cex-2d-flux-fw-morse-agree}, the ``tilt'' of $\tilde{U}$ is too large. 
From Fig.~\ref{fig:cex-2d-flux-all-min-trees-coincide-as-undirected} we see that all minimizers in $\RST(\compd;v_i)$ of $\sum g(\slot)$ for $i\in \{1,2,3\}$ agree when viewed as undirected graphs, so it follows that the answer to Question~\ref{quest:reformulated-q1} is negative. 
\end{Ex}

\begin{Rem}\label{rem:h-star-neither-smaller-nor-bigger-in-general}
In Ex.~\ref{ex:cex-2d-flux-fw-morse-agree} the value of $h_*$ was larger than the right side of the conclusion \eqref{eq:flux-sde-mc-ld-expression} of Theorem~\ref{th:flux-manifold-mc-CRST-ld}, whereas $h_*$ was smaller than the right side of \eqref{eq:flux-sde-mc-ld-expression} in Ex.~\ref{ex:cex-2d-flux-all-min-trees-coincide-as-undirected}.
We therefore see that, when the conclusion of Theorem~\ref{th:qualitative} does not hold, the exponential decay rate of the flux correctly given by Theorem~\ref{th:flux-manifold-mc-CRST-ld} is neither smaller nor larger, in general, than that of the conclusion of Theorem~\ref{th:qualitative}.
\end{Rem}

Finally, the following example shows that the answer to Question \ref{quest:weaken-hypotheses-measure} is negative even in the special case that $\dim(M) = 1$.

\begin{Ex}\label{ex:cex-1d-measure}
\begin{figure}
	\centering
	\def\svgwidth{0.5\columnwidth}
	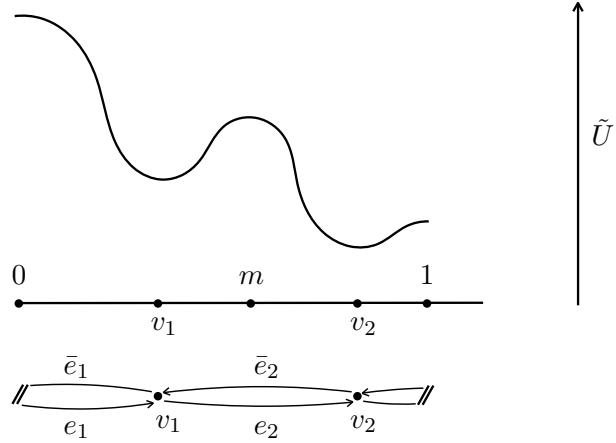
	\caption{An illustration of Ex.~\ref{ex:cex-1d-measure}.}\label{fig:cex-1d-measure}
\end{figure}
Consider a $1$-periodic Morse function $\tilde{U}\in\Cont^\infty(\R)$ such that $0$ is a local maximizer of $\tilde{U}$ and $\tilde{U}|_{[0,1]}$ has exactly two local minima. 
See Fig.~\ref{fig:cex-1d-measure}; we use the notation therein.
We assume that 
\begin{equation}\label{eq:cex-1d-U-assumps}
\tilde{U}(0)-\tilde{U}(v_1) > \tilde{U}(m) - \tilde{U}(1) > \tilde{U}(m)-\tilde{U}(v_1) > \tilde{U}(1) - \tilde{U}(v_2).
\end{equation}
Since $\tilde{U}$ is $1$-periodic, the exact one-form $-d\tilde{U}$ descends to a closed one-form $\cfo$ on $M = \sph^1$ identified with $\R/\Z$.
We give $M$ the Euclidean metric inherited from the standard one on $\R$ and we define the vector field $\dft \coloneqq \cfo^\sharp$ on $M$.
The edges of the Morse graph $\compd = (\vtx, \Ed, \src, \tgt)$ are shown in Fig.~\ref{fig:cex-1d-measure}; in the present $1$-dimensional case, these can be identified with those edges $e$ in the path-homotopical graph $\gh = (\vtx, \Eh, \src,\tgt)$ for which $\tqpd(e)<+\infty$.
For each $\varepsilon > 0$ we let $\dfte$ be a smooth vector field on $M$ such that $\dfte\to \dft$ uniformly as $\varepsilon \to 0$.
We denote by $B_r(x)$ the closed metric ball of radius $r$ centered at $x\in M$ and we consider the diffusion process $(X^\varepsilon_t, \Prob^\varepsilon_x)$ on $M$ with generator $\dfte + \varepsilon \Delta$ and stationary density $\dme\in \Cont^\infty(M)$.
It follows from \cite[Ch.~6, Thm~4.1]{freidlin2012random} that for any $\delta > 0$ there is $k > 0$ such that, for any $v\in \vtx$ and $\varepsilon, r\in (0,k)$,
\begin{equation}\label{eq:cex-1d-meas-fw-41}
\begin{split}
-\delta + \left( \min_{\edg\in \RST(\gh;v)}\sum_{e\in \edg}\tqpd(e)\right) &- \left( \min_{\edg\in \RST(\gh)}\sum_{e\in \edg}\tqpd(e)\right) < -\varepsilon \ln \left(\int_{B_r(v)}\dme(x) dx\right)\\
  &< \delta + \left( \min_{\edg\in \RST(\gh;v)}\sum_{e\in \edg}\tqpd(e)\right) - \left( \min_{\edg\in \RST(\gh)}\sum_{e\in \edg}\tqpd(e)\right) .
 \end{split}
\end{equation}
It follows from \eqref{eq:cex-1d-U-assumps} and Lem.~\ref{lem:qp-leq-g-W} that
\begin{equation}\label{eq:cex-1d-min-trees-1}
\begin{split}
&\min_{\edg\in \RST(\gh)}\sum_{e\in \edg}\tqpd(e) = \min_{\edg\in \RST(\gh;v_1)}\sum_{e\in \edg}\tqpd(e) = \tqpd(e_1) = \tilde{U}(1)-\tilde{U}(v_2)\\
&= \min_{\edg\in \RST(\compd)} \sum_{e\in \edg}g(e) = \min_{\edg\in \RST(\compd;v_1)}\sum_{e\in \edg}g(e) = g(e_1)
\end{split}
\end{equation}
and
\begin{equation}\label{eq:cex-1d-min-trees-2}
\begin{split}
&\min_{\edg\in \RST(\gh;v_2)}\sum_{e\in \edg}\tqpd(e) = \tqpd(e_2) = \tilde{U}(m)-\tilde{U}(v_1)\\
&= \min_{\edg\in \RST(\compd;v_2)}\sum_{e\in \edg}g(e) = g(e_2).
\end{split}
\end{equation}
From \eqref{eq:cex-1d-U-assumps}, \eqref{eq:cex-1d-min-trees-1}, \eqref{eq:cex-1d-min-trees-2}, the definition \eqref{eq:morse-rst-minimizer} of $T_*$, and Def.~\ref{def:morse-heights} of $h(\slot)$ it follows that

\begin{equation}\label{eq:ex-cex-1d-measure-tstar-props}
\textnormal{$T_* = \{e_1\}$, $v_* = v_1$, $h(v_*) = h(v_1) = 0$, and $h(v_2) = \tilde{U}(0)-\tilde{U}(v_1) - (\tilde{U}(1)-\tilde{U}(v_2))$.}
\end{equation}
From \eqref{eq:cex-1d-meas-fw-41}, \eqref{eq:cex-1d-min-trees-1}, and \eqref{eq:cex-1d-min-trees-2} we see that, for any $\varepsilon, r\in (0,k)$, 
\begin{equation*}
\begin{split}
-\delta + \tilde{U}(m) - \tilde{U}(v_1) -( \tilde{U}(1)-\tilde{U}(v_2)) &< -\varepsilon \ln \left(\int_{B_r(v_2)}\dme(x) dx\right)\\
&< \delta +\tilde{U}(m) - \tilde{U}(v_1) -( \tilde{U}(1)-\tilde{U}(v_2)).
 \end{split}
\end{equation*}
But from \eqref{eq:cex-1d-U-assumps} and \eqref{eq:ex-cex-1d-measure-tstar-props} we have that
$$\tilde{U}(m) - \tilde{U}(v_1) -( \tilde{U}(1)-\tilde{U}(v_2)) < \tilde{U}(0) - \tilde{U}(v_1) -( \tilde{U}(1)-\tilde{U}(v_2)) = h(v_2),$$
so the conclusion of Theorem~\ref{th:qualitative-measure} is violated by this example; it follows that the answer to Question~\ref{quest:weaken-hypotheses-measure} is negative.
The reason for this disagreement is that the rooted spanning trees minimizing \eqref{eq:cex-1d-min-trees-1} and \eqref{eq:cex-1d-min-trees-2} do not agree as undirected graphs (cf. Rem.~\ref{rem:all-min-morse-trees-obtained-by-rev-edges}), a property used in the proof of Theorem~\ref{th:qualitative-measure}.
This occurs because $\cfo$ is not sufficiently $\Cont^1$-close to an exact one-form as Theorem~\ref{th:qualitative-measure} requires; the ``tilt'' of $\tilde{U}$ is too large.
\end{Ex}

\section{Conclusion}\label{sec:conclusion}
Motivated by problems of physics and biology, we have provided a general mathematical definition of ``flux'' and established its basic properties in \S\ref{sec:flux} for a broad class of nondegenerate diffusion processes (\S \ref{sec:general-setting}) on closed manifolds.
In the case that the noise is small, we have also calculated the small-noise  asymptotics of the flux (in the sense of large deviations \cite{varadhan2016large}) when the chain recurrent set of the limiting drift vector field $\dft$ consists of a finite number of hyperbolic zeros \cite{robinson1999dynamical}.

At this level of generality, we showed in Theorem~\ref{th:flux-manifold-mc-CRST-ld} that the small-noise flux asymptotics are governed by the difference of the sums of $\tqpd(e)$ over the edges $e$ in the minimizing cycle-rooted and rooted spanning trees in the directed graph $\gh = (\vtx,\Eh,\src,\tgt)$ of path homotopy classes based at the index-$0$ zeros $\vtx$ of $\dft$, where $\tqpd(\slot)$ is the restricted version (Def.~\ref{def:quasipotential-restricted}) of our refinement (Def.~\ref{def:quasipotential}) of the Freidlin-Wentzell quasipotential \cite[p.~150]{freidlin2012random}.
Loosely speaking, the proof of Theorem~\ref{th:flux-manifold-mc-CRST-ld} consisted of establishing various estimates to show that, in the sense of large deviations, the flux of the diffusion behaves identically to the flux that would be generated by a Markov chain on the directed graph $\gh$ with finite vertex set $\vtx$, infinite edge set $\Eh$ equipped with a suitable cocycle, transitions occurring at uniformly spaced points in time and governed by probabilities $P(e) = \exp(-\frac{1}{\varepsilon}\tqpd(e))$, and such that all but finitely many transition probabilities are negligible for large deviations (see Lem.~\ref{lem:tau-1-estimates} and \ref{lem:FW-transition-refinement}). 
Using these estimates, we then proved Theorem~\ref{th:flux-manifold-mc-CRST-ld} using the Markov chain tree formula \cite{pitman2018tree}.
In Ex.~\ref{ex:tilt-pot-circle} we showed in particular that, in the special case that $\dim(M) = 1$, this result is compatible with the closed-form formula for flux appearing in the literature.

A drawback of Theorem~\ref{th:flux-manifold-mc-CRST-ld} is that evaluation of $\tqpd(e)$ is generally a difficult problem in the calculus of variations, so Theorem~\ref{th:flux-manifold-mc-CRST-ld} seems difficult to apply for general limiting drift vector fields $\dft$ satisfying its hypotheses.
However, if $\dft = \cfo^\sharp$ is the dual of a closed one-form with respect to the metric induced by the diffusion (\S \ref{sec:general-setting}), and if $\cfo$ is sufficiently $\Cont^1$-close to a generic exact one-form, we showed in Theorems~\ref{th:qualititative-special-case-intro} and \ref{th:qualitative} that a drastic simplification occurs in the expression of the small-noise flux asymptotics.
Under these assumptions, the small-noise flux asymptotics are governed by the optimal ``height'' of certain loops $\gamma$ satisfying $\int_\gamma \cfo > 0$ (Theorem~\ref{th:qualititative-special-case-intro}) or by the optimal ``height'' of certain oriented edges in the Morse graph of $1$-dimensional unstable manifolds of $\dft$ (Theorem~\ref{th:qualitative}); these two viewpoints are equivalent (Lem.~\ref{lem:h-star-defs-coincide}) and also admit a description in terms of persistent homology (Rem.~\ref{rem:persistence}). 
We illustrated Theorems~\ref{th:qualititative-special-case-intro} and \ref{th:qualitative} along with a related result (Prop.~\ref{prop:one-minimizer}) in an example (\S \ref{sec:nr-torus} and Ex.~\ref{ex:nr-example-2}) on the $2$-torus.
In this example we rigorously proved ``by hand'' that the flux displays negative resistance (or conductance, or mobility)---a harder push in a certain direction results in a smaller flux in the same direction---analogous to the negative resistance phenomenon numerically demonstrated in an example in \cite{cecchi1996negative}.

We speculate that such mathematically rigorous results having hypotheses verifiable by ``pen and paper'' methods could prove useful for scientists and engineers by enabling efficient analysis and synthesis of ``force-flux'' (analogous to voltage-current) characteristics of ``Brownian conductors''.
Unfortunately, the small-noise flux asymptotics we have obtained are fairly coarse since they are in the sense of large deviations, and this limits their usefulness beyond identifying candidate force-flux characteristics and/or Brownian conductor designs to experimentally test.
(However, the capability to theoretically identify such candidates for experimental testing may still enable significant reductions in experimental costs by mitigating the need for trial and error.)
For example, in the $2$-torus example just mentioned (\S \ref{sec:nr-torus} and Ex.~\ref{ex:nr-example-2}) we have rigorously proved that negative resistance exists at some parameter values, but we have not obtained sharp information on the precise parameter values at which negative resistance occurs.
For this reason it would be useful and interesting to derive sharper small-noise flux asymptotics, perhaps by employing PDE estimates rather than our probabilistic methods.
Sharper versions of some asymptotic results in \cite{freidlin2012random} have been obtained in \cite{bovier2015meta} using potential theory and in \cite{lepeutrec2013precise} using the Witten Laplacian \cite{witten1982supersymmetry}, and perhaps one of these approaches may enable the derivation of sharper flux asymptotics. 
(The steady-state current $\cme$ of \eqref{eq:fokker-planck} can be expressed in terms of a generalization \cite{pazhitnov1987analytic,pajitnov2006circle} of Witten's deformation of the exterior derivative \cite{witten1982supersymmetry} applied to $\dme$ in the case that the drift vector field is the dual $\cfo^\sharp$ of a closed one-form, and this deformed exterior derivative is used in defining the Witten Laplacian.) 
Other approaches reviewed in \cite{berglund2013kramers} may also prove fruitful.

Next, we mention a possible alternative approach to formulating and proving Theorem~\ref{th:flux-manifold-mc-CRST-ld}.
While we have relied on (cycle-)rooted spanning trees for the formulation and proof of this theorem, it seems natural to instead view Theorem~\ref{th:flux-manifold-mc-CRST-ld} through the lens of ``Freidlin's cycles'' \cite[Sec.~6.6]{freidlin2012random}, \cite{cameron2013computing,gan2017graph}.
As defined in the literature, the hierarchy of Freidlin's cycles do not contain the path-homotopical (or homological) information needed to study flux, but it seems plausible that a path-homotopical (or homological) refinement of Freidlin's cycles could be defined using the path-homotopical refinement of the Freidlin-Wentzell quasipotential that we introduced in \S \ref{sec:setup-sec:drift-finite-hyperbolic-chain}.
In addition to the steady-state flux considered in the present paper, it seems that such a refinement of Freidlin's cycles might provide more detailed information concerning ``transient'' flux on different time scales (cf. \cite[p.~182, Thm~6.3]{freidlin2012random}).

In closing, we mention another suggestion for future work.
We have studied the small-noise asymptotics for flux under the assumption (in particular) that $R(\dft)$ is finite (Theorem~\ref{th:flux-manifold-mc-CRST-ld}).
It would be interesting to study these asymptotics under more general assumptions.
In \cite[p.~146]{freidlin2012random} an equivalence relation $\sim$ on $M$ is defined using the quasipotential so that $x\sim y$ if and only if $\qpd(x,y) = 0 = \qpd(y,x)$.
Under the assumption that there are finitely many (necessarily compact \cite[p.~146]{freidlin2012random}) equivalence classes $K_1,\ldots, K_\ell$ for $\sim$, small-noise asymptotics of various quantities (such as the invariant measure and exit times) are studied in \cite[Ch.~6]{freidlin2012random} for the diffusion $(X^\varepsilon_t, \Prob^\varepsilon_x)$ on $M$  with generator $\dfte + \varepsilon \Delta$ satisfying $\dfte\to \dft$ uniformly.
It seems interesting to study the small-noise asymptotics of flux under such general assumptions.
It also seems interesting to study the relationship between $\sim$ and other equivalence relations in the dynamical systems literature, such as chain equivalence \cite[Def.~2.7.3]{alongi2007recurrence}; note that Prop.~\ref{prop:qp-int-conditions-iff-connecting-piecewise-orbit} implies that $\sim$ coincides with chain equivalence in the special case that the chain recurrent set $R(\dft)$ is finite.

\section*{Acknowledgments}
This work is supported in part by the Army Research Office (ARO) under the SLICE Multidisciplinary University Research Initiatives (MURI) Program, award W911NF1810327.
The authors gratefully acknowledge helpful conversations with Maria K. Cameron, J. Diego Caporale, Wei-Hsi Chen, Matthias Heymann, Daniel E. Koditschek, and Shai Revzen.

	\bibliographystyle{amsalpha}
	\bibliography{ref}
	
\appendix	

\section{Proofs of Lem.~\ref{lem:orbits-bounded-length} and  Prop.~\ref{prop:qp-int-conditions-iff-connecting-piecewise-orbit}, \ref{prop:qp-continuity}}\label{app:proofs}
In this appendix we prove Lem.~\ref{lem:orbits-bounded-length} and Prop.~\ref{prop:qp-int-conditions-iff-connecting-piecewise-orbit}, \ref{prop:qp-continuity}; we also restate these results for convenience.
We first prove Prop.~\ref{prop:qp-int-conditions-iff-connecting-piecewise-orbit} using the following Lem.~\ref{lem:qp-zero-iff-connecting-piecewise-orbit}, which we prove using a technique from \cite[p.~146,~Lem.~1.5]{freidlin2012random}.

\begin{Lem}\label{lem:qp-zero-iff-connecting-piecewise-orbit}
Let $\dft$ be a $\Cont^1$ vector field on a closed Riemannian manifold $M$. 
Assume that $R(\dft)$ is finite.
Then for any $e\in \Pi(M)$, $\qpd(e) = 0$ if and only if $e$ contains a piecewise $\dft$-integral curve.
\end{Lem}
\begin{proof}
Assume that $e$ contains a piecewise integral curve  in the sense of Def.~\ref{def:int-curve-homotopy-class}.
I.e., there is a finite sequence $(\gamma_j)_{j=1}^N$ of segments of $\dft$-integral curves having well-defined path homotopy classes  $[\gamma_j]\in \Pi(M)$ satisfying $[\gamma_1][\gamma_2] \cdots [\gamma_N] = e$.	
By concatenating suitable restrictions $\tilde{\gamma}_j=\gamma_j|_{[a_j,b_j]}$ with short paths $c_j\colon [0,t_j]\to M$ having arbitrarily small actions $\af_{t_j}(c_j)$ \cite[p.~143,~Lem.~1.1]{freidlin2012random} and using local simply connectedness of $M$, a path $\vp\in \Cont_e([0,T],M)$ with $\af_T(\vp)<\varepsilon$ can be constructed for every $\varepsilon > 0$ (with $T\geq 0$ depending on $\varepsilon$), so $\qpd(e) = 0$.
	
To prove the converse, let $e\in \Pi(M)$ satisfy $\qpd(e) = 0$.
Assume, to obtain a contradiction, that $e$ does not contain a piecewise $\dft$-integral curve.
Then $e$ does not contain a constant curve since otherwise the trivial piecewise integral curve  $c \colon \{0\}\to \{\src(e)\}$ satisfies $[c]=e$.
It follows that there exists a piecewise $\dft$-integral curve $c = (\gamma_1,\ldots, \gamma_N)$ such that
\begin{equation}\label{eq:lem-qp-zero-no-ext}
\textnormal{$\qpd([c]^{-1}e) = 0$, $[c]\neq e$, and there is no forward extension of $c$ satisfying this property}.
\end{equation}
Here $[c]^{-1}$ is the reversal (groupoid inverse) of the path homotopy class $[c]$.
By a forward extension of $c$ we mean (i) a forward extension of $\dom(\gamma_N)$ if $\dom(\gamma_N)$ is bounded above, or (ii) the addition to $c$ of a new integral curve segment $\gamma_{N+1}$ satisfying $\omega^*(\gamma_{N+1})= \omega(\gamma_N)$ if $\dom(\gamma_N)$ is unbounded above.
 
Let $e_c \coloneqq [c]^{-1}e$, $x\coloneqq \src(e_c) = \tgt([c])$, and let $\vp^{(k)}\in \Cont_{e_c}([0,T_k],M)$ be a sequence with $\af_{T_k}(\vp^{(k)})\to 0$.
Since $\dft^{-1}(0)$ is finite, for sufficiently small $\varepsilon > 0$ the metric ball $B_{\varepsilon}(x)$ of radius $\varepsilon$ centered at $x = \tgt([c])$ is simply connected, does not contain the entire image of any path representing $e_c$, and is disjoint from $\dft^{-1}(0)\setminus \{x\}$.
Let $S_\varepsilon(x)\coloneqq \partial B_{\varepsilon}(x)$. 
Since by continuity each $\vp^{(k)}$ must pass through $S_\varepsilon(x)$, for each $k$
\begin{equation}
\tau_k\coloneqq \inf \{t> 0 \colon \vp^{(k)}([0,t])\not \subset B_{\varepsilon}(x)\} < \infty
\end{equation}
is well-defined and $q_k\coloneqq \vp^{(k)}(\tau_k)\in S_\varepsilon(x)$ by continuity. 

First assume that $(\tau_k)$ is bounded.
Then by passing to a subsequence we may assume that $\tau_k \to T \geq 0$.
For each $k$ define $\psi^{(k)}\coloneqq \vp^{(k)}|_{[0,T]}$ if $\tau_k\geq T$ and otherwise define $\psi^{(k)}$ to be the extension of $\vp^{(k)}|_{[0,\tau_k]}$ by the constant path $[\tau_k, T]\to \{\vp^{(k)}(\tau_k)\}$.
Then $\af_T(\psi^{(k)})\to 0$ as $k\to\infty$.
Since $\af_T$-sublevel sets are compact in the compact-open topology on $\Cont([0,T],M)$ by \cite[p.~74; p.~135,~Thm~3.2]{freidlin2012random} it follows that a subsequence of $(\psi^{(k)})$ converges uniformly to an absolutely continuous path $\gamma\in \Cont([0,T],M)$ satisfying $\af_T(\gamma)= 0$.
From this and \eqref{eq:action-functional} it  follows that $\gamma$ is a $\dft$-integral curve segment such that $[c][\gamma]\not = e$, and $\qpd(([c][\gamma])^{-1}e) = \qpd([\gamma]^{-1}e_c)=0$ by continuity of $\qpd$ \cite[p.~143, Lem.~1.1]{freidlin2012random}.
This contradicts \eqref{eq:lem-qp-zero-no-ext}.

It remains only to consider the case that $(\tau_k)$ is unbounded.
In this case, by passing to a subsequence we may assume that $\tau_k \geq k$ for all $k\in \N$.
Hence for each $\ell\in \N$, $\psi^{(k,\ell)}\colon [-\ell,0]\to M$ given by $\psi^{(k,\ell)}(t)\coloneqq \vp^{(k)}(t+\tau_k)$ is well-defined for all $k\geq \ell$.
By the compactness of sublevel sets of $\af_{-\ell,0}$  and a diagonal argument, we can construct a \emph{single subsequence} of the $\vp^{(k)}$ such that, after passing to this subsequence, for each $\ell$  the paths $\psi^{(k,\ell)}_\ell$ converges uniformly as $k\to \infty$ to an absolutely continuous $\gamma^{(\ell)}\colon [-\ell,0]\to \infty$  satisfying  $\gamma^{(\ell)}(0)\in S_\varepsilon(x)$ and $\gamma^{(\ell+1)}|_{[-\ell,0]}=\gamma^{(\ell)}$ for all $\ell$.
Hence there is a nonconstant $\dft$-integral curve segment $\gamma\colon (-\infty,0]\to M$ satisfying $\gamma|_{[-\ell,0]}=\gamma^{(\ell)}$ for all $\ell$, $\gamma(0)\in S_\varepsilon(x)$, and $\gamma((-\infty,0])\subset B_\varepsilon(x)$ since the image of each $\psi^{k,\ell}$ is contained in $B_{\varepsilon}(x)$ by the definition of $\tau_k$.
Since $R(\dft)$ is finite, it follows that $\omega^*(\gamma) \subset B_\varepsilon(x) \cap \dft^{-1}(0)$.
Since $B_\varepsilon(x)$ is disjoint from $\dft^{-1}(0)\setminus \{x\}$ by construction, it follows that $x\in \dft^{-1}(0)$ and $\omega(\gamma_N) = \{x\}=\omega^*(\gamma)$, so $(\gamma_1,\ldots,\gamma_N,\gamma)$ is a piecewise $\dft$-integral curve.
Moreover, $[c][\gamma]\neq e$, and $\qpd(([c][\gamma])^{-1}e)=\qpd([\gamma]^{-1}e_c)=0$ by continuity.
This contradicts \eqref{eq:lem-qp-zero-no-ext} and completes the proof.

\end{proof}
Lem.~\ref{lem:qp-zero-iff-connecting-piecewise-orbit} now enables an easy proof of Prop.~\ref{prop:qp-int-conditions-iff-connecting-piecewise-orbit}.
For convenience we restate the proposition.

\PropQpIntIffPiecewise*
\begin{Rem}\label{rem:continuous-uniquely-integrable}
If a continuous vector field $\dft$ on a closed manifold $M$ has unique maximal integral curves, then there is a unique continuous flow $\Phi$ satisfying $\frac{d}{dt}\Phi^t(x)|_{t=0}=\dft(x)$ \cite[App.~A.1]{kvalheim2021necessary}.
In this case Def.~\ref{def:omega-limit}, \ref{def:piecewise-integral-curves}, and \ref{def:int-curve-homotopy-class} and the definition $R(\dft)\coloneqq R(\Phi)$ still make sense, and Lem.~\ref{lem:qp-zero-iff-connecting-piecewise-orbit} and Prop.~\ref{prop:qp-int-conditions-iff-connecting-piecewise-orbit} still hold for such a $\dft$ with all other hypotheses unchanged.
The proofs are identical.
\end{Rem}

\begin{proof}
Fix $e\in \Pi(M)$.
The first displayed statement is the content of Lem.~\ref{lem:qp-zero-iff-connecting-piecewise-orbit}, and the third displayed statement is immediate from the second. 
To prove the second displayed statement we observe that  Lem.~\ref{lem:S-alt-expression} implies that, for any $\gamma\in \Cont_e([T_1,T_2],M)$ satisfying $\af(\gamma)<+\infty$,
\begin{equation*}
\af(\gamma)
\geq \qp_{(-\dft)}(e) + \inf_{[\vp]=e}\left(-\int_{T_1}^{T_2}\ip{\dot{\vp}}{\dft(\vp)}dt\right).
\end{equation*}
Taking the infimum over all such $\gamma$ yields 
\begin{equation*}
\qpd(e)
\geq \qp_{(-\dft)}(e) + \inf_{[\vp]=e}\left(-\int_{T_1}^{T_2}\ip{\dot{\vp}}{\dft(\vp)}dt\right),
\end{equation*}
so the equality $$\qpd(e)=\inf_{[\vp]=e}\left(-\int_{T_1}^{T_2}\ip{\dot{\vp}}{\dft(\vp)}dt\right)$$ holds if and only if $\qp_{(-\dft)}(e) = 0$.
Since $R(\dft)$ is finite if and only if $R(-\dft)$ is finite, the second statement of the proposition now follows from Lem.~\ref{lem:qp-zero-iff-connecting-piecewise-orbit} applied to the reversed vector field $(-\dft)$.
\end{proof}

We now prove Lem.~\ref{lem:orbits-bounded-length}.
For convenience we restate the lemma.

\OrbitsBoundedLength*
\begin{proof}
    We begin by bounding the length of trajectory segments near a zero $z\in \dft^{-1}(0)$.
	Fix $z\in \dft^{-1}(0)$ and any $\varepsilon > 0$.
	By the hyperbolicity of $z$, there exists $K > 0$ and a smooth local chart $\psi\colon U \to \R^n = \R^{n_x+n_y}$ in which any integral curve of $\dft$ satisfies the ODE
	\begin{equation}
	\begin{split}
	\dot{x} &= Ax + R(x,y)x\\
	\dot{y} &= By + Q(x,y)y
	\end{split}
	\end{equation}
	on $\R^{n_x+n_y}$ with all eigenvalues of $A$ having negative real part, all eigenvalues of $B$ having positive real part, and with $\norm{R}_{a_x}, \norm{Q}_{a_y} < \varepsilon$ on the product $B_{K}^{n_x}\times B_K^{n_y}$ of balls of radius $K$ centered at the origins of $\R^{n_x}$ and $\R^{n_y}$. Here $\norm{\slot}_{a_x}, \norm{\slot}_{a_y}$ are the norms induced by an adapted inner products $\ip{\slot}{\slot}_{a_x}$, $\ip{\slot}{\slot}_{a_y}$ on $\R^{n_x},\R^{n_y}$ chosen so that $\ip{x}{Ax}_{a_x} < -k_0 \norm{x}^2_{a_x}$ and $\ip{y}{By}_{a_y}>k_0\norm{y}_{a_y}^2$ for some $k_0 > 0$ \cite[pp.~279--280]{hirschsmale1974}.
	We may and do assume that the balls $B_{K}^{n_x}$, $B_{K}^{n_y}$ are defined with respect to these adapted norm.
	We compute
	\begin{equation*}
	\begin{split}
	\frac{d}{dt}\norm{x}^2_{a_x} &= \ip{x}{Ax+R(x,y)x}_{a_x} < -(k_0-\varepsilon) \norm{x}^2_{a_x}\\
	\frac{d}{dt}\norm{y}^2_{a_y} &= \ip{y}{By + Q(x,y)y}_{a_y}>(k_0-\varepsilon)\norm{y}^2_{a_y},
	\end{split}
	\end{equation*}
    and by taking $\varepsilon$ smaller if necessary we may assume that $(k_0-\varepsilon) > 0$.
    Define $k\coloneqq (k_0-\varepsilon)/2>0$.
	Gr\"{o}nwall's inequality \cite[App.~E]{kvalheim2018aspects} implies that any integral curve segment $(x(t),y(t))$ with image contained in $B_K^{n_x}\times B_K^{n_y}$ satisfies
	\begin{equation}\label{eq:lem-orbits-bounded-length-decay-rates-local}
	\begin{split}
	\forall t\geq 0&\colon \norm{x(t)}_{a_x}\leq e^{-kt}\norm{x(0)}_{a_x} \quad \textnormal{and} \quad \norm{y(t)}_{a_y}\geq e^{kt}\norm{y(0)}_{a_y}\\
	\forall t\leq 0&\colon \norm{x(t)}_{a_x}\geq e^{-kt}\norm{x(0)}_{a_x} \quad \textnormal{and} \quad \norm{y(t)}_{a_y}\leq e^{kt}\norm{y(0)}_{a_y}.
	\end{split}
	\end{equation}
	The second and third inequalities imply that the the smallest nonnegative time $T^+\in [0,+\infty)$ an integral curve with initial condition in the region $B_K^{n_x}\times B_K^{n_y}$ exits this region through $B_{K}^{n_x}\times \partial B_K^{n_y}$ and the largest nonpositive exit time $T^{-}\in (-\infty,0]$ through $\partial B_{K}^{n_x}\times B_{K}^{n_y}$ are well-defined.
	Define $\norm{(x,y)}_a\coloneqq \norm{x}_{a_x} + \norm{y}_{a_y}$.
    If $(x,y)$ is an integral curve with initial condition $(x(0),y(0))\in B_{K}^{n_x}\times B_{K}^{n_y}$, then the first and fourth inequalities imply that the connected component of $(x(0),y(0))$ contained in $B_K^{n_x}\times B_K^{n_y}$ has length
    \begin{equation*}
    \begin{split}
    \int_{T^-}^{T^+}\norm{(x(t),y(t))}dt&\leq  K_0  \int_{T^-}^{T^+}\norm{(x(t),y(t))}_a\, dt\\
    &\leq  K_0 \int_{0}^{+\infty}\norm{x(t+T^-)}_{a_x}\, dt + K_0 \int_{-\infty}^{0}\norm{y(t+T^+)}_{a_y}dt
    \\
    &\leq 2KK_0\int_{0}^{+\infty}e^{-kt}dt = \frac{2K K_0}{k}\eqqcolon C_z,
    \end{split}
    \end{equation*}
	where $\norm{\slot}$ is induced by the Riemannian metric from the statement of the lemma and the constant $K_0$ satisfies $\frac{1}{K_0}\norm{\slot}_a \leq \norm{\slot}\leq K_0 \norm{\slot}_a$ on $B_K^{n_x}\times B_{K}^{n_y}$.\footnote{$K_0$ exists since any pair of continuous Riemannian metrics are uniformly equivalent on compact sets.}
	Defining $\tilde{U}_z$ to be the interior of the set $\psi^{-1}(B_{K}^{n_x}\times B_K^{n_y})$, it follows that every connected component of $\gamma(\R)\cap \tilde{U}_z$ for every maximal $\dft$-integral curve $\gamma$ has length smaller than $C_z$. 
	
    We now construct an open neighborhood $U_z\subset \tilde{U}_z$ of $z$ such that $\gamma(\R)\cap U_z$ has at most one connected component for any maximal $\dft$-integral $\gamma$. 
	Define $D^u\coloneqq \psi^{-1}\left(\partial B_{K}^{n_x}\times B_K^{n_y}\right)$ and $D^s\coloneqq \psi^{-1}\left(B_{K}^{n_x}\times \partial B_K^{n_y}\right)$ and let $\Phi\colon \R\times M \to M$ be the flow of $\dft$. Let $h\in \Cont^\infty(M)$ be a smooth \cite{fathi2019smoothing} complete Lyapunov function \cite{conley1978isolated} for $\dft$.\footnote{In our case this means that $\dft h< 0$ on $M\setminus \dft^{-1}(0)$, $\dft h = 0$ on $\dft^{-1}(0)$, and $h$ assumes distinct values on distinct zeros of $\dft$.
	Here $\dft h$ denotes the Lie derivative of $h$ along the flow of $\dft$.} 
	Let $N\subset \tilde{U}_z$ be a small neighborhood of $z$ and define $U_z$ to be the connected component containing $z$ of $\tilde{U}_z\cap \Phi^{\R}(N)$.
	By the first and fourth inequalities in \eqref{eq:lem-orbits-bounded-length-decay-rates-local},
	there exists $\delta > 0$ and $N$ sufficiently small such that $h|_{\partial U_z\cap D^u}$ is bounded below by $h(z)+\delta$ and $h|_{\partial U_z\cap D^s}$ is bounded above by $h(z)-\delta$.
    Fix such an $N$.
    From the same inequalities in \eqref{eq:lem-orbits-bounded-length-decay-rates-local} and the definition of $U_z$ it follows that every integral curve of $\dft$ entering $U_z$ does so at some point $p$ satisfying $h(p) > h(z)+\delta$ and leaves at some point $q$ satisfying $h(q) < h(z)-\delta$.
    Since $h$ is nonincreasing along trajectories it follows that $\gamma(\R)\cap U_z$ has at most one connected component for any maximal integral curve $\gamma$ of $\dft$, as desired.

	Note that $\dft^{-1}(0) = R(\dft)$ is a finite set by assumption.
	We construct open sets $U_z$ and associated constants $C_z$ as above for each $z\in \dft^{-1}(0)$, and we define $U\coloneqq \bigcup_{z\in \dft^{-1}(0)}U_z$ and $\bar{C}_1\coloneqq \sum_{z\in \dft^{-1}(0)}C_z$.
	Since the intersection of any maximal integral curve with each $U_z$ has at most one connected component it follows that, for any maximal integral curve $\gamma$ of $\dft$,
	\begin{equation}\label{eq:U-length-bound}
	\length(\gamma(\R)\cap U) \leq \bar{C}_1 < \infty.
	\end{equation}
	
	It remains to bound $\length(\gamma(\R)\setminus U)$.
    Since every $x\in M$ converges to $\dft^{-1}(0) = R(\dft)$ in both forward and backward time, for every $x\in M$ there exists $t_x$ such that both $\Phi^{-t_x}(x), \Phi^{t_x}(x)\in U$.
    By continuity, $x$ has a neighborhood $W_x$ such that both $\Phi^{-t_x}(W_x), \Phi^{t_x}(W_x)\subset U$.
    Since $M\setminus U$ is compact, we extract a finite subcover $W_{x_1},\ldots,W_{x_k}$ of $M\setminus U$ and define $T\coloneqq \max \{t_{x_1},\ldots t_{x_k}\}$.
    It follows that, for every $x\in M\setminus U$, both $\Phi^{[0,T]}(x)\cap U\neq \varnothing$ and $\Phi^{[-T,0]}(x)\cap U\neq \varnothing$.
    Since any maximal integral curve $\gamma$ intersects each of the $\#(\dft^{-1}(0))$ connected components of $U$ at most once, $\gamma(\R)\setminus U$ contains at most $\#(\dft^{-1}(0))$ segments, and these segments are defined on intervals of length at most $T$.
    Since 
    \begin{equation*}
    \int_{T_1}^{T_2} \norm{\dot{\gamma}(t)}\,dt = \int_{T_1}^{T_2} \norm{\dft(\gamma(t))}\, dt \leq (T_2-T_1)\max_{x\in M}\norm{\dft(x)}<\infty
    \end{equation*}
    for a $\dft$-integral curve $\gamma$, it follows that 
    \begin{equation}\label{eq:notU-length-bound}
    \length(\gamma(\R)\setminus U) \leq T \#(\dft^{-1}(0)) \max_{x\in M}\norm{\dft(x)} \eqqcolon \bar{C}_2 < \infty
    \end{equation}
    for any maximal integral curve $\gamma$.
    
    Finally, for any $\dft$-integral curve $\gamma$ we obtain from \eqref{eq:U-length-bound} and \eqref{eq:notU-length-bound}
    $$\length(\gamma(\R)) = \length(\gamma(\R) \cap U) + \length(\gamma(\R)\setminus U) \leq \bar{C}_1+ \ \bar{C}_2\eqqcolon C < \infty,$$
    as desired.
\end{proof}

We now prove a preliminary lemma in preparation for the proofs of Lem.~\ref{lem:lower-semi-abs-cont} and Prop.~\ref{prop:qp-continuity}.
The following notation will be used for the statements and proofs of the remaining results in this appendix.
Given a continuous vector field $\bw$ on $M$, we denote by $\af^\bw$ the action functional defined according to \eqref{eq:action-functional} but with the vector field $\bw$ replacing $\dft$, and we denote by $\qp_{\bw}$ the associated quasipotential.
If $\bw$ is a vector field on the universal cover $\tM$ of $M$, we still use the notations $\af^{\bw}$ and $\qp_{\bw}$ for the actional functional and quasipotential associated to a vector field $\bw$ on $\tM$, with context dictating whether the notation refers to $M$ or $\tM$.

\begin{Lem}\label{lem:qp-upper-semicontinuity}
Denote by $\vf^0(M)$ the space of $\Cont^0$ vector fields on the Riemannian manifold $M$ equipped with the $\Cont^0$ topology.
Let $\Pi(M)$ have the topology induced by its bijection with the smooth manifold $(\tM \times \tM)/\Aut(\pi)$, where $\pi\colon \tM\to M$ is the universal cover and the deck transformation group $\Aut(\pi)$ acts diagonally on $\tM\times \tM$ (Rem.~\ref{rem:manifold-structure-fundamental-groupoid}). 
Then the map
$$(\dft,e)\in \vf^0(M)\times \Pi(M) \mapsto \qpd(e)\in [0,+\infty) \quad \textnormal{is upper semicontinuous}.$$
\end{Lem}

\begin{proof}
Our task is equivalent to proving that the map $$(\bw,x,y)\in \vf^0(\tM)\times \tM \times \tM\mapsto \qp_{\bw}(x,y)\in [0,+\infty)$$ is upper semicontinuous, where $\qp_{\bw}$ is defined with respect to the pullback metric on $\tM$.  
Fix $\bw\in \vf^0(\tM)$, let $B\subset \tM$ be a precompact open set, and let $C_0$ be an upper bound for $\norm{\bw}$ on $B$.
Then if $\bu\in \vf^0(\tM)$ satisfies $\norm{\bu-\bw}\leq 1$ on $B$ and $\vp_a$ is a unit speed length minimizing geodesic from $a\in B$ to $x\in B$, a computation using the definition \eqref{eq:action-functional} of $\af^{\bu}$ yields  $\af^{\bu}(\vp_a)\leq L\dist{a}{x}$, where $L = (1/4)(2+C_0)^2$.
Similarly, $\af^{\bu}(\vp_b)\leq L\dist{y}{b}$ if $\vp_b$ is a unit speed length minimizing geodesic from $y\in B$ to $b\in B$.
Thus, $\qp_{\bu}(a,b)\leq \qp_{\bu}(x,y) + L\dist{a}{x} + L\dist{y}{b}$ if $a,x\in B$ and $y,b\in B$ are sufficiently close.
Hence it suffices to prove that the map
\begin{equation}\label{eq:lem-qp-usc-rest-map}
\bw\in \vf^0(\tM)\mapsto \qp_{\bw}(x,y)\in [0,+\infty)
\end{equation}
is upper semicontinuous for each fixed $x,y\in \tM$.

Fix $\bw\in \tM$ and $\varepsilon > 0$.
Let $\vp\colon [0,T]\to \tM$ be a continuous path from $x$ to $y$ satisfying $\af^{\bw}_T(\vp) < \qp_{\bw}(x,y) + \varepsilon.$
It is immediate from the definition \eqref{eq:action-functional} that also $\af^{\bu}_T(\vp) < \qp_{\bw}(x,y) + \varepsilon$ if $\bu$ is sufficiently close to $\bw$ on the compact set $\vp([0,T])$.
Since $\qp_{\bu}(x,y) \leq \af^{\bu}_T(\vp)$, it follows that the map in \eqref{eq:lem-qp-usc-rest-map} is upper semicontinuous.
\end{proof}

\begin{Lem}\label{lem:lower-semi-abs-cont}
Let $(\bu_n)_{n\in \N}$ be a sequence of $\Cont^0$ vector fields on the closed Riemannian manifold $M$ converging uniformly to a continuous vector field $\dft$.
Let $k > 0$ and $(\vp^{(n)}\in \Cont([T_1,T_2],M))_{n\in \N}$ be a sequence of paths with fixed domain $[T_1,T_2]$ satisfying $\af^{\bu_n}_{T_1,T_2}(\vp^{(n)})\leq k$ for all $n$.
Then the family $(\vp^{(n)})$ is uniformly equicontinuous, and there is a subsequence $(\vp^{(n_k)})_{k\in \N}$ which converges uniformly to a path $\vp \in \Cont([T_1,T_2],M)$ satisfying $\af^{\dft}_{T_1,T_2}(\vp)\leq k$. 
\end{Lem}

\begin{proof}
By a translation of $\R$ we may assume that $T_1 = 0$ and $T_2 = T$. 
By the Nash embedding theorem we may assume that $M$ is isometrically embedded in some $\R^N\supset M$, so the Riemannian metric on $M$ is the restriction of the Euclidean inner product \cite[Thm~2]{nash1956embedding}, and we may view $\vp, \vp^{(n)}, \dot{\vp}^{(n)}\in \Cont([0,T],\R^N)$ as $\R^N$-valued.
We may also view $\dft$ and $\bu_n$ as $\R^N$-valued, and we arbitrarily extend the $\bu_n$ and $\dft$ to $\Cont^0$ maps $\R^N\to \R^N$.

We first show uniform equicontinuity of the family $(\vp^{(n)})$ and convergence of a subsequence $\vp^{(n_k)}$ to some $\vp\in \Cont([0,T],M)$.
Given $t,h\geq 0$ satisfying $t+h \leq T$, the triangle and Cauchy-Schwarz inequalities imply that 
\begin{equation}\label{eq:vp-equi-cont-estimate}
\begin{split}
\norm{\vp^{(n)}(t+h)-\vp^{(n)}(t)}&\leq \int_t^{t+h}\norm{\dot{\vp}^{(n)}(s)}ds\leq \int_t^{t+h}\norm{\dot{\vp}^{(n)}-\bu_n(\vp^{(n)})}ds + \int_t^{t+h}\norm{\bu_n(\vp^{(n)})}ds \\
&\leq \sqrt{h\int_t^{t+h}\norm{\dot{\vp}^{(n)}-\bu_n(\vp^{(n)})}^2ds} + h\norm{\bu_n}_{\infty}\\
&\leq \sqrt{4h\af^{\bu_n}_T(\vp^{(n)})} + h\norm{\bu_n}_{\infty} \leq \sqrt{4hk} + 2hK,
\end{split}
\end{equation}
where $\norm{\slot}_\infty$ is the supremum norm on the restrictions of functions to $M\subset \R^N$ and $K>0$ is an upper bound on the convergent sequence $(\norm{\bu_n}_\infty)$.
Thus, the resulting family of paths $(\vp^{(n)})$ is uniformly equicontinuous.
Since $M\subset \R^N$ is compact and hence bounded, the Arzel\`a-Ascoli theorem and closedness of $\Cont([0,T],M)\subset \Cont([0,T],\R^N)$  imply the existence of a subsequence of $(\vp^{(n)})$ converging uniformly to some $\vp\in \Cont([0,T],M)$ as claimed.
For simplicity we relabel the subsequence so as to use the same notation $(\vp^{(n)})$ below, so that $\vp^{(n)}\to \vp$, and we continue to relabel in this same way after passing to further subsequences.

It remains to show that $\af_T^\dft(\vp)\leq k$.
Define the $L^2$ inner product $\ip{f}{g}_2\coloneqq \int_0^T \ip{f}{g} dt$ and norm $\norm{g}_2\coloneqq \sqrt{\ip{g}{g}_2}$ of measurable functions $[0,T]\to \R^m$ for some $m$.
Using the Cauchy-Schwarz inequality we note (cf. Rem.~\ref{rem:finite-action-iff-L2}) that
\begin{equation}\label{eq:af-rewrite-l2-bounded}
\begin{split}
4\af_{T}^{\bu_n}(\vp^{(n)}) &= \norm{\dot \vp^{(n)}}_2^2 - 2\ip{\dot{\vp}^{(n)}}{\bu_n(\vp^{(n)})}_2 +  \norm{\bu_n(\vp^{(n)})}_2^2\\
&\geq \norm{\dot \vp^{(n)}}_2^2 - 2\norm{\dot \vp^{(n)}}_2\norm{\bu_n(\vp^{(n)})}_2 + \norm{\bu_n(\vp^{(n)})}_2^2 \\
&= (\norm{\dot \vp^{(n)}}_2-\norm{\bu_n(\vp^{(n)})}_2)^2.
\end{split}
\end{equation}
Since $\af_{T}^{\bu_n}(\vp^{(n)})$ and $\norm{\bu_n(\vp^{(n)})}_2\leq T \norm{\bu_n}_\infty\to T\norm{\dft}_\infty$ are uniformly bounded, it follows from \eqref{eq:af-rewrite-l2-bounded} that so are the $\norm{\dot \vp^{(n)}}_2$.
Writing $\vp^{(n)} = (\vp^{(n)}_1, \ldots, \vp^{(n)}_N)$, it follows that the $L^2$ norms of the derivatives of each of the components $\dot{\vp}^{(n)}_i\in L^2([0,T],\R)$ are uniformly bounded.
Hence the Banach-Alaoglu theorem \cite[p.~23, Thm~6.2]{showalter1994hilbert} implies that, after passing to a subsequence, each component sequence $\dot \vp^{(n)}_i$ converge weakly in the Hilbert space $L^2([0,T],\R)$ to some $g_i\in L^2([0,T],\R)\subset L^1([0,T],\R)$.
This means that
\begin{equation}\label{eq:weak-l2-conv}
\forall i\in \{1,\ldots, N\}\colon \forall f\in L^2([0,T],\R)\colon \lim_{n\to\infty} \ip{\dot \vp^{(n)}_i}{f}_2 = \ip{g_i}{f}_2. 
\end{equation} 
Taking $f$ to be the indicator functions $\mathbf{1}_{[0,t]}$ in \eqref{eq:weak-l2-conv} yields, for all $1\leq i\leq N$: 
\begin{equation*}
\begin{split}
\forall t\in [0,T]\colon \vp_i(t) &= \lim_{n\to\infty}\vp_i^{(n)}(t)= \lim_{n\to\infty}\vp^{(n)}_i(0) + \int_0^t \dot\vp_i^{(n)}(s)ds = \vp_i(0) +  \lim_{n\to\infty} \ip{\dot\vp_i^{(n)}}{\mathbf{1}_{[0,t]}}_2 \\
&= \vp_i(0) + \ip{g_i}{\mathbf{1}_{[0,t]}}_2
= \vp_i(0) + \int_0^t g_i(s)ds.
\end{split}
\end{equation*}
Since each $g_i\in L^1([0,T],\R)$, it follows that $\vp$ is absolutely continuous with derivative $\dot \vp = (g_1,\ldots, g_n)$ almost everywhere \cite[Thm~3.35]{folland1999real}.

Since Hilbert space norms are weakly lower semicontinuous \cite[p.~355]{reed1980methodsi} and since the $\dot\vp^{(n)}$ converge weakly to $g = \dot \vp$, we also have $\norm{\dot \vp}_2\leq \liminf_{n\to\infty}\norm{\dot{\vp}_n}_2$. 
Thus,
\begin{equation*}
\begin{split}
\af_T^\dft(\vp) &= \norm{\dot \vp}_2^2 - 2\ip{\dot{\vp}}{\dft(\vp)}_2 + \norm{\dft(\vp)}_2
= \norm{\dot \vp}_2^2 - 2\lim_{n\to\infty}\ip{\dot{\vp}^{(n)}}{\dft(\vp)}_2 + \lim_{n\to\infty} \norm{\bu_n(\vp^{(n)})}_2\\
&= \norm{\dot \vp}_2^2 + \lim_{n\to\infty}\left[ - 2\ip{\dot{\vp}^{(n)}}{\bu_n(\vp^{(n)})}_2 +  \norm{\bu_n(\vp^{(n)})}_2 + 2\ip{\dot{\vp}^{(n)}}{\bu_n(\vp^{(n)})-\dft(\vp)}_2\right]\\
&\leq \liminf_{n\to\infty}\norm{\dot \vp^{(n)}}_2^2+ \lim_{n\to\infty}\left[ - 2\ip{\dot{\vp}^{(n)}}{\bu_n(\vp^{(n)})}_2 +  \norm{\bu_n(\vp^{(n)})}_2 + 2\ip{\dot{\vp}^{(n)}}{\dft(\vp)-\bu_n(\vp^{(n)})}_2\right]\\
&\leq \liminf_{n\to\infty} \af_T^{\bu_n}(\vp^{(n)}) + 2\limsup_{n\to\infty} \ip{\dot{\vp}^{(n)}}{\dft(\vp)-\bu_n(\vp^{(n)})}_2,
\end{split}
\end{equation*}
where the second equality follows from \eqref{eq:weak-l2-conv} and the uniform convergence of $\bu_n(\vp^{(n)})$ to $\dft(\vp)$, and the second inequality follows from the first equality in \eqref{eq:af-rewrite-l2-bounded} and the general fact that
$$\liminf_{n\to\infty}a_n+\liminf_{n\to\infty}(b_n+c_n)\leq \liminf_{n\to\infty} (a_n+b_n+c_n) \leq \liminf_{n\to\infty}(a_n+b_n)+\limsup_{n\to\infty}c_n.$$
Using $\norm{\dft(\vp)-\bu_n(\vp^{(n)})} \leq \norm{\dft(\vp)-\dft(\vp^{(n)})} + \norm{\dft(\vp^{(n)}) -\bu_n(\vp^{(n)})}$, it follows that
\begin{equation*}
\begin{split}
\af_T^\dft(\vp) &\leq \liminf_{n\to\infty} \af_T^{\bu_n}(\vp^{(n)}) + 2\sqrt{T} \limsup_{n\to\infty} \norm{\dot{\vp}^{(n)}}_2\sup_{t\in[0,T]}\norm{\dft(\vp(t))-\bu_n(\vp^{(n)}(t))}\\
&\leq \liminf_{n\to\infty} \af_T^{\bu_n}(\vp^{(n)}) + 2\sqrt{T} \limsup_{n\to\infty} \norm{\dot{\vp}^{(n)}}_2\left(\sup_{t\in[0,T]}\norm{\dft(\vp(t))-\dft(\vp^{(n)}(t))} + \norm{\dft-\bu_n}_\infty\right).
\end{split}
\end{equation*}
The $L^2$ norms $\norm{\dot \vp^{(n)}}_2$ are bounded (as noted following \eqref{eq:af-rewrite-l2-bounded}) and $\dft$ is uniformly continuous on the compact $M$, so the uniform convergence of $\vp^{(n)}$ to $\vp$ and of $\bu_n$ to $\dft$ implies that the $\limsup$ is zero.
Since $\af_T^{\bu_n}(\vp^{(n)}) \leq k$ for all $n$ by assumption, this proves the desired remaining claim 
$$\af_T^\dft(\vp)\leq \liminf_{n\to\infty} \af_T^{\bu_n}(\vp^{(n)}) \leq k.$$
\end{proof}
We now prove Prop.~\ref{prop:qp-continuity}.
For convenience we restate the proposition.
\PropQpContMS*
\begin{proof}
Our task is equivalent to proving that the map $$(\bw,x,y)\in \vf^1(\tM)\times \tM \times \tM\mapsto \qp_{\bw}(x,y)\in [0,+\infty)$$ is continuous at $(\tdft_0,x,y)$ for each $x,y\in \tM$, where $\tdft_0\in \vf^1(\tM)$ is the unique lift of $\dft_0$ to $\tM$ ($\pi_*\tdft_0 = \dft_0$) and $\qpd$ is defined with respect to the pullback metric on $\tM$. 

Fix $x,y\in \tM$.
As shown in the proof of Lem.~\ref{lem:qp-upper-semicontinuity}, if $\norm{\bu-\tdft_0}\leq 1$ on some compact neighborhood $B$ of $(x,y)$, then there exists $L > 0$ such that $\qp_{\bu}(a,b)\leq \qp_{\bu}(x,y) + L\dist{a}{x} + L\dist{y}{b}$ for all $a,b\in B$ with $(a,b)$ sufficiently close to $(x,y)$.
Reversing the roles of $(a,b)$ and $(x,y)$ yields $\qp_{\bu}(x,y)\leq \qp_{\bu}(a,b) + L\dist{a}{x} + L\dist{y}{b}$, so $$|\qp_{\bu}(x,y)-\qp_{\bu}(a,b)| \leq  L\dist{a}{x} + L\dist{y}{b} \to 0 \quad \textnormal{as} \quad \dist{a}{x} + \dist{y}{b}\to 0.$$
Since the triangle inequality yields
\begin{equation}\label{eq:triangle-yields}
|\qp_{\bw}(x,y)-\qp_{\bu}(a,b)|\leq |\qp_{\bw}(x,y)-\qp_{\bu}(x,y)| + |\qp_{\bu}(x,y) - \qp_{\bu}(a,b)|,
\end{equation}
we see it suffices to prove that the map
\begin{equation*}
\bw\in \vf^1(\tM)\mapsto \qp_{\bw}(x,y)\in [0,+\infty)
\end{equation*}
is continuous at $\tdft_0$ for each fixed $x,y\in \tM$, and this is in turn equivalent to proving that the map
\begin{equation}\label{eq:lem-qp-continuity-rest-map}
\dft\in \vf^1(M)\mapsto \qp_{\dft}(e)\in [0,+\infty)
\end{equation}
is continuous at $\dft_0$ for each fixed $e\in \Pi(M)$.

Since the $\Cont^1$ topology is finer than the $\Cont^0$ topology, upper semicontinuity at $\dft_0$ follows from Lem.~\ref{lem:qp-upper-semicontinuity}.
It remains to establish lower semicontinuity.

Since Morse-Smale vector fields are open in the $\Cont^1$ topology \cite[Thm~3.5]{palis1968ms} and structurally stable \cite[Thm~5.2]{palis1968structural}, there exists a neighborhood $\cN\subset \vf^1(M)$ of $\dft_0$ such that every $\dft\in \cN$ is Morse-Smale without nonstationary periodic orbits.
Suppose (to obtain a contradiction) that the map in \eqref{eq:lem-qp-continuity-rest-map} is not lower semicontinuous at $\dft_0$ for arbitrary $e\in \Pi(M)$.
Then there exists $e\in \Pi(M)$, $k > 0$, and a sequence $(\bu_n)_{n\in \N}\subset \cN$ with $\bu_n\to \dft_0$ in $\vf^1(M)$ such that $\qp_{\bu_n}(e) < \qp_{\dft_0}(e)-2k$ for all $n$.
Hence for each $n$ there exists a path $\vp^{(n)}\in \Cont_e([0,T_n],M)$ with
\begin{equation}\label{eq:qp-k-lb}
\af^{\bu_n}(\vp^{(n)})< \qp_{\dft_0}(e)-2k
\end{equation}
for all $n$.

If $(T_n)$ is bounded, then by passing to a subsequence we may assume that $T_n\to T\geq 0$. 
In this case we define $\psi^{(n)}\coloneqq \vp^{(n)}|_{[0,T]}$ if $T_n\geq T$ and otherwise we define $\psi^{(n)}$ to be the extension of $\vp^{(n)}|_{[0,T_n]}$ by the constant path $[T_n, T]\to \{\tgt(e)\}$.
Since $T_n\to T$ and $\norm{\bu_n-\dft_0}\to 0$, $\af_T^{\bu_n}(\psi^{(n)})< \qp_{\dft_0}(e)-2k$ for all $n$ large enough.
After passing to a subsequence if necessary, it follows from Lem.~\ref{lem:lower-semi-abs-cont} that the $\psi^{(n)}$ converge to a path $\vp\in \Cont_e([0,T],M)$ satisfying $\af_T^{\dft_0}(\vp)\leq \qp_{\dft_0}(e)-2k$, a contradiction.

It remains to consider the case that $(T_n)$ is unbounded.
We first make some preliminary observations.
Define $\edg \subset \Pi(M)$ via $\edg\coloneqq \src^{-1}(\dft^{-1}(0)) \cap \tgt^{-1}(\dft^{-1}(0))$, and define $\edg_0\subset \edg$ via  $$\edg_0\coloneqq \{e\in \edg\colon \src(e)=\tgt(e) \textnormal{ and $e$ is not a constant path homotopy class}\}.$$
The chain recurrent set $R(\dft)$ consists of a finite number of hyperbolic zeros since $\dft$ is Morse-Smale without nonstationary periodic orbits, so Prop.~\ref{prop:qp-int-conditions-iff-connecting-piecewise-orbit} and Lem.~\ref{lem:univ-cover-qp-growth} imply the existence of $C_0 > 0$ such that, for any finite sequence $e_1,\ldots, e_n \in \edg$ such that $e_1\cdots e_n\in \edg_0$,
\begin{equation}\label{eq:prop-qp-continuity-loop-lb}
\min_{i\in \{1,\ldots, n\}}\qpd(e_i) > C_0.
\end{equation}
Let $\kappa_0 > 0$ be sufficiently small that the closed metric balls $B_\kappa(z)$ of radius $\kappa$ centered at each $z\in \dft^{-1}(0)$ are geodesically convex \cite[Thm~6.17]{lee2018riemannian} and pairwise disjoint, and define $B_\kappa\coloneqq \bigcup_{z\in \dft^{-1}(0)}B_\kappa(z)$ for $\kappa\in (0,\kappa_0)$.
Given any path $\vp\in \Cont([T_1,T_2],M)$ with initial and terminal points in $B_\kappa$, we denote by $e(\vp)\in \edg$ the path homotopy class of the path defined by first following the unique minimizing geodesic in $B_\kappa$ from $\dft^{-1}(0)$ to $\vp(0)$, then following $\vp$, then following the unique minimizing geodesic in $B_\kappa$ from $\vp(T)$ to $\dft^{-1}(0)$.
By \cite[p.~143, Lem.~1.1]{freidlin2012random} we may choose $\kappa$ small enough that any pair of points in the same component of $B_\kappa$ may be joined by a path $\vp$ satisfying $\af(\vp)< \varepsilon/2$.
It follows that, for any $\varepsilon > 0$ such that $C_1\coloneqq C_0 - \varepsilon > 0$, all sufficiently small $\kappa > 0$, and any finite sequence of paths $\gamma_1,\ldots, \gamma_n$ with initial and terminal points in $B_{\kappa}$ satisfying $e(\gamma_1)\cdots e(\gamma_n)\in \edg_0$, 
\begin{equation}\label{eq:prop-qp-continuity-approx-loop-lb}
\begin{split}
\af(\gamma_1)+\cdots + \af(\gamma_n) \geq \min_{i\in\{1,\ldots,n\}}\af(\gamma_i) \geq \min_{i\in \{1,\ldots, n\}}\qpd(e(\gamma_i))-\varepsilon >  C_1 > 0.
\end{split}
\end{equation}
Fix $\varepsilon \in (0,k)$ (cf. \eqref{eq:qp-k-lb}) and $\kappa_1\in (0,\kappa_0)$ small enough that \eqref{eq:prop-qp-continuity-approx-loop-lb} holds for some $C_1 > 0$ for all $\kappa \in (0,\kappa_1)$.

Next, the implicit function theorem implies that $(\bu_n)^{-1}(0) \subset \interior(B_{\kappa})$ for all sufficiently large $n$ since $\bu_n\to \dft_0$ in $\vf^1(M)$ and $\dft^{-1}(0)$ consists of finitely many hyperbolic zeros.
Let $\Phi_{\dft}\colon \R\times M\to M$ denote the flow of $\dft\in \vf^1(M)$.
For any $\kappa \in (0,\kappa_1)$, since $R(\bu_n) = \bu_n^{-1}(0)\subset \interior(B_\kappa)$ for each $n$ and since $\bu_n\to \dft_0$, joint continuity of the map $$(t,x,\dft)\in \R\times M\times \vf^1(M)\mapsto \Phi_{\dft}^t(x) \in M$$ \cite[Thm~B.3]{duistermaat2000lie} implies the existence of  $T(\kappa) > 0$ such that, for any $x\not \in \interior(B_\kappa)$ and sufficiently large $n\in \N$, there exists $t_x, t_{x,n}\in [0,T(\kappa)/2]$ such that $\Phi_{\dft_0}^{t_x}(x),\Phi_{\bu_n}^{t_x^{(n)}}(x)\in \interior(B_\kappa)$.
Thus, $\af_{T(\kappa)}^{\dft_0}(\vp), \af_{T(\kappa)}^{\bu_n}(\vp) > 0$ for any $\vp\in ([0,T(\kappa)],M\setminus \interior(B_\kappa))$ for large enough $n$.
Since $\Cont([0,T(\kappa)],M\setminus \interior(B_\kappa))\subset \Cont([0,T(\kappa)],M)$ is closed, Lem.~\ref{lem:lower-semi-abs-cont} implies the existence of $C_2(\kappa)>0$ such that
\begin{equation}\label{eq:af-out-of-b-lb}
\af_{T(\kappa)}^{\dft_0}(\vp), \af_{T(\kappa)}^{\bu_n}(\vp^{(n)}) > C_2(\kappa) > 0
\end{equation}
for all $\vp\in \Cont([0,T(\kappa)],M\setminus \interior(B_\kappa))$, all $\kappa\in (0,\kappa_1)$, and all $n$ large enough.

Next, for each $n$ we define the interval $[a^{(n)}_1, b^{(n)}_1]$ by the properties $\vp^{(n)}(a_1^{(n)}), \vp^{(n)}(b_1^{(n)})\in \partial B_\kappa$, \\ $e(\vp^{(n)}|_{[a_1^{(n)},b_1^{(n)}]})$ is not a constant path homotopy class, and $a^{(n)}_1, b^{(n)}_1$ are the smallest numbers in $[0,T_n]$ with these properties.
For each $n$ we then recursively define the intervals $[a^{(n)}_{i+1}, b^{(n)}_{i+1}]$ so that $b^{(n)}_{i+1} > a^{(n)}_{i+1} > b^{(n)}_i$ are the smallest numbers larger than $ b^{(n)}_i$ with the same properties; it follows in particular that each $e(\vp^{(n)}|_{[b_{i}^{(n)},a_{i+1}^{(n)}]})$ is a constant path homotopy class.
Suppose (to obtain a contradiction) that the numbers $N_n$ of such intervals $[a^{(n)}_{i}, b^{(n)}_{i}]$ are unbounded; after passing to a subsequence we may assume that $N_n\to +\infty$.
Eq.~\eqref{eq:af-out-of-b-lb} implies that the numbers $(b_i^{(n)}-a_i^{(n)})$ are bounded, so after passing to a diagonal subsequence we may further assume that, for each $i$, $(b_i^{(n)}-a_i^{(n)})$ converges to some $c_i \geq 0$.
For each $i$ and large enough $n$ we define $\psi_i^{(n)}\colon [0,c_i]\to M$ via $\psi_i^{(n)}(t)\coloneqq \vp^{(n)}(a_i^{(n)}+t)$ for all $t\in [0,c_i]$ if $c_i \leq (b_i^{(n)}-a_i^{(n)})$, and otherwise we define $\psi_i^{(n)}$ to be given by this formula for $t\in [0, b_i^{(n)}-a_i^{(n)}]$ and constant on $[b_i^{(n)}-a_i^{(n)},c_i]$.
For each $i$ we have that $\af^{\bu_n}(\psi_i^{(n)})$ is bounded, so Lem.~\ref{lem:lower-semi-abs-cont} and a diagonal argument imply that, after passing to a subsequence of $(\vp^{(n)})$, the $\psi_i^{(n)}$ converge to paths $\gamma_i$ such that $\tgt(e(\gamma_i)) = \src(e(\gamma_{i+1}))$ and $\af^{\bu_n}(\psi_i^{(n)})\to \af^{\dft_0}(\gamma_i)$.
Hence
\begin{equation}\label{eq:af-le-qp-2k}
\begin{split}
\sum_{i=1}^\infty \af^{\dft_0}(\gamma_i) &= \sum_{i=1}^\infty \lim_{n\to\infty}\af^{\bu_n}(\psi_i^{(n)}) = \sum_{i=1}^\infty \lim_{n\to\infty}\af^{\bu_n}(\vp|^{(n)}_{[a_i^{(n)},b_i^{(n)}]}) \leq \liminf_{n\to\infty}\sum_{i=1}^\infty \af^{\bu_n}(\vp|^{(n)}_{[a_i^{(n)},b_i^{(n)}]})\\&\leq \liminf_{n\to\infty}\af^{\bu_n}(\vp^{(n)})\leq \qpd(e)-2k,
\end{split}
\end{equation}
where the second equality follows since $|\af^{\bu_n}(\psi_i^{(n)})-\af^{\bu_n}(\vp^{(n)}|_{[a^{(n)}_i,b^{(n)}_i]})| \to 0$, the first inequality follows from Fatou's lemma, the second inequality follows since for each $n$ the intervals $(a_i^{(n)},b_i^{(n)})$ are disjoint, and the final inequality follows from \eqref{eq:qp-k-lb}.
But since $\tgt(e(\gamma_i)) = \src(e(\gamma_{i+1}))$ for all $i$, \eqref{eq:prop-qp-continuity-approx-loop-lb} and the pigeonhole principle imply that the left side of \eqref{eq:af-le-qp-2k} is larger than $C_1 + C_1 + \dots = +\infty$, so we have arrived at a contradiction.
It follows that there is an integer $N\geq 1$ such that
\begin{equation}\label{eq:at-most-N-intervals}
\forall n\in \N\colon \textnormal{ there are at most $N$ such intervals $[a_i^{(n)},b_i^{(n)}]$.}
\end{equation}

Finally, observe that \eqref{eq:at-most-N-intervals} holds with the same constant $N$ for all $\kappa\in (0,\kappa_1)$ even though the constants $T(\kappa)$, $C_2(\kappa)$ in \eqref{eq:af-out-of-b-lb} depend on the specific value of $\kappa$.
Additionally, \eqref{eq:prop-qp-continuity-approx-loop-lb} holds with the same constant $C_1$ for all $\kappa\in(0,\kappa_1)$.
Using \cite[p.~143, Lem.~1.1]{freidlin2012random} again, there is $\kappa_2\in (0,\kappa_1)$ such that, for any $\kappa\in (0,\kappa_2)$, any pair of points in the same component of $B_\kappa$ may be joined by a short path $\vp\colon [0,\tau]\to B_{\kappa}$ satisfying $\af(\vp)< \varepsilon/N$ and $\tau< \varepsilon/N$.
For each $n$ we modify $\vp^{(n)}$ by deleting each of the $N_n \leq N$ restrictions  $\vp^{(n)}|_{[b_i^{(n)},a_{i+1}^{(n)}]}$ and replacing them with such short paths from $\vp^{(n)}(b_i^{(n)})$ to $\vp^{(n)}(a_{i+1}^{(n)})$.
Thus, after reparametrizing appropriately, we obtain well-defined paths $\theta^{(n)}\in \Cont_e([0,\tau_n],M)$ satisfying
\begin{equation}\label{eq:qp-cont-final}
\af(\theta^{(n)})\leq \qpd(e)-2k+\varepsilon < \qpd(e)-k \qquad \textnormal{and} \qquad \tau_n \leq \varepsilon + \sum_{i=1}^{N_n}(b_i^{(n)}-a_i^{(n)}).
\end{equation}
Eq.~\eqref{eq:af-out-of-b-lb} (now with new constants $C_2(\kappa), T(\kappa)$) again implies that $(b_i^{(n)}-a_i^{(n)})$ is bounded for each $i\leq N_n \leq N$, so the second inequality in \eqref{eq:qp-cont-final} implies that $(\tau_n)$ is bounded.
Hence with $k$ replacing $2k$, $\theta^{(n)}$ replacing $\vp^{(n)}$, and $\tau_n$ replacing $T_n$, we have reduced to the case of bounded $(T_n)$ which, as explained previously (following \eqref{eq:qp-k-lb}), leads to a contradiction.
This completes the proof.
\end{proof}

\end{document}